\documentclass[12pt,a4paper]{article}
\usepackage{amsmath,amssymb,amsthm}
\usepackage{amsfonts}
\usepackage[utf8]{inputenc}
\usepackage[T1]{fontenc}
\usepackage{geometry}
\usepackage{cite}
\usepackage{bm}
\usepackage{bbm}
\usepackage{dsfont}
\usepackage[hidelinks]{hyperref}
\usepackage{enumitem}
\newtheorem{theorem}{Theorem}[section]
\newtheorem{lemma}[theorem]{Lemma}
\newtheorem{proposition}[theorem]{Proposition}
\newtheorem{definition}[theorem]{Definition}
\newtheorem{remark}[theorem]{Remark}
\newtheorem{corollary}[theorem]{Corollary}
\numberwithin{equation}{section}

\newcommand{\R}{\mathbb{R}}

\newcommand{\Pp}{\mathbb{P}}
\newcommand{\F}{\mathcal{F}}

\newcommand{\dd}{\,\mathrm{d}}

\DeclareMathOperator{\Div}{div}
\DeclareMathOperator{\supp}{supp}
\DeclareMathOperator{\Tr}{Tr}

\title{Compressible stochastic fluid--structure interaction with transport noise and Navier slip: A stochastic-flow approach}

\author{Yong Chen\thanks{Department of Mathematics, Zhejiang Sci-Tech University, Hangzhou 310018, China. E-mail: youngchen329@126.com}
	\and 
	Hongjun Gao\thanks{School of Mathematics, Southeast University, Nanjing 211189, China. E-mail: hjgao@seu.edu.cn}}

\date{}
\allowdisplaybreaks
\begin{document}
	
	\maketitle

	\begin{abstract}
		This paper establishes the existence of local martingale solutions for a three-dimensional stochastic fluid--structure interaction (FSI) problem coupling a viscous compressible isentropic fluid with an elastic shell through Navier-slip boundary conditions. The fluid and shell are perturbed by compatible Stratonovich transport noise. A stochastic-flow transformation removes the stochastic transport integrals and produces a pathwise random FSI system whose pullback coefficients are only H\"older continuous in time. We combine a multi-layer approximation scheme with localization on bounded-flow events and freezing of the random coefficients on short deterministic time subintervals. The shell transport fields are not required to generate isometries: the bending-energy martingale and It\^o correction are controlled at the natural $H^2$ level by commutator identities without derivative loss. Tightness is obtained in Jakubowski spaces, pressure compactness is proved through localized effective viscous flux, and the final artificial-pressure and penalty limit uses boundary-pressure tightness and $H^2$-graph trace tools. For $\gamma>3/2$ and non-colliding initial data, we obtain a martingale solution up to an almost surely positive stopping time.
		
	\end{abstract}
	\begin{center}
		\begin{minipage}{0.9\textwidth} 
		\end{minipage}
	\end{center}
	
	\vspace{0.5cm}

		\section{Introduction}
	
	\subsection{Background}
	
	The analysis of compressible fluid--structure interaction (FSI) lies at the
	intersection of two compactness theories that are already delicate on
	their own.  For compressible Navier--Stokes equations, weak sequential
	stability is based on the Lions--Feireisl mechanism
	\cite{Lions1993,Lions1998,Feireisl2001}: the density is only weakly
	compact at the energy level, and strong convergence is recovered through
	renormalization and the effective viscous flux.  In fluid--structure
	interaction (FSI), one must combine this mechanism with a time-dependent
	fluid domain whose regularity is dictated by the elastic structure.
	Even in deterministic problems, this requires a careful compatibility
	between the moving geometry, traces, the momentum equation, and the
	compactness of the density; see, among others,
	\cite{BoulakiaGuerrero2009,Haak2019,Boulakia2005,BoulakiaGuerrero2010,
		KukavicaTuffaha2012,BoulakiaGuerrero2017,FloriOrenga1999,Flori2000,
		BreitSchwarzacher2018,BreitSchwarzacher2023,Macha2022,Mitra2020}.
	
	The interface law changes the compactness problem in an essential way.
	Under a no-slip condition the full fluid trace is tied to the structure
	velocity.  Under Navier slip, by contrast, only the normal velocities
	coincide, while the tangential mismatch is controlled through a
	friction law.  This distinction is physically natural but analytically
	important: a penalty approximation must recover the normal kinematic
	constraint without accidentally imposing tangential no-slip.  Weak
	theories with slip have been developed in deterministic and stochastic
	FSI; see
	\cite{MuhaCanic2016,NeustupaPenel2010,Liu2024,Tawri2024}.  In
	particular, the stochastic Navier--slip model of \cite{Tawri2024}
	concerns an incompressible fluid.  Hence it does not involve the density
	compactness, effective-viscous-flux, and pressure-identification mechanisms
	that are unavoidable in the present compressible setting.  Here the
	normal/tangential decomposition must therefore be closed simultaneously
	with the density and pressure compactness.
	
	Stochastic FSI introduces a second layer of difficulty.  Existing
	moving-boundary stochastic theories concern predominantly incompressible
	fluids \cite{KuanCanic2024,TawriCanic2025,Tawri2025a,Tawri2024},
	while \cite{KuanTawri2025} treats a compressible no-slip problem in which
	stochastic forcing enters the fluid momentum and structure equations, while
	the continuity equation remains deterministic; see also
	\cite{BreitMensahMoyo2024} for martingale solutions in stochastic FSI.  On fixed domains, stochastic
	compressible Navier--Stokes equations are well developed
	\cite{Feireisl2013,BreitHofmanova2016,BreitFeireislHofmanova2018}, and
	transport-type perturbations have recently been studied in
	\cite{BreitFeireislHofmanovaZatorska2022,
		BreitFeireislHofmanovaMucha2026}.  Such noise is geometrically different
	from an additive or multiplicative body force: in Stratonovich form it
	acts by random transport and is therefore naturally associated with a
	stochastic flow; cf. \cite{Holm2015,Cotter2017,
		HofmanovaLangePappalettera2024}.
	
	The present problem combines features that are separated in these
	theories.  The density and momentum are transported by Stratonovich
	noise, the elastic boundary is driven by the same stochastic
	characteristics, the fluid is compressible, and the coupling is of
	Navier-slip type.  The use of a common stochastic flow is therefore not
	only a convenient change of variables: it is the mechanism that keeps
	the fluid and shell geometrically synchronized.  After the
	transformation the stochastic transport integrals disappear, but the
	price is a random moving-domain problem whose pullback coefficients are
	only H\"older continuous in time.  Thus the classical deterministic
	compactness tools cannot be inserted verbatim; they have to be combined
	with localization, coefficient freezing, and stochastic compactness in
	a way that preserves the common Wiener structure.

	\subsection{The stochastic FSI model}\label{sec:physical_model}
	
	Let $\Gamma=\mathbb T^2$ denote the two-dimensional flat torus used as
	the reference midsurface of the shell. We write
	\[
	\bm x=(\bm y,z),\qquad \bm y=(x,y)\in\Gamma,
	\]
	and $\mathbf e_z=(0,0,1)$ for the vertical unit vector. We regard
	$\bm y$ as the shell material coordinate. The scalar field
	$\eta=\eta(t,\bm y)$ denotes the vertical displacement of the shell from
	the reference graph $z=1$. Thus the fluid domain and moving interface are
	\begin{align*}     & \mathcal O_\eta(t)=\{\bm x:\bm y\in\Gamma,\ 0<z<1+\eta(t,\bm y)\}, \\
		& \Gamma_\eta(t)=\{(\bm y,1+\eta(t,\bm y)):\bm y\in\Gamma\}.
	\end{align*}
	The fluid satisfies
	\begin{align}
		& \dd\rho+\Div(\rho\bm u)\,\dd t
		=\sum_{k=1}^K\Div(\rho\mathbf Q_k)\circ\dd W_k,
		\label{con} \\
		& \dd(\rho\bm u)
		+\bigl[\Div(\rho\bm u\otimes\bm u)+\nabla p(\rho)
		-\Div\mathbb S(\nabla\bm u)\bigr]\dd t
		=\sum_{k=1}^K\Div(\rho\bm u\otimes\mathbf Q_k)\circ\dd W_k,
		\label{mom} 
	\end{align}
	with $p(\rho)=a\rho^\gamma$, $\gamma>3/2$, and
	\[
	\mathbb S(\nabla\bm u)
	=\mu(\nabla\bm u+\nabla\bm u^\top)
	+\lambda\Div\bm u\,\mathbb I,
	\qquad \mu>0,\quad \lambda+\frac23\mu\ge0.
	\]
	The horizontal variables are periodic, and the fixed bottom boundary is
	\[
	\Gamma_b:=\Gamma\times\{0\},
	\]
	on which the no-slip condition is imposed.
	
	Write
	\[
	\Phi_\eta(\bm y)=(\bm y,1+\eta(\bm y)),\qquad
	\mathbf N^\eta=(-\nabla_\Gamma\eta,1),\qquad
	J_\eta=|\mathbf N^\eta|,\qquad
	\mathbf n^\eta=\frac{\mathbf N^\eta}{J_\eta},
	\]
	and let
	$\mathbf a_{\tau^\eta}=(\mathbb I-\mathbf n^\eta\otimes\mathbf n^\eta)\mathbf a$.
	The fluid and shell are coupled by the standard Navier-slip interface law
	used in moving-boundary FSI; see, for example,
	\cite{MuhaCanic2016,Liu2024,Tawri2024}. Namely,
	\begin{align}
		& (\bm u\cdot\mathbf n^\eta)\circ\Phi_\eta
		=(r\mathbf e_z)\cdot(\mathbf n^\eta\circ\Phi_\eta),
		\label{eq:physical_kinematic} \\
		& \bigl(\mathbb S(\nabla\bm u)\mathbf n^\eta\bigr)_{\tau^\eta}\circ\Phi_\eta
		=-\iota
		\left(\bm u-r\mathbf e_z\circ\Phi_\eta^{-1}\right)_{\tau^\eta}
		\circ\Phi_\eta,
		\label{eq:physical_dynamic_slip}
	\end{align}
	where $\iota\ge0$. Condition~\eqref{eq:physical_kinematic} is the
	impermeability condition: only the normal components of the fluid and shell
	velocities are required to agree. Condition~\eqref{eq:physical_dynamic_slip}
	is the linear Navier friction law, relating tangential traction to the
	relative tangential velocity. For $\iota>0$ it contributes a nonnegative
	interface dissipation, whereas $\iota=0$ corresponds to vanishing
	tangential traction. The vertical fluid force per unit reference area is
	\begin{equation*}		\mathcal F^\eta(\rho,\bm u)
		=-\left[
		\bigl(\mathbb S(\nabla\bm u)-p(\rho)\mathbb I\bigr)\mathbf N^\eta
		\right]\circ\Phi_\eta\cdot\mathbf e_z.
	\end{equation*}
	
	The shell is transported by the same surface fields that drive the fluid
	noise. Introducing the stochastic material velocity $r$ as the drift velocity
	along the shell characteristics, so that the drift part of $\dd\eta$ is
	$r\,\dd t$, we write
	\begin{align}
		& \dd\eta
		=r\,\dd t
		+\sum_{k=1}^K(\mathcal Y_k\cdot\nabla_\Gamma\eta)\circ\dd W_k,
		\label{eq:material_eta} \\
		& \dd r
		+\bigl(\Delta_\Gamma^2\eta-\nu_s\Delta_\Gamma r
		-\mathcal F^\eta(\rho,\bm u)\bigr)\dd t
		=\sum_{k=1}^K(\mathcal Y_k\cdot\nabla_\Gamma r)\circ\dd W_k.
		\label{eq:material_r} 
	\end{align}
	Thus $r$ is the shell velocity along the stochastic characteristics; after
	the common-flow pullback below it becomes the ordinary transformed velocity
	$w=\partial_t\zeta$.
	The fluid and shell transport fields are chosen compatibly:
	\begin{equation}
		\mathbf Q_k(\bm y,z)=(\mathcal Y_k(\bm y),0),\qquad
		\Div_\Gamma\mathcal Y_k=0,\qquad k=1,\ldots,K.
		\label{eq:noise_geometry}
	\end{equation}
	Thus the fluid transport has no vertical component and restricts on each
	horizontal slice to the same field that transports the shell. With this
	choice, one common stochastic flow removes all transport integrals from the
	fluid and shell equations.
	Our main result establishes the existence of a martingale solution for
	$\gamma>3/2$ up to an almost surely positive stopping time preceding loss of
	the graph geometry.
	
\subsection{Proof strategy and organization}

Let $\phi$ be the volume-preserving Stratonovich flow generated by the
transport fields and let $\chi=\phi|_\Gamma$. Following the stochastic-flow
approach to transport noise
\cite{BreitFeireislHofmanovaZatorska2022,
	BreitFeireislHofmanovaMucha2026,
	HofmanovaLangePappalettera2024},
we pull back the fluid and shell along the same stochastic characteristics:
\[
\varrho=\rho\circ\phi,\qquad
\mathbf v=\bm u\circ\phi,\qquad
\zeta=\eta\circ\chi,\qquad
w=r\circ\chi .
\]
The stochastic transport integrals then disappear from the transformed
equations, and the problem becomes a pathwise random FSI system on a graph
domain. A feature specific to the present coupling is that the same flow
acts on both phases, so that the fluid domain and the shell remain
geometrically synchronized after the transformation. On bounded-flow events,
the resulting pullback coefficients are smooth in space and
$C_t^\vartheta$, $\vartheta<1/2$, in time.

The construction is performed on the fixed maximal cylinder
\[
\mathcal O_\alpha=\Gamma\times(0,2+\alpha^{-1})
\]
with five approximation parameters,
\[
\Delta t\to0,\qquad
m\to\infty,\qquad
\epsilon\to0,\qquad
l\to0,\qquad
\delta\to0,
\]
removed in this order. The overall approximation architecture shares several
ingredients with the compressible stochastic FSI construction of
\cite{KuanTawri2025}, including time splitting, extension to a fixed maximal
domain, localization by stopped structure displacements, and a tubular
penalty near the moving interface. Weak compactness is combined with the
Jakubowski--Skorokhod representation theorem \cite{Jakubowski1997}.
At fixed $m$ and $\epsilon>0$, the splitting defect is controlled
quantitatively, and the Galerkin limit is taken while the density remains
parabolic and the inertial coefficient is still mollified.

The limit $\epsilon\to0$ is the first stage at which strong density
compactness must be recovered. We use the classical Lions--Feireisl
renormalization and effective-viscous-flux mechanism
\cite{Lions1998,Feireisl2001}. In the present setting, however, the
effective-flux argument is applied to pullback operators whose coefficients
are random and only H\"older continuous in time. We therefore localize on
bounded-flow events and freeze these coefficients on short deterministic
time intervals before applying the local elliptic estimates; compare
\cite{BreitFeireislHofmanovaMucha2026}. The freezing errors are removed
before the localization is released.

The limit $l\to0$ removes the mollified inertial coefficient. The resulting
defect is decomposed into Friedrichs commutators and ordinary mollification
errors, which are combined with momentum compactness to recover the
unmollified convective term. The commutator estimate itself is classical;
its role here is to preserve the effective-viscous-flux structure through
the removal of the inertial regularization.

The final limit $\delta\to0$ removes simultaneously the artificial pressure,
the exterior viscosity, and the normal penalty. The exterior-mass control
and the use of adapted test functions are related to the final
$\delta$-limit strategy of \cite{KuanTawri2025}; in the present problem,
however, the Navier-slip condition requires a different interface
decomposition. The singular penalty acts only on the normal mismatch, while
the tangential relative velocity is controlled by the Navier-friction
energy. The analysis further combines boundary-pressure tightness, global
pressure identification, and subcritical Korn--Sobolev estimates on
$H^2$ graph domains. The tubular penalty yields the limiting normal trace,
whereas the tangential trace is identified independently through the
friction dissipation.

The limiting weak formulation is recovered with geometry-adapted
semimartingale tests. In contrast to a no-slip test construction, these
tests enforce only the transformed normal compatibility condition and leave
the tangential component unconstrained. They satisfy the $\delta$-level
normal condition exactly, so the singular penalty vanishes identically on
the test pair. Their construction is compatible with the stochastic
filtration and with the random pullback geometry, and allows the complete
coupled residual to pass to the limit without introducing a separate
pairing of the $L^1$ pressure with a rough divergence.

The stochastic energy inequality is recovered only after the nonlinear
limits have been identified. The shell transport fields are assumed smooth
and divergence free, but not Killing. Consequently, the bending energy does
not enjoy an exact stochastic isometry cancellation. Instead, the
bending-energy martingale and its It\^o correction are controlled at the
natural $H^2(\Gamma)$ level by commutator identities. These identities are
combined, at the same regularization scale, with a normal-compatible
kinetic--internal-energy inequality. The elastic cross terms cancel before
the regularization is removed, yielding the final total-energy inequality
without introducing an independent low-regularity fluid traction.

The remainder of the paper is organized as follows.
Section~\ref{sec:model} introduces the stochastic-flow formulation, the
solution concept, and the existence theorem for the transformed system.
Section~\ref{sec:approx} constructs the multi-layer approximation.
Sections~\ref{s4}--\ref{s9} carry out the successive limit passages.
The effective-viscous-flux estimates, the removal of the inertial
mollification, and the $H^2$ graph estimates required in the final
$\delta$-limit are developed at the corresponding stages.
Appendix~\ref{app:shell_ito} contains the geometry-adapted test construction
and the low-regularity identities used in the final stochastic energy
recovery.

\section{The transformed system and main results}
	\label{sec:model}

	\subsection{The stochastic flow transformation}

	In this subsection we introduce the common stochastic-flow transformation
	for the fluid and the elastic shell.

	Under \eqref{eq:noise_geometry}, the shell transport fields enter the
	transformed energy estimate through the following commutator bound.
	
	\begin{lemma}[Shell-energy commutator estimate]
		\label{lem:shell_energy_commutator}
		Let
		\[
		A_k:=\mathcal Y_k\cdot\nabla_\Gamma,
		\qquad
		C_k:=[\Delta_\Gamma,A_k]
		=\Delta_\Gamma A_k-A_k\Delta_\Gamma.
		\]
		Assume $\Div_\Gamma\mathcal Y_k=0$ and
		$\mathcal Y_k\in W^{3,\infty}(\Gamma;\mathbb R^2)$. Then, for every
		$f\in C^\infty(\Gamma)$,
		\begin{align}
			& \mathcal G_k(f)
			:=
			\left\langle
			\Delta_\Gamma f,
			\Delta_\Gamma A_k f
			\right\rangle_{L^2(\Gamma)}
			=
			\left\langle
			\Delta_\Gamma f,
			C_k f
			\right\rangle_{L^2(\Gamma)},
			\label{eq:Gk_commutator} \\
			& \mathcal R_k(f)
			:=
			\|\Delta_\Gamma A_k f\|_{L^2(\Gamma)}^2
			+
			\left\langle
			\Delta_\Gamma f,
			\Delta_\Gamma A_k^2f
			\right\rangle_{L^2(\Gamma)}
			\notag \\
			& =
			\|C_k f\|_{L^2(\Gamma)}^2
			+
			\left\langle
			\Delta_\Gamma f,
			[C_k,A_k]f
			\right\rangle_{L^2(\Gamma)}.
			\label{eq:Rk_commutator} 
		\end{align}
		Moreover, there exists a constant $c_k$, depending only on
		$\|\mathcal Y_k\|_{W^{3,\infty}}$ and on $\Gamma$, such that
		\begin{equation}
			|\mathcal G_k(f)|+|\mathcal R_k(f)|
			\le
			c_k\|f\|_{H^2(\Gamma)}^2.
			\label{eq:commutator_energy_bound}
		\end{equation}
		On the flat torus, since the operators involved contain no zeroth-order
		terms, the right-hand side may equivalently be bounded by
		$c_k\|\Delta_\Gamma f\|_{L^2(\Gamma)}^2$ up to the harmless constant mode.
	\end{lemma}
	
	\begin{proof}
		Since $\Div_\Gamma\mathcal Y_k=0$, the transport operator $A_k$ is
		skew-adjoint on $L^2(\Gamma)$, whereas $\Delta_\Gamma$ is self-adjoint.
		Consequently $C_k=[\Delta_\Gamma,A_k]$ is self-adjoint. Using
		$\Delta_\Gamma A_k=A_k\Delta_\Gamma+C_k$ gives
		\[
		\left\langle\Delta_\Gamma f,\Delta_\Gamma A_kf\right\rangle
		=
		\left\langle\Delta_\Gamma f,C_kf\right\rangle,
		\]
		which proves \eqref{eq:Gk_commutator}. Furthermore,
		\[
		\Delta_\Gamma A_k^2
		=A_k^2\Delta_\Gamma+A_kC_k+C_kA_k.
		\]
		Expanding $\|\Delta_\Gamma A_kf\|_2^2$ and using $A_k^*=-A_k$ yields
		\[
		\|\Delta_\Gamma A_kf\|_2^2
		+
		\left\langle\Delta_\Gamma f,\Delta_\Gamma A_k^2f\right\rangle
		=
		\|C_kf\|_2^2
		+
		\left\langle\Delta_\Gamma f,[C_k,A_k]f\right\rangle,
		\]
		which is \eqref{eq:Rk_commutator}. In flat coordinates,
		\[
		C_k f
		=(\Delta_\Gamma\mathcal Y_{k,j})\partial_jf
		+2(\partial_i\mathcal Y_{k,j})\partial_{ij}f.
		\]
		Thus $C_k$ is a second-order operator. The commutator $[C_k,A_k]$ is
		again of order at most two; its coefficients involve derivatives of
		$\mathcal Y_k$ of order at most three. Hence
		\[
		\|C_kf\|_{L^2}
		+
		\|[C_k,A_k]f\|_{L^2}
		\le
		c_k\|f\|_{H^2},
		\]
		and \eqref{eq:commutator_energy_bound} follows by Cauchy--Schwarz inequality.
	\end{proof}
	
	Define $\phi(t,\bm{x})$ as the solution of the Stratonovich SDE
	\begin{align}
		& \dd \phi(t,\bm{x})
		=
		-
		\sum_{k=1}^{K}
		\mathbf{Q}_k(\phi(t,\bm{x}))
		\circ \dd W_k(t),
		\qquad
		\phi(0,\bm{x})=\bm{x},
		\qquad
		\bm{x}\in\mathcal{O}_{\alpha}.
		\label{phiv} 
	\end{align}
	Let $\chi(t,\bm{y})$ denote the induced stochastic flow on
	the reference interface $\Gamma$:
	\begin{align}
		& \dd \chi(t,\bm{y})
		=
		-
		\sum_{k=1}^{K}
		\mathcal{Y}_k(\chi(t,\bm{y}))
		\circ \dd W_k(t),
		\qquad
		\chi(0,\bm{y})=\bm{y}.
		\label{eq:surface_flow} 
	\end{align}
	Under \eqref{eq:noise_geometry},
	\[
	\phi(t,\bm{y},z)
	=
	\bigl(\chi(t,\bm{y}),z\bigr).
	\]
	Moreover, since $\Div\mathbf{Q}_k=0$ and
	$\Div_{\Gamma}\mathcal{Y}_k=0$, the flows $\phi(t,\cdot)$ and
	$\chi(t,\cdot)$ are volume- and area-preserving, respectively. In
	particular,
	\[
	\det\nabla\phi(t,\bm{x})=1,
	\qquad
	\det\nabla_{\Gamma}\chi(t,\bm{y})=1
	\]
	almost surely.
	
	We introduce the transformed fluid variables by
	\[
	\varrho(t,\bm{x})
	=
	\rho(t,\phi(t,\bm{x})),
	\qquad
	\mathbf{v}(t,\bm{x})
	=
	\mathbf{u}(t,\phi(t,\bm{x})),
	\]
	and the transformed structure variables by
	\begin{equation*}		\zeta(t,\bm{y})
		=
		\eta(t,\chi(t,\bm{y})),
		\qquad
		w(t,\bm{y})
		=
		r(t,\chi(t,\bm{y})).
	\end{equation*}
	Notice that the transformed structure velocity is defined as the pullback
	of the stochastic material velocity $r$, rather than the pullback of the
	ordinary time derivative of $\eta$.
	
	The graph-preserving property of $\phi$ yields
	\[
	\begin{aligned}
		\phi^{-1}
		\bigl(t,\mathcal{O}_{\eta}(t)\bigr)
		=
		\left\{
		(\bm{y},z):
		0<z<1+\eta(t,\chi(t,\bm{y}))
		\right\} \\
		=
		\left\{
		(\bm{y},z):
		0<z<1+\zeta(t,\bm{y})
		\right\}
		=:
		\mathcal{O}_{\zeta}(t).
	\end{aligned}
	\]
	For later use, we define the pullback differential operators by
	\begin{align*}		& \Div^{\phi}\bm{\psi}
		:=
		\bigl[
		\Div(\bm{\psi}\circ\phi^{-1})
		\bigr]\circ\phi,
		\;\;
		\nabla^{\phi}\psi
		:=
		\bigl[
		\nabla(\psi\circ\phi^{-1})
		\bigr]\circ\phi,  
	\end{align*}
	and, on $\Gamma$,
	\begin{align*}		& \nabla^{\chi}_{\Gamma}\psi
		:=
		\bigl[
		\nabla_{\Gamma}(\psi\circ\chi^{-1})
		\bigr]\circ\chi,
		\;\;
		\Delta^{\chi}_{\Gamma}\psi
		:=
		\bigl[
		\Delta_{\Gamma}(\psi\circ\chi^{-1})
		\bigr]\circ\chi.  
	\end{align*}
	
	The following lemma makes explicit the cancellation of the stochastic
	transport terms.
	
	\begin{lemma}[Random-flow cancellation]
		
		Assume \eqref{eq:noise_geometry}, and let $\phi$ and $\chi$ be the
		Stratonovich flows defined by \eqref{phiv} and \eqref{eq:surface_flow}.
		Suppose initially that $(\rho,\mathbf u,\eta,r)$ is smooth and satisfies
		\eqref{con}--\eqref{mom} and
		\eqref{eq:material_eta}--\eqref{eq:material_r}. Then the
		transformed variables
		\[
		\varrho=\rho\circ\phi,
		\qquad
		\mathbf{v}=\mathbf{u}\circ\phi,
		\qquad
		\zeta=\eta\circ\chi,
		\qquad
		w=r\circ\chi
		\]
		satisfy, pathwise,
		\begin{align}
			& \partial_t\varrho
			+
			\Div^{\phi}(\varrho\mathbf{v})
			=0,
			\label{eq:cont_trans} \\
			& \partial_t(\varrho\mathbf{v})
			+
			\Div^{\phi}
			(\varrho\mathbf{v}\otimes\mathbf{v})
			+
			\nabla^{\phi}p(\varrho)
			=
			\Div^{\phi}
			\mathbb{S}(\nabla^{\phi}\mathbf{v}),
			\label{eq:mom_trans} \\
			& \partial_t\zeta
			=w,
			\label{eq:zeta_w} \\
			& \partial_t w
			+
			(\Delta^{\chi}_{\Gamma})^{2}\zeta
			-
			\nu_s\Delta^{\chi}_{\Gamma}w
			=
			\mathcal{F}^{\chi}_{\mathrm{fluid}}
			(\varrho,\mathbf{v}),
			\label{eq:str_trans_first_order} 
		\end{align}
		where
		\begin{equation*}			\mathcal{F}^{\chi}_{\mathrm{fluid}}
			(\varrho,\mathbf{v})
			:=
			\bigl[
			\mathcal{F}^{\eta}(\rho,\mathbf{u})
			\bigr]\circ\chi .
		\end{equation*}
	\end{lemma}
	
	\begin{proof}
		For smooth solutions this follows directly from the Stratonovich
		It\^o--Wentzell formula.  Since $\Div\mathbf Q_k=0$,
		\[
		\Div(\rho\mathbf Q_k)=(\mathbf Q_k\cdot\nabla)\rho,
		\qquad
		\Div(\rho\mathbf u\otimes\mathbf Q_k)
		=(\mathbf Q_k\cdot\nabla)(\rho\mathbf u),
		\]
		and composition with the flow generated by $-\mathbf Q_k$ cancels the fluid
		transport terms.  The same calculation on $\Gamma$ cancels the shell
		transport terms and gives $\partial_t\zeta=w$.
	\end{proof}
	
	Combining \eqref{eq:zeta_w} and
	\eqref{eq:str_trans_first_order}, the transformed shell equation may
	equivalently be written in second-order form as
	\begin{align}
		& \partial_t^2\zeta
		+
		(\Delta_{\Gamma}^{\chi})^{2}\zeta
		-
		\nu_s\Delta_{\Gamma}^{\chi}\partial_t\zeta
		=
		\mathcal{F}^{\chi}_{\mathrm{fluid}}
		(\varrho,\mathbf{v})
		\qquad\text{on }\Gamma.
		\label{eq:str_trans} 
	\end{align}
	Thus, after the stochastic flow transformation, the coupled system
	\eqref{eq:cont_trans}, \eqref{eq:mom_trans}, and
	\eqref{eq:str_trans_first_order} contains no stochastic integrals.
	The randomness enters only through the coefficients of the pullback
	operators and through the transformed moving domain
	$\mathcal{O}_{\zeta}(t)$.
	
	We record the transformed interface geometry. Here and below,
	$\operatorname{cof}F$ denotes the cofactor matrix of $F$; for invertible $F$,
	$\operatorname{cof}F=(\det F)F^{-T}$. The physical unit normal pulled back
	to $\Gamma_\zeta(t)$ is
	\begin{equation*}		\mathbf n^{\zeta,\phi}
		:=(\mathbf n^\eta)\circ\phi
		=\frac{\operatorname{cof}(\nabla\phi)\mathbf n^\zeta}
		{|\operatorname{cof}(\nabla\phi)\mathbf n^\zeta|},
		\qquad\text{on }\Gamma_\zeta(t),
	\end{equation*}
	and the corresponding tangential projection is
	\[
	(\cdot)_{\tau^{\zeta,\phi}}
	:=\bigl(\mathbb I-\mathbf n^{\zeta,\phi}\otimes
	\mathbf n^{\zeta,\phi}\bigr)(\cdot).
	\]
	The surface Jacobian associated with the restriction of $\phi$ to $\Gamma_\zeta(t)$ is
	\begin{equation*}		J_\Gamma^\phi
		:=|\operatorname{cof}(\nabla\phi)\mathbf n^\zeta|.
	\end{equation*}
	Because $\chi$ is area preserving on the flat reference torus, the weighted normals satisfy
	\[
	\mathbf N^\eta\circ\chi
	=\operatorname{cof}(\nabla\phi)\mathbf N^\zeta.
	\]
	Hence the pulled-back vertical traction appearing in the shell equation is
	\begin{align*}		& \mathcal F^\chi_{\mathrm{fluid}}(\varrho,\mathbf v)
		=-\left[
		\bigl(\mathbb S(\nabla^\phi\mathbf v)-p(\varrho)\mathbb I\bigr)
		\operatorname{cof}(\nabla\phi)\mathbf N^\zeta
		\right]\circ\Phi_\zeta\cdot\mathbf e_z \\
		& =-J_\zeta J_\Gamma^\phi
		\left[
		\bigl(\mathbb S(\nabla^\phi\mathbf v)-p(\varrho)\mathbb I\bigr)
		\mathbf n^{\zeta,\phi}
		\right]\circ\Phi_\zeta\cdot\mathbf e_z.
		\notag 
	\end{align*}
	
	\paragraph{Transformed Navier--slip condition.}
	Pulling back the physical interface laws
	\eqref{eq:physical_kinematic}--\eqref{eq:physical_dynamic_slip} by the
	common stochastic flow, and using the transformed material shell velocity
	\[
	w=r\circ\chi=\partial_t\zeta,
	\]
	gives
	\begin{align}
		& (\mathbf v\cdot\mathbf n^{\zeta,\phi})\circ\Phi_\zeta
		=(w\mathbf e_z)\cdot
		(\mathbf n^{\zeta,\phi}\circ\Phi_\zeta),
		\qquad\text{on }\Gamma,
		\label{eq:trans_kinematic} \\
		& \bigl(\mathbb S(\nabla^\phi\mathbf v)\mathbf n^{\zeta,\phi}\bigr)_{\tau^{\zeta,\phi}}
		\circ\Phi_\zeta
		=-\iota
		\left(\mathbf v-w\mathbf e_z\circ\Phi_\zeta^{-1}\right)_{\tau^{\zeta,\phi}}
		\circ\Phi_\zeta,
		\qquad\text{on }\Gamma.
		\label{eq:trans_dynamic_slip} 
	\end{align}
	In particular, \eqref{eq:trans_kinematic} imposes only normal-velocity continuity and does not imply
	$\mathbf v|_{\Gamma_\zeta}=w\mathbf e_z$. The factors
	$\mathbf n^{\zeta,\phi}$ and $J_\Gamma^\phi$ are retained unless an additional isometry property of the stochastic flow has been established.

	\medskip
	\noindent\textbf{Notation and conventions.}
	Time-dependent $L^p$-spaces on $\mathcal O_{\zeta(t)}$ are understood through
	zero extension to the fixed cylinder $\mathcal O_\alpha$, whereas Sobolev
	spaces are understood intrinsically on the moving domain. Statements prior
	to a stopping time are understood after localization to deterministic
	intervals $T'<\tau$.
	Since the vector fields $\mathbf{Q}_k$ are smooth, solenoidal, and satisfy
	the geometric compatibility condition introduced above, the stochastic
	flow $\phi(t,\cdot)$ defined by \eqref{phiv} is a
	$\Pp$-a.s. volume-preserving $C^\infty$-diffeomorphism of
	$\mathcal{O}_{\alpha}$. For every $\kappa\in(0,1/2)$,
	\begin{align*}		& \mathbb{E}
		\|\phi\|_{C^{\kappa}
			([0,T];C^2(\overline{\mathcal{O}}_{\alpha}))}
		+
		\mathbb{E}
		\|\phi^{-1}\|_{C^{\kappa}
			([0,T];C^2(\overline{\mathcal{O}}_{\alpha}))}
		<\infty. 
	\end{align*}
	Likewise, the induced surface flow $\chi(t,\cdot)$ is an
	area-preserving $C^\infty$-diffeomorphism of $\Gamma$ and satisfies
	\[
	\mathbb{E}
	\|\chi\|_{C^{\kappa}([0,T];C^2(\Gamma))}
	+
	\mathbb{E}
	\|\chi^{-1}\|_{C^{\kappa}([0,T];C^2(\Gamma))}
	<\infty,
	\qquad
	\kappa\in(0,1/2).
	\]
	
	The pullback coefficients
	\[
	\mathbf{A}_{\phi}(t,\bm{x})
	=
	\nabla\phi^{-1}(t,\phi(t,\bm{x}))
	\]
	have, in the general flow formulation, only H\"older continuity in time with exponent $\kappa\in(0,1/2)$. We therefore formulate the pressure compactness argument pathwise. On sufficiently short time intervals the coefficients are compared with frozen coefficients, which permits the use of the classical elliptic estimates entering the effective viscous flux argument; see \cite{BreitFeireislHofmanovaMucha2026}. If the admissible transport fields generate an isometric flow, then $\mathbf A_\phi$ simplifies and this localization step correspondingly reduces to the classical fixed-coefficient argument.

	For a given geometry $(\zeta,\phi)$, let
	$\mathcal T_{\rm ad}(\zeta,\phi)$ denote the geometry-adapted
	progressively measurable It\^o-semimartingale core constructed in
	Appendix~\ref{app:semimartingale_tests}. Each core element is obtained
	from a spatially smooth precursor pair $(\widetilde{\bm q},\psi)$ by
	projecting only the fluid component in a tubular neighborhood of the
	moving graph; the shell component $\psi$ is left unchanged. Thus $\psi$
	retains its spatial $H^2$ regularity, while the projected fluid test is
	only required to have energy-level $H^1$ regularity across the $H^2$
	interface. The core is determined by the geometry and does not depend on
	$(\varrho,\mathbf v,w)$. Its elements satisfy
	\[
	\bm q|_{\Gamma_b}=0,
	\qquad
	(\bm q\cdot\mathbf n^{\zeta,\phi})\circ\Phi_\zeta
	=
	(\psi\mathbf e_z)\cdot
	(\mathbf n^{\zeta,\phi}\circ\Phi_\zeta)
	\qquad\text{on }\Gamma.
	\]
	
	For a projected core pair, the pressure term is not extended as a
	separate functional on the rough divergence. Instead, the complete
	coupled momentum--structure residual is defined through the fixed causal
	smooth representatives of
	Appendix~\ref{app:semimartingale_tests}. For a spatially smooth
	compatible pair this residual coincides with the ordinary smooth weak
	residual.
	
	We further set
	\[
	\widehat\eta(t):=\zeta(t)\circ\chi^{-1}(t),
	\qquad
	\widehat r(t):=w(t)\circ\chi^{-1}(t),
	\]
	and define
	\begin{align}
		\mathcal E_{\rm tr}(t)
		&:=
		\int_{\mathcal O_\zeta(t)}
		\left(
		\frac12\varrho|\mathbf v|^2
		+\frac{a}{\gamma-1}\varrho^\gamma
		\right)\,\dd\bm x
		+\frac12\|w(t)\|_{L^2(\Gamma)}^2
		+\frac12
		\|\Delta_\Gamma^\chi\zeta(t)\|_{L^2(\Gamma)}^2,
		\label{eq:trans_energy_def}
		\\
		\mathcal D_{\rm tr}(t)
		&:=
		\int_0^t\!\!\int_{\mathcal O_\zeta(s)}
		\mathbb S(\nabla^\phi\mathbf v):
		\nabla^\phi\mathbf v
		\,\dd\bm x\,\dd s
		+
		\nu_s\int_0^t
		\|\nabla_\Gamma^\chi w\|_{L^2(\Gamma)}^2\,\dd s
		\notag\\
		&\quad+
		\iota
		\int_0^t\!\!\int_{\Gamma_\zeta(s)}
		J_\Gamma^\phi
		\left|
		\left(
		\mathbf v-w\mathbf e_z\circ\Phi_\zeta^{-1}
		\right)_{\tau^{\zeta,\phi}}
		\right|^2
		\,\dd S\,\dd s .
		\label{eq:trans_dissipation_def}
	\end{align}
	The energy is understood through its canonical lower-semicontinuous
	representative. Endpoint pairings in the coupled identity are understood
	through the canonical weakly continuous representative of the combined
	momentum--structure functional.
	
	We now give the definition of a martingale solution to
	\eqref{eq:cont_trans}--\eqref{eq:str_trans}.
	
	\begin{definition}[Martingale solution of the transformed system]
		\label{def:trans_solution}
		
		Let
		\[
		\zeta_0\in H^2(\Gamma),
		\qquad
		1+\zeta_0(\bm y)\ge\alpha_0>0,
		\]
		and
		\[
		\varrho_0\in L^\gamma(\mathcal O_{\zeta_0}),
		\qquad
		\bm m_0\in
		L^{\frac{2\gamma}{\gamma+1}}(\mathcal O_{\zeta_0}),
		\qquad
		\bm m_0=0
		\quad\text{a.e. on }\{\varrho_0=0\},
		\]
		with
		\[
		\frac{|\bm m_0|^2}{\varrho_0}
		\in L^1(\mathcal O_{\zeta_0}),
		\]
		where the quotient is set equal to zero on
		$\{\varrho_0=0\}$. Let $w_0\in L^2(\Gamma)$.
		
		A tuple
		\[
		\bigl(
		\Omega,\F,(\F_t)_{t\ge0},\Pp,W,
		\varrho,\mathbf v,\zeta,w,\tau^\zeta
		\bigr)
		\]
		is called a martingale solution of the transformed stochastic FSI
		system if the following conditions hold.
		
		\begin{enumerate}[label=(\arabic*)]
			
			\item
			$(\Omega,\F,(\F_t)_{t\ge0},\Pp)$ is a filtered probability
			space satisfying the usual conditions,
			$W=(W_1,\ldots,W_K)$ is a $K$-dimensional
			$(\F_t)$-Wiener process, and
			$\varrho,\mathbf v,\zeta,w$ are progressively measurable.
			
			\item
			$\tau^\zeta$ is an a.s. strictly positive
			$(\F_t)$-stopping time and, $\Pp$-a.s.,
			\begin{align*}
				&\varrho
				\in
				C_w([0,\tau^\zeta);
				L^\gamma(\mathcal O_{\zeta(\cdot)})),
				\qquad
				\varrho|\mathbf v|^2
				\in
				L^\infty(0,\tau^\zeta;
				L^1(\mathcal O_{\zeta(\cdot)})),
				\\
				&\mathbf v
				\in
				L^2(0,\tau^\zeta;
				H^1(\mathcal O_{\zeta(\cdot)})),
				\qquad
				\zeta
				\in
				C_w([0,\tau^\zeta);H^2(\Gamma)),
				\\
				&w
				\in
				L^\infty(0,\tau^\zeta;L^2(\Gamma))
				\cap
				L^2(0,\tau^\zeta;H^1(\Gamma)).
			\end{align*}
			
			\item
			For every $t<\tau^\zeta$,
			\[
			\zeta(t)
			=
			\zeta_0
			+
			\int_0^t w(s)\,\dd s
			\qquad
			\text{in }L^2(\Gamma).
			\]
			
			\item
			For every
			\[
			\varphi
			\in
			C_c^\infty
			\bigl(
			[0,T)\times
			\overline{\mathcal O}_{\zeta(\cdot)}
			\bigr)
			\]
			and every $b\in C^1(\R)$ with $b'$ compactly supported,
			the renormalized continuity equation
			\begin{align*}
				&\int_{\mathcal O_\zeta(t)}
				b(\varrho(t))\varphi(t)\,\dd\bm x
				-
				\int_{\mathcal O_{\zeta_0}}
				b(\varrho_0)\varphi(0)\,\dd\bm x
				\\
				&=
				\int_0^t\!\!\int_{\mathcal O_\zeta(s)}
				b(\varrho)
				\left(
				\partial_t\varphi
				+
				\mathbf v\cdot\nabla^\phi\varphi
				\right)
				\,\dd\bm x\,\dd s
				\\
				&\quad+
				\int_0^t\!\!\int_{\mathcal O_\zeta(s)}
				\bigl(
				b'(\varrho)\varrho-b(\varrho)
				\bigr)
				\Div^\phi\mathbf v\,
				\varphi
				\,\dd\bm x\,\dd s
			\end{align*}
			holds $\Pp$-a.s. for every $t<\tau^\zeta$.
			
			\item
			For every
			$(\bm q,\psi)\in\mathcal T_{\rm ad}(\zeta,\phi)$
			and every $T'<\tau^\zeta$, the fixed causal coupled residual
			defined in Appendix~\ref{app:semimartingale_tests} satisfies
			\[
			\mathfrak C_t
			(\varrho,\mathbf v,\zeta,w;\bm q,\psi)
			=0,
			\qquad
			0\le t\le T',
			\quad \Pp\text{-a.s.}
			\]
			
			For a spatially smooth compatible pair, the fixed causal residual
			coincides with the ordinary smooth residual. Writing
			\[
			\dd\bm q
			=
			\bm q_0\,\dd t
			+
			\sum_{k=1}^K
			\bm q_k\,\dd W_k,
			\qquad
			\dd\psi
			=
			\psi_0\,\dd t
			+
			\sum_{k=1}^K
			\psi_k\,\dd W_k,
			\]
			the coupled momentum--structure identity reads
			\begin{align*}
				&\int_{\mathcal O_\zeta(t)}
				\varrho(t)\mathbf v(t)\cdot\bm q(t)\,\dd\bm x
				+
				\int_\Gamma
				w(t)\psi(t)\,\dd\bm y
				\\
				&=
				\int_{\mathcal O_{\zeta_0}}
				\bm m_0\cdot\bm q(0)\,\dd\bm x
				+
				\int_\Gamma
				w_0\psi(0)\,\dd\bm y
				\\
				&\quad+
				\int_0^t\!\!\int_{\mathcal O_\zeta(s)}
				\Bigl[
				\varrho\mathbf v\cdot\bm q_0
				+
				\varrho\mathbf v\otimes\mathbf v:
				\nabla^\phi\bm q
				-
				\mathbb S(\nabla^\phi\mathbf v):
				\nabla^\phi\bm q
				+
				p(\varrho)\Div^\phi\bm q
				\Bigr]
				\,\dd\bm x\,\dd s
				\\
				&\quad+
				\int_0^t\!\!\int_\Gamma
				\Bigl[
				w\psi_0
				-
				\nu_s\nabla_\Gamma^\chi w
				\cdot\nabla_\Gamma^\chi\psi
				-
				\Delta_\Gamma^\chi\zeta\,
				\Delta_\Gamma^\chi\psi
				\Bigr]
				\,\dd\bm y\,\dd s
				\\
				&\quad-
				\iota
				\int_0^t\!\!\int_{\Gamma_\zeta(s)}
				J_\Gamma^\phi
				\left(
				\mathbf v
				-
				w\mathbf e_z\circ\Phi_\zeta^{-1}
				\right)_{\tau^{\zeta,\phi}}
				\cdot
				\left(
				\bm q
				-
				\psi\mathbf e_z\circ\Phi_\zeta^{-1}
				\right)_{\tau^{\zeta,\phi}}
				\,\dd S\,\dd s
				\\
				&\quad+
				\sum_{k=1}^K
				\int_0^t
				\left[
				\int_{\mathcal O_\zeta(s)}
				\varrho\mathbf v\cdot\bm q_k\,\dd\bm x
				+
				\int_\Gamma
				w\psi_k\,\dd\bm y
				\right]
				\,\dd W_k(s).
			\end{align*}
			
			For a general projected core pair, only the complete coupled
			residual is interpreted through the fixed causal closure; the
			pressure contribution is not defined separately on the rough
			divergence.
			
			\item
			The normal kinematic condition
			\[
			(\mathbf v\cdot\mathbf n^{\zeta,\phi})
			\circ\Phi_\zeta
			=
			(w\mathbf e_z)\cdot
			(\mathbf n^{\zeta,\phi}\circ\Phi_\zeta)
			\qquad\text{on }\Gamma
			\]
			holds $\Pp$-a.s. for a.e. $t<\tau^\zeta$.
			
			\item
			For every $t<\tau^\zeta$,
			\[
			\mathcal E_{\rm tr}(t)
			+
			\mathcal D_{\rm tr}(t)
			\le
			\mathcal E_{\rm tr}(0)
			+
			\sum_{k=1}^K
			\int_0^t
			\mathcal G_k(\widehat\eta(s))\,\dd W_k(s)
			+
			\frac12
			\sum_{k=1}^K
			\int_0^t
			\mathcal R_k(\widehat\eta(s))\,\dd s,
			\]
			where $\mathcal G_k$ and $\mathcal R_k$ are defined in
			\eqref{eq:Gk_commutator}--\eqref{eq:Rk_commutator}.
			
		\end{enumerate}
	\end{definition}

	We now state the existence theorem for the transformed system
	\eqref{eq:cont_trans}--\eqref{eq:str_trans_first_order}.
	
	\begin{theorem}[Existence for the transformed system]
		\label{thm:transformed_existence}
		Let $\gamma>3/2$ and assume \eqref{eq:noise_geometry}, with
		$\mathbf Q_k\in
		C^\infty(\overline{\mathcal O}_\alpha;\mathbb R^3)$.
		Let
		\[
		\varrho_0\in L^\gamma(\mathcal O_{\zeta_0}),
		\qquad
		\bm m_0\in
		L^{\frac{2\gamma}{\gamma+1}}(\mathcal O_{\zeta_0}),
		\qquad
		\bm m_0=0
		\quad\text{a.e. on }\{\varrho_0=0\},
		\]
		and assume
		\[
		\frac{|\bm m_0|^2}{\varrho_0}
		\in L^1(\mathcal O_{\zeta_0}),
		\]
		where the quotient is set equal to zero on
		$\{\varrho_0=0\}$. Let
		\[
		\zeta_0\in H^2(\Gamma),
		\qquad
		w_0\in L^2(\Gamma),
		\qquad
		1+\zeta_0\ge\alpha_0>0.
		\]
		Fix $s\in(3/2,2)$ and choose
		\begin{equation}
			\label{eq:initial_localization_margin}
			0<\alpha<
			\min\left\{
			\frac{\alpha_0}{2},
			\frac{1}{
				2(1+\|\zeta_0\|_{H^s(\Gamma)})
			}
			\right\}.
		\end{equation}
		Then there exist a stochastic basis, a $K$-dimensional Wiener process
		$W$, and processes $(\varrho,\mathbf v,\zeta,w)$ with a
		$\Pp$-a.s. positive stopping time $\tau^\zeta$ that form a martingale
		solution, in the sense of Definition~\ref{def:trans_solution}, of
		\eqref{eq:cont_trans}--\eqref{eq:str_trans_first_order} with
		\eqref{eq:trans_kinematic}--\eqref{eq:trans_dynamic_slip}.
	\end{theorem}

\subsection{Physical variables and main result}

The physical variables are recovered by the inverse stochastic flow.
This allows us to avoid repeating in physical coordinates every item of
the transformed weak formulation.

\begin{definition}[Martingale solution of the stochastic FSI system]
	\label{def:physical_solution}
	\label{def:weak_martingale}
	
	Let $\phi$ and $\chi=\phi|_\Gamma$ be the stochastic flows generated by
	\eqref{phiv} and \eqref{eq:surface_flow}. A tuple
	\[
	\bigl(
	\Omega,\mathcal F,(\mathcal F_t)_{t\ge0},\mathbb P,W,
	\rho,\bm u,\eta,r,\tau^\eta
	\bigr)
	\]
	is called a martingale solution of the physical system
	\eqref{con}--\eqref{mom},
	\eqref{eq:material_eta}--\eqref{eq:material_r}
	with Navier slip if, setting
	\[
	\varrho=\rho\circ\phi,
	\qquad
	\mathbf v=\bm u\circ\phi,
	\qquad
	\zeta=\eta\circ\chi,
	\qquad
	w=r\circ\chi,
	\qquad
	\tau^\zeta:=\tau^\eta,
	\]
	the tuple
	\[
	\bigl(
	\Omega,\mathcal F,(\mathcal F_t)_{t\ge0},\mathbb P,W,
	\varrho,\mathbf v,\zeta,w,\tau^\zeta
	\bigr)
	\]
	is a martingale solution in the sense of
	Definition~\ref{def:trans_solution}.
	
	The initial data are related by the same pullback and therefore coincide
	in the two coordinate systems at $t=0$, since
	$\phi(0)=\chi(0)=\mathrm{Id}$.
\end{definition}

The next lemma shows that Definition~\ref{def:physical_solution} agrees
with the distributional physical formulation and is not an additional
solution concept introduced by the stochastic-flow transformation.

\begin{lemma}[Weak It\^o--Wentzell equivalence]
	\label{lem:equivalence}
	
	Up to the common non-collision stopping time, composition with
	$(\phi,\chi)$ gives a one-to-one correspondence between martingale
	solutions in the sense of Definition~\ref{def:physical_solution} and
	martingale solutions in the sense of
	Definition~\ref{def:trans_solution}. Equivalently, inverse-flow
	transported smooth tests recover the distributional Stratonovich
	continuity equation and the coupled momentum--structure identity in
	physical coordinates.
\end{lemma}

\begin{proof}
	For smooth solutions, the assertion follows from the Stratonovich
	It\^o--Wentzell formula, the volume- and area-preserving changes of
	variables, and the Piola transformation of the interface normal.
	
	At energy regularity, fix $T'<\tau^\eta$ and localize on a bounded-flow
	event. Spatially regularize the continuity equation, the shell transport
	equation, and the coupled momentum--structure identity. The Friedrichs
	commutators
	\[
	[J_\ell,\mathbf Q_k\cdot\nabla]f,
	\qquad
	[J_\ell,\mathcal Y_k\cdot\nabla_\Gamma]g
	\]
	converge to zero in the negative Sobolev dualities determined by the
	energy bounds. The bulk terms pass to the limit by the
	volume-preserving change of variables, while the interface terms pass
	through the Piola identity and the coupled trace duality. In particular,
	no standalone low-regularity fluid traction is introduced.
	
	Letting $\ell\to0$ yields the physical weak identities on
	$[0,T']$. Since $T'<\tau^\eta$ is arbitrary, the identities hold up to
	the stopping time. Applying the same argument to $(\phi^{-1},\chi^{-1})$
	gives the converse implication.
\end{proof}

\begin{theorem}[Main result]
	\label{thm:main}
	
	Let $\gamma>3/2$ and assume \eqref{eq:noise_geometry}. Suppose that the
	transport fields are smooth on an open cylinder containing
	$\overline{\mathcal O}_\alpha$.
	
	Let
	\[
	\rho_0\in L^\gamma(\mathcal O_{\eta_0}),
	\qquad
	\rho_0\ge0
	\quad\text{a.e.},
	\]
	and
	\[
	\bm m_0
	\in
	L^{\frac{2\gamma}{\gamma+1}}(\mathcal O_{\eta_0}),
	\qquad
	\bm m_0=0
	\quad\text{a.e. on }\{\rho_0=0\},
	\]
	with
	\[
	\frac{|\bm m_0|^2}{\rho_0}
	\in L^1(\mathcal O_{\eta_0}),
	\]
	where the quotient is set equal to zero on $\{\rho_0=0\}$.
	Let
	\[
	\eta_0\in H^2(\Gamma),
	\qquad
	r_0\in L^2(\Gamma),
	\qquad
	1+\eta_0\ge\alpha_0>0.
	\]
	
	Fix $s\in(3/2,2)$ and choose $\alpha$ as in
	\eqref{eq:initial_localization_margin} with $\zeta_0=\eta_0$.
	Then there exist a stochastic basis, a $K$-dimensional Wiener process
	$W$, processes $(\rho,\bm u,\eta,r)$, and a
	$\mathbb P$-a.s. positive stopping time $\tau^\eta$ which form a
	martingale solution in the sense of
	Definition~\ref{def:physical_solution}.
\end{theorem}
	
	\section{The multi-layer approximation scheme}
	\label{sec:approx}
	
	We construct the transformed problem on the fixed cylinder
	\[
	\mathcal O_\alpha=\mathbb T^2\times(0,2+\alpha^{-1}),
	\]
	with the stopping thresholds chosen as in
	\eqref{eq:initial_localization_margin}.  
	
	The approximation uses five parameters:
	a time step $\Delta t$, a Galerkin dimension $m$, artificial density diffusion
	$\epsilon$, a spatial mollification scale $l$, and the parameter $\delta$ for
	artificial pressure, exterior viscosity, and the normal interface penalty.  They
	are removed in the order stated in the Introduction.  The shell transport
	fields are not assumed to commute with $\Delta_\Gamma$; the stochastic
	bending-energy terms are controlled by Lemma~\ref{lem:shell_energy_commutator}.
	
	\subsection{Extension to a fixed maximal domain}
	
	Given the extension of the transformed problem from the physical moving
	domain $\mathcal{O}_\zeta(t)$ to the fixed maximal domain
	$\mathcal{O}_\alpha$, we extend both the viscosity coefficients and the
	initial fluid data to $\mathcal{O}_\alpha$.

	\medskip
	\noindent
	{\bf Extension of the viscosity coefficients.}
	
	Following the fixed-domain extension strategy of \cite{KuanTawri2025}, we construct a smooth cut-off which is
	equal to one on the physical fluid domain and vanishes away from it; the
	extended shear viscosity itself remains bounded below by $\kappa$. Let
	$\varphi_\kappa$ be a standard spatial mollifier on $\Gamma$ at scale
	$\kappa$. For a given structure displacement $\zeta$, define
	\[
	a_\kappa^\zeta
	:=
	1+
	\bigl(
	\zeta+2C_\alpha\kappa^{1/2}
	\bigr)*\varphi_\kappa,
	\qquad
	b_\kappa^\zeta
	:=
	1+
	\bigl(
	\zeta-2C_\alpha\kappa^{1/2}
	\bigr)*\varphi_\kappa,
	\]
	where $C_\alpha$ is chosen so that, on the stopped set,
	\[
	\|\zeta*\varphi_\kappa-\zeta\|_{L^\infty(\Gamma)}
	\le
	C_\alpha\kappa^{1/2}.
	\]
	Then
	\[
	b_\kappa^\zeta
	\le
	1+\zeta
	\le
	a_\kappa^\zeta,
	\]
	and
	\[
	C_\alpha\kappa^{1/2}
	\le
	a_\kappa^\zeta-(1+\zeta)
	\le
	3C_\alpha\kappa^{1/2},
	\]
	\[
	C_\alpha\kappa^{1/2}
	\le
	(1+\zeta)-b_\kappa^\zeta
	\le
	3C_\alpha\kappa^{1/2}.
	\]
	
	Let
	$
	\vartheta_\kappa^\zeta:
	\mathcal O_\alpha\to[0,1]
	$
	be defined by
	\[
	\vartheta_\kappa^\zeta(x,y,z)
	:=
	\phi_0
	\left(
	\frac{z-a_\kappa^\zeta(x,y)}{\kappa^{1/2}}
	\right),
	\]
	where $\phi_0\in C^\infty(\mathbb R)$ is non-increasing and satisfies
	\[
	\phi_0(s)=1\quad\text{for }s\le0,
	\qquad
	\phi_0(s)=0\quad\text{for }s\ge1.
	\]
	Since $1+\zeta\le a_\kappa^\zeta$, we have
	\[
	\vartheta_\kappa^\zeta=1
	\qquad\text{on }\mathcal O_\zeta,
	\]
	whereas $\vartheta_\kappa^\zeta=0$ sufficiently far outside the physical
	fluid region.
	
	We then set
	\[
	\mu_\kappa^\zeta
	:=
	\vartheta_\kappa^\zeta\mu
	+
	(1-\vartheta_\kappa^\zeta)\kappa,
	\qquad
	\lambda_\kappa^\zeta
	:=
	\vartheta_\kappa^\zeta\lambda.
	\]
	Thus
	\[
	\mu_\kappa^\zeta=\mu,
	\qquad
	\lambda_\kappa^\zeta=\lambda
	\qquad
	\text{on }\mathcal O_\zeta,
	\]
	while outside the physical region the shear viscosity remains strictly
	positive. Moreover,
	\[
	\lambda_\kappa^\zeta+\frac23\mu_\kappa^\zeta
	=
	\vartheta_\kappa^\zeta
	\left(\lambda+\frac23\mu\right)
	+
	(1-\vartheta_\kappa^\zeta)\frac23\kappa
	\ge0.
	\]
	Hence the physical coercivity condition is preserved by the extension.
	
	The corresponding extended viscous stress is
	\begin{equation*}		\mathbb S_\kappa^\zeta(\nabla^\phi\mathbf v)
		:=
		\mu_\kappa^\zeta
		\left(
		\nabla^\phi\mathbf v+(\nabla^\phi\mathbf v)^\top
		\right)
		+
		\lambda_\kappa^\zeta
		\Div^\phi\mathbf v\,\mathbb I.
	\end{equation*}
	
	\medskip
	\noindent
	{\bf Extension and approximation of the initial data.}
	
	The artificial-pressure exponent is first needed in the initial-data
	approximation. We therefore fix, once and for all,
	\begin{equation}
		\label{eq:global_beta_choice}
		\beta>\max\{6,\gamma\},
		\qquad
		a_\beta:=\frac12-\frac1\beta>\frac13.
	\end{equation}
	The same fixed exponent is used at all subsequent approximation levels.
	
	Write
	\[
	\mathbf m_0:=(\varrho\mathbf v)_0
	\quad\text{on }\mathcal O_{\zeta_0},
	\]
	and extend $(\varrho_0,\mathbf m_0)$ by zero to $\mathcal O_\alpha$; the
	extended fields are denoted by $(\widetilde\varrho_0,\widetilde{\mathbf m}_0)$.
	The initial approximation is chosen in the same order as the limits
	$\epsilon\to0$ and $\delta\to0$.
	
	First, for each $\delta>0$ choose
	\[
	\varrho_{0,\delta}\in L^\beta(\mathcal O_\alpha),
	\qquad
	\mathbf m_{0,\delta}\in
	L^{\frac{2\beta}{\beta+1}}(\mathcal O_\alpha;\mathbb R^3),
	\]
	with $\varrho_{0,\delta}\ge0$ and both fields supported in
	$\overline{\mathcal O_{\zeta_0}}$, such that, as $\delta\to0$,
	\begin{align}& \label{eq:delta_initial_density}
		\varrho_{0,\delta}\to\widetilde\varrho_0
		\;\;\text{in }L^\gamma(\mathcal O_\alpha),\;\;  
		\mathbf m_{0,\delta} \to\widetilde{\mathbf m}_0
		\;\;\text{in }L^{\frac{2\gamma}{\gamma+1}}(\mathcal O_\alpha),\;\;  
		\delta\int_{\mathcal O_\alpha}\varrho_{0,\delta}^\beta\,\dd\bm x
		\to0,  
	\end{align}
	and
	\begin{equation}
		\label{eq:delta_initial_kinetic}
		\limsup_{\delta\to0}
		\int_{\mathcal O_\alpha}
		\frac{|\mathbf m_{0,\delta}|^2}{\varrho_{0,\delta}}\,\dd\bm x
		\le
		\int_{\mathcal O_{\zeta_0}}
		\frac{|\mathbf m_0|^2}{\varrho_0}\,\dd\bm x,
	\end{equation}
	where the quotient is zero on the vacuum set. Set
	$\mathbf v_{0,\delta}:=\mathbf m_{0,\delta}/\varrho_{0,\delta}$ on
	$\{\varrho_{0,\delta}>0\}$ and zero elsewhere. Such a family is obtained by
	truncating the zero extensions at levels $M_\delta\uparrow\infty$ with
	$M_\delta\le\frac12\delta^{-1/\beta}$, followed by a standard diagonal
	approximation of the momentum; then
	$\delta M_\delta^{\beta-\gamma}\to0$ gives
	\eqref{eq:delta_initial_density}.
	
	Second, for each fixed $\delta>0$ and $\epsilon>0$ choose smooth data
	\begin{equation*}		\varrho_{0,\delta,\epsilon}\in
		C^{2+\nu}(\overline{\mathcal O}_\alpha),
		\qquad
		\mathbf v_{0,\delta,\epsilon}\in
		C^{2+\nu}(\overline{\mathcal O}_\alpha;\mathbb R^3),
	\end{equation*}
	with
	$
	0<\epsilon\le\varrho_{0,\delta,\epsilon}\le\delta^{-1/\beta}.
	$
	Writing
	$\mathbf m_{0,\delta,\epsilon}
	:=\varrho_{0,\delta,\epsilon}\mathbf v_{0,\delta,\epsilon}$,  for
	fixed $\delta$,
	\begin{align*}		& \varrho_{0,\delta,\epsilon}\to\varrho_{0,\delta}
		\;\;\text{in }L^\beta(\mathcal O_\alpha),  
		\;\;	\mathbf m_{0,\delta,\epsilon} \to\mathbf m_{0,\delta}\;\;
		\text{in }L^{\frac{2\beta}{\beta+1}}(\mathcal O_\alpha),  
	\end{align*}
	and
	\begin{align*}		& \int_{\mathcal O_\alpha}
		\left[
		\frac12\frac{|\mathbf m_{0,\delta,\epsilon}|^2}{\varrho_{0,\delta,\epsilon}}
		+\frac{a}{\gamma-1}\varrho_{0,\delta,\epsilon}^\gamma
		+\frac{\delta}{\beta-1}\varrho_{0,\delta,\epsilon}^\beta
		\right]\dd\bm x
		\notag \\
		& \qquad\to
		\int_{\mathcal O_\alpha}
		\left[
		\frac12\frac{|\mathbf m_{0,\delta}|^2}{\varrho_{0,\delta}}
		+\frac{a}{\gamma-1}\varrho_{0,\delta}^\gamma
		+\frac{\delta}{\beta-1}\varrho_{0,\delta}^\beta
		\right]\dd\bm x. 
	\end{align*}
	This can be achieved by positivity-preserving mollification and a diagonal
	choice of the smoothing scale. Hence the $\epsilon\to0$ limit has initial
	data $(\varrho_{0,\delta},\mathbf v_{0,\delta})$, while
	\eqref{eq:delta_initial_density}--\eqref{eq:delta_initial_kinetic} recover the
	original data at the final $\delta\to0$ limit.
	
	We also prescribe
	\[
	\zeta_N(0)=\zeta_0,\qquad w_N(0)=w_0,
	\]
	with $\zeta_0\in H^2(\Gamma)$ and $w_0\in L^2(\Gamma)$. The preceding
	construction yields the uniform approximation-level initial-energy bound used
	below.
	
	\subsection{Splitting scheme}
	
	We employ an operator splitting scheme on a fixed time interval $[0,T]$.
	For
	$
	\Delta t=\frac{T}{N},
	t_j=j\Delta t,
	$
	we solve successively a structure subproblem and a fluid subproblem on each
	interval
	$
	[t_j,t_{j+1}],
	j=0,\ldots,N-1.
	$

	Let
	$
	\{\psi_i\}_{i=1}^\infty
	$
	be an orthonormal basis of $L^2(\mathcal{O}_\alpha;\mathbb{R}^3)$ which
	is also an orthogonal basis of
	$H_0^{s_f}(\mathcal{O}_\alpha;\mathbb{R}^3)$ for some fixed
	$s_f>5/2$, and let
	$
	\{\xi_i\}_{i=1}^\infty
	$
	be an orthonormal basis of $L^2(\Gamma)$ which is also an orthogonal basis
	of $H^2(\Gamma)$. We define
	\[
	X_m^f
	:=
	\operatorname{span}
	\{\psi_1,\ldots,\psi_m\},\;\;\hbox{	and}\;\;
	X_m^{st}
	:=
	\operatorname{span}
	\{\xi_1,\ldots,\xi_m\},
	\]
	with corresponding $L^2$-orthogonal projections
	$P_m^f$ and $P_m^{st}$.

	\medskip
	\noindent
	{\bf The structure subproblem.}
	
	Let
	$
	w_N=\partial_t\zeta_N.
	$
	For a graph $\xi$ we write
	\[
	\mathbf n^{\xi,\phi}
	:=
	\frac{\operatorname{cof}(\nabla\phi)\mathbf n^\xi}
	{\left|\operatorname{cof}(\nabla\phi)\mathbf n^\xi\right|},
	\qquad
	J_\Gamma^\phi[\xi]
	:=
	\left|\operatorname{cof}(\nabla\phi)\mathbf n^\xi\right|,
	\]
	and use $(\cdot)_{\tau^{\xi,\phi}}$ for the corresponding tangential
	projection. Whenever the normal is used inside a tubular neighborhood, we use
	the vertical-fiber extension
	\[
	\mathbf n^{\xi,\phi}(\bm y,1+\xi(\bm y)+r)
	:=
	\mathbf n^{\xi,\phi}(\bm y,1+\xi(\bm y)),
	\qquad
	0<r<h_\delta,
	\]
	where
	$
	h_\delta:=\delta^{\frac12-\frac1\beta}.
	$
	The shell velocity
	$w_N\mathbf e_z$ is extended in the same way.  Before the stopping time
	and at fixed Galerkin dimension, these extensions depend locally
	Lipschitz-continuously on the graph in every spatial norm used below.
	
	For every $\psi\in X_m^{st}$, the structure approximation on
	$[t_j,t_{j+1}]$ is defined by
	\begin{align}
		& \label{spstr}
		\int_\Gamma w_N(t_{j+1})\psi\,\dd\bm y
		=\int_\Gamma w_N(t_j)\psi\,\dd\bm y
		-\int_{t_j}^{t_{j+1}}\!\!\int_\Gamma
		\Bigl(\Delta_\Gamma^\chi\zeta_N\,\Delta_\Gamma^\chi\psi
		+\nu_s\nabla_\Gamma^\chi w_N\cdot\nabla_\Gamma^\chi\psi\Bigr)\,\dd\bm y\dd s \\
		& \quad-\frac1\delta\int_{t_j}^{t_{j+1}}\!\!\int_{T^\delta_{\mathcal T_{\Delta t}\zeta_N^*}}
		((w_N\mathbf e_z-\mathcal T_{\Delta t}\mathbf v_N)\cdot\mathbf n^{\mathcal T_{\Delta t}\zeta_N^*,\phi})
		(\psi\mathbf e_z\cdot\mathbf n^{\mathcal T_{\Delta t}\zeta_N^*,\phi})\,\dd\bm x\dd s \notag \\
		& \quad-\iota\int_{t_j}^{t_{j+1}}\!\!\int_{\Gamma_{\mathcal T_{\Delta t}\zeta_N^*}}
		J_\Gamma^\phi[\mathcal T_{\Delta t}\zeta_N^*]
		\left(w_N\mathbf e_z\circ\Phi_{\mathcal T_{\Delta t}\zeta_N^*}^{-1}-\mathcal T_{\Delta t}\mathbf v_N\right)_{\tau^{\mathcal T_{\Delta t}\zeta_N^*,\phi}}\notag \\
		& \qquad\cdot\left(\psi\mathbf e_z\circ\Phi_{\mathcal T_{\Delta t}\zeta_N^*}^{-1}\right)_{\tau^{\mathcal T_{\Delta t}\zeta_N^*,\phi}}\,\dd S\dd s. \notag 
	\end{align}
	The last term is essential: it is the structure-side work of the Navier
	friction force. Together with the corresponding term in the fluid
	subproblem it yields dissipation of the tangential relative velocity rather
	than an artificial no-slip constraint.
	
	The backward shift operator is
	\[
	\mathcal T_{\Delta t}f(t)
	:=
	\begin{cases}
		f(t-\Delta t),& t\ge\Delta t,\\
		f(0),&0\le t<\Delta t.
	\end{cases}
	\]
	
	For $\delta>0$, define the exterior tubular neighborhood
	\[
	T^\delta_\xi
	:=
	\left\{
	(x,y,z)\in
	\mathcal O_\alpha\setminus\mathcal O_\xi:
	0<z-1-\xi(x,y)
	<
	\delta^{\frac12-\frac1\beta}
	\right\}.
	\]
	Only the normal mismatch is penalized; no term of the form
	$\delta^{-1}|\mathbf v_N-w_N\mathbf e_z|^2$ is introduced.
	
	The stopped displacement is
	\[
	\zeta_N^*(t)
	:=
	\zeta_N(t\wedge\tau_N^\zeta),
	\]
	where
	\begin{align*}		& \tau_N^\zeta
		:=
		T\wedge
		\inf
		\Bigg\{
		t>0:
		\inf_\Gamma(1+\zeta_N(t))\le\alpha
		\;\;\text{or}\quad
		\|\zeta_N(t)\|_{H^s(\Gamma)}
		\ge\frac1\alpha
		\Bigg\}.
		\nonumber 
	\end{align*}

	\medskip
	\noindent
	{\bf The fluid subproblem.}
	
	On each interval $[t_j,t_{j+1}]$, the fluid unknowns are
	$	\varrho_N,
	\mathbf{v}_N.$
	We seek
	\[
	\varrho_N
	\in
	L^2
	\Bigl(
	\Omega;
	C(
	[t_j,t_{j+1}];
	C^{2+\nu}
	(\overline{\mathcal{O}}_\alpha)
	)
	\Bigr),\;\;
	\mathbf{v}_N
	\in
	L^2
	\Bigl(
	\Omega;
	C(
	[t_j,t_{j+1}];
	X_m^f
	)
	\Bigr).
	\]
	Define
	\[
	\Delta^\phi
	:=
	\Div^\phi\nabla^\phi.
	\]
	The approximate continuity equation is
	\begin{align}
		& \label{spcon}
		\int_{\mathcal{O}_\alpha}
		\varrho_N(t)\varphi
		\,\dd\bm{x}
		=
		\int_{\mathcal{O}_\alpha}
		\varrho_N(t_j)\varphi
		\,\dd\bm{x}
		+
		\int_{t_j}^{t}
		\int_{\mathcal{O}_\alpha}
		\varrho_N\mathbf{v}_N
		\cdot
		\nabla^\phi\varphi
		\,\dd\bm{x}\dd s
		\\
		& \quad
		-
		\epsilon
		\int_{t_j}^{t}
		\int_{\mathcal{O}_\alpha}
		\nabla^\phi\varrho_N
		\cdot
		\nabla^\phi\varphi
		\,\dd\bm{x}\dd s \notag 
	\end{align}
	for every
	$\varphi\in C^\infty(\overline{\mathcal{O}}_\alpha)$.
	
	Equivalently,
	\[
	\partial_t\varrho_N
	+
	\Div^\phi
	(\varrho_N\mathbf{v}_N)
	=
	\epsilon\Delta^\phi\varrho_N
	\qquad
	\text{in }\mathcal{O}_\alpha,
	\]
	with the transformed homogeneous Neumann condition
	\[
	\nabla^\phi\varrho_N
	\cdot
	\mathbf{n}^\phi
	=
	0
	\qquad
	\text{on }\partial\mathcal{O}_\alpha.
	\]

	For fixed
	$
	\beta>\max\{6,\gamma\},
	\delta>0,
	\epsilon>0,
	l>0,
	$
	the momentum equation is imposed in the following weak form. For every
	$\mathbf{q}\in X_m^f$.  To shorten the interface terms in this display, set
	$
	\mathbf n_N^*:=\mathbf n^{\zeta_N^*,\phi},
	\tau_N^*:=\tau^{\zeta_N^*,\phi},
	J_N^*:=J_\Gamma^\phi[\zeta_N^*].
	$ Then
	\begin{align}
		& \label{spmom}
		\int_{\mathcal O_\alpha}(\epsilon+[\varrho_N]_l)\mathbf v_N(t_{j+1})\cdot\mathbf q\,\dd\bm x\\
		& \quad=\int_{\mathcal O_\alpha}(\epsilon+[\varrho_N]_l)\mathbf v_N(t_j)\cdot\mathbf q\,\dd\bm x \notag \\
		& \quad+\int_{t_j}^{t_{j+1}}\!\!\int_{\mathcal O_\alpha}
		\Bigl([\varrho_N]_l\mathbf v_N\otimes\mathbf v_N:\nabla^\phi\mathbf q
		+(a\varrho_N^\gamma+\delta\varrho_N^\beta)\Div^\phi\mathbf q\notag \\
		& \qquad\qquad-\mathbb S_\kappa^{\zeta_N^*}(\nabla^\phi\mathbf v_N):\nabla^\phi\mathbf q
		-\frac\epsilon2[\nabla^\phi\varrho_N]_l\cdot\nabla^\phi\mathbf v_N\cdot\mathbf q\Bigr)\,\dd\bm x\dd s \notag \\
		& \quad-\frac1\delta\int_{t_j}^{t_{j+1}}\!\!\int_{T^\delta_{\zeta_N^*}}
		((\mathbf v_N-w_N\mathbf e_z)\cdot\mathbf n_N^*)
		(\mathbf q\cdot\mathbf n_N^*)\,\dd\bm x\dd s \notag \\
		& \quad-\iota\int_{t_j}^{t_{j+1}}\!\!\int_{\Gamma_{\zeta_N^*}}J_N^*
		\left(\mathbf v_N-w_N\mathbf e_z\circ\Phi_{\zeta_N^*}^{-1}\right)_{\tau_N^*}\cdot\mathbf q_{\tau_N^*}\,\dd S\dd s. \notag 
	\end{align}
	Here
	$
	[\varrho]_l=\kappa_l*\varrho,
	$
	where $\kappa_l$ is a standard spatial mollifier of scale $l>0$.
	The notation $J_\Gamma^\phi[\xi]$ and
	$\mathbf n^{\xi,\phi}$ is the one introduced in the structure step above.

	\medskip
	\noindent
	{\bf Bending-energy identity.}
	
	For the approximate transformed shell variables define the corresponding
	physical-coordinate fields
	\[
	\widehat\eta_N(t,\bm y)
	:=
	\zeta_N(t,\chi^{-1}(t,\bm y)),
	\qquad
	\widehat r_N(t,\bm y)
	:=
	w_N(t,\chi^{-1}(t,\bm y)).
	\]
	By the Stratonovich chain rule,
	\begin{equation}
		\label{eq:approx_eta_physical}
		\dd\widehat\eta_N
		=
		\widehat r_N\,\dd t
		+
		\sum_{k=1}^K
		A_k\widehat\eta_N\circ\dd W_k,
		\qquad
		A_k:=\mathcal Y_k\cdot\nabla_\Gamma.
	\end{equation}
	Since $\chi$ is area-preserving,
	\[
	\|\Delta_\Gamma\widehat\eta_N\|_{L^2(\Gamma)}
	=
	\|\Delta_\Gamma^\chi\zeta_N\|_{L^2(\Gamma)},
	\qquad
	\langle
	\Delta_\Gamma\widehat\eta_N,
	\Delta_\Gamma\widehat r_N
	\rangle
	=
	\langle
	\Delta_\Gamma^\chi\zeta_N,
	\Delta_\Gamma^\chi w_N
	\rangle.
	\]
	
	\begin{lemma}[Approximate shell bending-energy formula]
		\label{lem:approx_bending_energy}
		Let $\mathcal G_k$ and $\mathcal R_k$ be defined by
		\eqref{eq:Gk_commutator}--\eqref{eq:Rk_commutator}. Then
		\begin{align*}			& \frac12
			\|\Delta_\Gamma^\chi\zeta_N(t)\|_{L^2(\Gamma)}^2
			=
			\frac12
			\|\Delta_\Gamma\zeta_0\|_{L^2(\Gamma)}^2
			+
			\int_0^t
			\left\langle
			\Delta_\Gamma^\chi\zeta_N,
			\Delta_\Gamma^\chi w_N
			\right\rangle
			\,\dd s
			\notag \\
			& \quad
			+
			M_N^{\rm sh}(t)
			+
			\frac12
			\sum_{k=1}^K
			\int_0^t
			\mathcal R_k(\widehat\eta_N(s))
			\,\dd s, 
		\end{align*}
		where
		\begin{equation*}			M_N^{\rm sh}(t)
			:=
			\sum_{k=1}^K
			\int_0^t
			\mathcal G_k(\widehat\eta_N(s))
			\,\dd W_k(s).
		\end{equation*}
		Moreover,
		\begin{align*}			& |\mathcal G_k(\widehat\eta_N)|
			+
			|\mathcal R_k(\widehat\eta_N)|
			\le
			c_k
			\|\Delta_\Gamma^\chi\zeta_N\|_{L^2(\Gamma)}^2, \\
			& \dd\langle M_N^{\rm sh}\rangle_t
			\le
			C
			\|\Delta_\Gamma^\chi\zeta_N\|_{L^2(\Gamma)}^4\,\dd t. 
		\end{align*}
	\end{lemma}
	
	\begin{proof}
		Equation \eqref{eq:approx_eta_physical} and It\^o's formula give
		\[
		\dd\frac12\|\Delta_\Gamma\widehat\eta_N\|_2^2
		=
		\left\langle
		\Delta_\Gamma\widehat\eta_N,
		\Delta_\Gamma\widehat r_N
		\right\rangle\dd t
		+
		\sum_{k=1}^K
		\mathcal G_k(\widehat\eta_N)\,\dd W_k
		+
		\frac12
		\sum_{k=1}^K
		\mathcal R_k(\widehat\eta_N)\,\dd t.
		\]
		The identities and bounds follow from the area-preserving change of
		variables and Lemma~\ref{lem:shell_energy_commutator}.
	\end{proof}

	\noindent	\textbf{Approximate energy and splitting defect.} Define
	\begin{align*}		& \mathcal E_N(t)
		:={}
		\frac12
		\int_{\mathcal O_\alpha}
		\bigl(\epsilon+[\varrho_N]_l\bigr)
		|\mathbf v_N|^2\,\dd\bm x
		+
		\frac12\|w_N(t)\|_{L^2(\Gamma)}^2
		+
		\frac12
		\|\Delta_\Gamma^\chi\zeta_N(t)\|_{L^2(\Gamma)}^2
		\notag \\
		& +
		\int_{\mathcal O_\alpha}
		\frac{a}{\gamma-1}\varrho_N^\gamma\,\dd\bm x
		+
		\int_{\mathcal O_\alpha}
		\frac{\delta}{\beta-1}\varrho_N^\beta\,\dd\bm x . 
	\end{align*}
	Let $\mathcal D_N(t)$ denote
	\begin{align}
		& \label{eq:approx_dissipation_def}
		\mathcal D_N(t)
		:={}
		\int_0^t\int_{\mathcal O_\alpha}
		\mathbb S_\kappa^{\zeta_N^*}
		(\nabla^\phi\mathbf v_N):
		\nabla^\phi\mathbf v_N
		\,\dd\bm x\dd s
		+
		\nu_s
		\int_0^t
		\|\nabla_\Gamma^\chi w_N\|_{L^2(\Gamma)}^2\,\dd s
		\\
		& +
		\epsilon
		\int_0^t\int_{\mathcal O_\alpha}
		\left(
		a\gamma\varrho_N^{\gamma-2}
		+
		\delta\beta\varrho_N^{\beta-2}
		\right)
		|\nabla^\phi\varrho_N|^2
		\,\dd\bm x\dd s
		\notag \\
		& +
		\frac1{2\delta}
		\int_0^t
		\int_{T^\delta_{\zeta_N^*}}
		\left|
		(
		\mathbf v_N-w_N\mathbf e_z
		)
		\cdot
		\mathbf n^{\zeta_N^*,\phi}
		\right|^2
		\,\dd\bm x\dd s
		\notag \\
		& +
		\frac1{2\delta}
		\int_0^t
		\int_{T^\delta_{\mathcal T_{\Delta t}\zeta_N^*}}
		\left|
		(
		\mathcal T_{\Delta t}\mathbf v_N-w_N\mathbf e_z
		)
		\cdot
		\mathbf n^{\mathcal T_{\Delta t}\zeta_N^*,\phi}
		\right|^2
		\,\dd\bm x\dd s
		\notag \\
		& +
		\frac{\iota}{2}
		\int_0^t
		\int_{\Gamma_{\zeta_N^*}}
		J_\Gamma^\phi[\zeta_N^*]
		\left|
		\left(
		\mathbf v_N
		-
		w_N\mathbf e_z\circ\Phi_{\zeta_N^*}^{-1}
		\right)_{\tau^{\zeta_N^*,\phi}}
		\right|^2
		\,\dd S\dd s
		\notag \\
		& +
		\frac{\iota}{2}
		\int_0^t
		\int_{\Gamma_{\mathcal T_{\Delta t}\zeta_N^*}}
		J_\Gamma^\phi[\mathcal T_{\Delta t}\zeta_N^*]
		\left|
		\left(
		\mathcal T_{\Delta t}\mathbf v_N
		-
		w_N\mathbf e_z\circ
		\Phi_{\mathcal T_{\Delta t}\zeta_N^*}^{-1}
		\right)_{\tau^{\mathcal T_{\Delta t}\zeta_N^*,\phi}}
		\right|^2
		\,\dd S\dd s . \notag 
	\end{align}
	
	For later use we make the splitting defect explicit.  For a graph
	$\xi$, a flow $\phi$, and an ambient vector field $\mathbf a$, set
	\begin{align*}		& \mathcal Q_{n,\delta}[\xi,\phi;\mathbf a]
		:=
		\int_{T^\delta_\xi}
		\left|
		\mathbf a\cdot\mathbf n^{\xi,\phi}
		\right|^2
		\,\dd\bm x, \\
		& \mathcal Q_{\tau}[\xi,\phi;\mathbf a]
		:=
		\int_{\Gamma_\xi}
		J_\Gamma^\phi[\xi]
		\left|
		\mathbf a_{\tau^{\xi,\phi}}
		\right|^2
		\,\dd S. 
	\end{align*}
	For the shell field in
	$\mathcal Q_\tau$ we use
	$\mathbf a=w_N\mathbf e_z\circ\Phi_\xi^{-1}$; in
	$\mathcal Q_{n,\delta}$ it denotes its vertical extension.
	
	The splitting remainder is
	\begin{equation}
		\label{eq:splitting_remainder_decomposition}
		\mathfrak R_{N,\delta}^{\rm split}
		=
		\mathfrak R_{N,\delta}^{n}
		+
		\mathfrak R_{N,\delta}^{\tau},
	\end{equation}
	where
	\begin{align}
		& \mathfrak R_{N,\delta}^{n}(t)
		:={}
		\frac1{2\delta}
		\int_0^t
		\Big[
		\mathcal Q_{n,\delta}
		[
		\mathcal T_{\Delta t}\zeta_N^*,
		\phi;
		\mathcal T_{\Delta t}\mathbf v_N
		]
		-
		\mathcal Q_{n,\delta}
		[
		\zeta_N^*,
		\phi;
		\mathbf v_N
		]
		\notag \\
		& \hspace{2.6cm}
		+
		\mathcal Q_{n,\delta}
		[
		\zeta_N^*,
		\phi;
		w_N\mathbf e_z
		]
		-
		\mathcal Q_{n,\delta}
		[
		\mathcal T_{\Delta t}\zeta_N^*,
		\phi;
		w_N\mathbf e_z
		]
		\Big]
		\,\dd s, \notag \\
		& \mathfrak R_{N,\delta}^{\tau}(t)
		:={}
		\frac{\iota}{2}
		\int_0^t
		\Big[
		\mathcal Q_{\tau}
		[
		\mathcal T_{\Delta t}\zeta_N^*,
		\phi;
		\mathcal T_{\Delta t}\mathbf v_N
		]
		-
		\mathcal Q_{\tau}
		[
		\zeta_N^*,
		\phi;
		\mathbf v_N
		]
		\notag \\
		& \hspace{2.6cm}
		+
		\mathcal Q_{\tau}
		[
		\zeta_N^*,
		\phi;
		w_N\mathbf e_z
		]
		-
		\mathcal Q_{\tau}
		[
		\mathcal T_{\Delta t}\zeta_N^*,
		\phi;
		w_N\mathbf e_z
		]
		\Big]
		\,\dd s.
		\label{eq:splitting_remainder_tangent} 
	\end{align}
	
	\begin{proposition}[Approximate existence and energy estimate]
		\label{thm:approx_existence}
		Fix $N,m,\epsilon,l,\delta>0$ and $\beta>\max\{6,\gamma\}$, and let the
		initial data satisfy the assumptions above. Then \eqref{spstr}--\eqref{spmom}
		has an $(\mathcal F_t)$-adapted solution
		$(\varrho_N,\mathbf v_N,\zeta_N,w_N)$ on $[0,T]$ with
		\[
		\varrho_N\in C([0,T];C^{2+\nu}(\overline{\mathcal O}_\alpha)),\quad
		\varrho_N>0,\quad
		\mathbf v_N\in C([0,T];X_m^f),
		\]
		\[
		\zeta_N\in C^1([0,T];X_m^{st}),\qquad
		w_N=\partial_t\zeta_N.
		\]
		With $\mathcal E_N$, $\mathcal D_N$, and
		$\mathfrak R_{N,\delta}^{\rm split}$ defined above, for every $t\in[0,T]$,
		\begin{align}
			& \label{energy}
			\mathcal E_N(t)+\mathcal D_N(t)
			\le{}\mathcal E_N(0)+M_N^{\rm sh}(t)
			+\frac12\sum_{k=1}^K\int_0^t
			\mathcal R_k(\widehat\eta_N(s))\,\dd s
			+\mathfrak R_{N,\delta}^{\rm split}(t). 
		\end{align}
		Moreover, for every $p\ge1$,
		\begin{equation}
			\label{eq:approx_energy_moment}
			\mathbb E\left[\sup_{0\le t\le T}\mathcal E_N(t)^p
			+\mathcal D_N(T)^p\right]
			\le C_{p,T,\delta}\bigl(1+\mathcal E_N(0)^p\bigr),
		\end{equation}
		with a constant independent of $N$ and, at the energy level, of $m$.
		The splitting remainder converges to zero as $N\to\infty$ by
		Lemma~\ref{lem:N_splitting_remainder}.
	\end{proposition}

	\begin{proof}
		We first construct the approximation interval by interval. Assume that the
		solution is known up to $t_j$. For a fixed realization of the stochastic
		flow, the coefficients of $\nabla_\Gamma^\chi$,
		$\Delta_\Gamma^\chi$, $\nabla^\phi$, and $\Div^\phi$ are continuous in
		time and smooth in space. Hence the finite-dimensional structure equation
		\eqref{spstr} is a Carath\'eodory ODE on $X_m^{st}$. Its coefficients at
		time $t$ depend only on the stochastic flow up to time $t$ and on the
		already constructed lagged fluid variables; consequently the solution is
		adapted.
		
		Once $(\zeta_N,w_N)$ is known on $[t_j,t_{j+1}]$, the continuity equation
		\eqref{spcon} is uniformly parabolic for fixed $\epsilon>0$. Positivity of
		$\varrho_N$ follows from the maximum principle. Given
		$\mathbf v\in C([t_j,t_{j+1}];X_m^f)$, solve \eqref{spcon} for
		$\varrho$, and then solve the finite-dimensional momentum equation
		\eqref{spmom} for a new velocity $\widetilde{\mathbf v}$. The resulting
		map
		\[
		\mathfrak T:\mathbf v\mapsto\widetilde{\mathbf v}.
		\]
		On every ball of radius $R$ in $C([t_j,t_j+h];X_m^f)$, the parabolic
		stability estimate for \eqref{spcon} and equivalence of norms on $X_m^f$
		give, for some $\theta>0$,
		\[
		\|\mathfrak T(\mathbf v_1)-\mathfrak T(\mathbf v_2)\|_{C_tX_m^f}
		\le C_{m,\epsilon,l,\delta,R}h^\theta
		\|\mathbf v_1-\mathbf v_2\|_{C_tX_m^f}.
		\]
		The constant is finite on bounded-flow localization events and uses only
		the spatial smoothness and time continuity of the flow coefficients. Choose
		$h$ so that the coefficient is below one. The local energy identity derived
		below, together with the density maximum principle, controls the
		finite-dimensional continuation norm; hence the solution extends
		successively to $[t_j,t_{j+1}]$ and then to $[0,T]$.

		We next derive the energy estimate. Testing the continuity equation with
		\[
		\frac{a\gamma}{\gamma-1}\varrho_N^{\gamma-1}
		+
		\frac{\delta\beta}{\beta-1}\varrho_N^{\beta-1},
		\]
		the momentum equation with $\mathbf v_N$, and the structure equation with
		$w_N$, and summing over the time steps yields the deterministic part of
		\eqref{energy}, including the viscous, artificial-diffusion, shell-damping,
		penalty, and Navier-friction dissipations.
		
		The only nonstandard point is the bending-energy contribution, since the
		coefficients induced by $\chi$ are only H\"older continuous in time.
		Lemma~\ref{lem:approx_bending_energy} gives the corresponding shell martingale
		$M_N^{\rm sh}$ and It\^o correction.
		Finally,
		Lemma~\ref{lem:shell_energy_commutator}, the
		Burkholder--Davis--Gundy inequality (BDG), and Gronwall's lemma imply
		\eqref{eq:approx_energy_moment}. 
	\end{proof}

	\medskip

\section{Time-discretization limit \texorpdfstring{$N\to\infty$}{N to infinity}}
	\label{s4}
	
	Let $\Delta t=T/N\to0$ with $m,\epsilon,l,\delta$ fixed.  The stochastic flow is
	not discretized.  At this level the density equation remains parabolic and the
	velocity and shell spaces are finite dimensional, so the only additional issue
	is the splitting error between the current and lagged geometries.  The
	translation estimates below show that this defect vanishes as $N\to\infty$.
	The shell martingale is kept in the joint law in order to preserve its relation
	with the common Wiener process after the representation step.
	
	\subsection{Uniform estimates independent of \texorpdfstring{$N$}{N}}
	
	The approximate stochastic energy estimate
	\eqref{eq:approx_energy_moment} immediately yields the following bounds.

	\begin{lemma}[Uniform estimates]
		
		For every $p\ge1$ there exists
		$
		C=C(p,T,m,\epsilon,l,\delta,\alpha),
		$
		independent of $N$, such that
		\begin{equation}
			\label{eq:N_energy_moment}
			\sup_N
			\mathbb E\left[
			\sup_{0\le t\le T}\mathcal E_N(t)^p
			+
			\mathcal D_N(T)^p
			\right]
			\le C.
		\end{equation}
		Consequently,
		\begin{align}	& \label{eq:N_basic_bounds}
			\sup_N\mathbb E\Big[
			\|\varrho_N\|_{L^\infty(0,T;L^\beta(\mathcal O_\alpha))}^{\beta p}
			+
			\|\mathbf v_N\|_{L^2(0,T;H^1(\mathcal O_\alpha))}^{2p}
			\\
			& +
			\|\zeta_N\|_{L^\infty(0,T;H^2(\Gamma))}^{2p}
			+
			\|w_N\|_{L^\infty(0,T;L^2(\Gamma))}^{2p}
			+
			\|w_N\|_{L^2(0,T;H^1(\Gamma))}^{2p}
			\Big]
			\le C, \notag 
		\end{align}
		and
		\begin{equation*}			\sup_N
			\mathbb E\left[
			\left(
			\epsilon
			\int_0^T\!\!\int_{\mathcal O_\alpha}
			\left(
			a\gamma\varrho_N^{\gamma-2}
			+
			\delta\beta\varrho_N^{\beta-2}
			\right)
			|\nabla^\phi\varrho_N|^2
			\,\dd\bm x\dd t
			\right)^p
			\right]
			\le C.
		\end{equation*}
		
		The normal-penalty and Navier-friction terms appearing in
		\eqref{eq:approx_dissipation_def} satisfy the corresponding uniform $L^p(\Omega)$ bounds.
		Moreover,
		\begin{equation}
			\label{eq:N_stochastic_bounds}
			\sup_N
			\mathbb E\left[
			\sup_{0\le t\le T}|M_N^{\rm sh}(t)|^p
			+
			\left(
			\sum_{k=1}^K
			\int_0^T
			|\mathcal R_k(\widehat\eta_N)|\,\dd t
			\right)^p
			\right]
			\le C.
		\end{equation}
	\end{lemma}

	\begin{proof}
		The estimate \eqref{eq:N_energy_moment} is precisely the stochastic
		approximation-level estimate of
		Proposition~\ref{thm:approx_existence}. The remaining bounds follow from the
		definition of $\mathcal E_N$ and $\mathcal D_N$, the coercivity of
		$\mathbb S_\kappa^{\zeta_N^*}$, Korn's inequality on the fixed maximal
		domain, and the equivalence of the standard Sobolev norms with the pullback
		norms generated by $\phi$ and $\chi$.
		
		For the shell-noise terms, Lemma~\ref{lem:shell_energy_commutator} gives
		\[
		|\mathcal G_k(\widehat\eta_N)|
		+
		|\mathcal R_k(\widehat\eta_N)|
		\le
		C_k
		\|\widehat\eta_N\|_{H^2(\Gamma)}^2.
		\]
		Since $\chi$ is area-preserving and has uniformly controlled spatial
		derivatives,
		\[
		\|\widehat\eta_N\|_{H^2(\Gamma)}
		\le
		C(\phi)
		\|\zeta_N\|_{H^2(\Gamma)}.
		\]
		The BDG inequality and
		\eqref{eq:N_basic_bounds} therefore imply
		\eqref{eq:N_stochastic_bounds}.
	\end{proof}

	Because $m$ is fixed, the finite-dimensional fluid and structure equations
	yield uniform time regularity. In particular, for every $p\ge1$,
	\begin{align}
		& \sup_N
		\mathbb E
		\|\zeta_N\|_{W^{1,\infty}(0,T;H^2(\Gamma))}^p
		\le C_{p,m},
		\label{eq:N_zeta_strong_time_bound} \\
		& \sup_N
		\mathbb E
		\|\partial_t w_N\|_{L^2(0,T;X_m^{st})}^p
		\le C_{p,m,\delta},
		\label{eq:N_w_time_bound} \\
		& \sup_N
		\mathbb E
		\|\partial_t\mathbf v_N\|_{L^2(0,T;X_m^f)}^{2p}
		\le C_{p,m,\epsilon,l,\delta}. \notag 
	\end{align}
	Consequently,
	\begin{equation}
		\label{eq:N_time_translation_v_moment}
		\sup_N
		\mathbb E
		\left[
		\left(
		\int_0^T
		\|\mathcal T_h\mathbf v_N-\mathbf v_N\|_{X_m^f}^2
		\,\dd t
		\right)^p
		\right]
		\le
		C_{p,m,\epsilon,l,\delta}h^{2p}.
	\end{equation}
	Moreover, since stopping preserves the time modulus,
	\begin{equation}
		\label{eq:N_zeta_shift_rate}
		\sup_N
		\mathbb E
		\left[
		\sup_{0<h<T}
		h^{-p}
		\|\mathcal T_h\zeta_N^*-\zeta_N^*\|_{L^\infty(0,T;C^1(\Gamma))}^p
		\right]
		\le C_{p,m}.
	\end{equation}
	
	Fix $0<\vartheta<1/2$ below the time H\"older exponent of the stochastic
	flow and set
	\[
	\Lambda_\phi
	:=
	1+
	\|\phi\|_{C^\vartheta([0,T];C^2)}
	+
	\|\phi^{-1}\|_{C^\vartheta([0,T];C^2)}.
	\]
	Then $\Lambda_\phi$ has finite moments of all required orders and
	\begin{equation}
		\label{eq:N_flow_shift_rate}
		\|\mathcal T_h\phi-\phi\|_{L^\infty C^2}
		+
		\|\mathcal T_h\phi^{-1}-\phi^{-1}\|_{L^\infty C^2}
		\le
		\Lambda_\phi h^\vartheta .
	\end{equation}
	
	\begin{lemma}[Quantitative vanishing of the splitting remainder]
		\label{lem:N_splitting_remainder}
		Let $h=\Delta t=T/N$ and let $\mathfrak R_{N,\delta}^{\rm split}$ be
		given by \eqref{eq:splitting_remainder_decomposition}--%
		\eqref{eq:splitting_remainder_tangent}. For fixed
		$m,\epsilon,l,\delta>0$ and every $p\ge1$,
		\begin{equation*}			\mathbb E
			\left[
			\sup_{0\le t\le T}
			|\mathfrak R_{N,\delta}^{\rm split}(t)|^p
			\right]
			\le
			C_{p,m,\epsilon,l,\delta}h^{p\vartheta}.
		\end{equation*}
		In particular,
		\begin{equation*}			\mathbb E
			\left[
			\sup_{0\le t\le T}
			|\mathfrak R_{N,\delta}^{\rm split}(t)|^p
			\right]
			\to0
			\qquad\text{as }N\to\infty.
		\end{equation*}
	\end{lemma}
	
	\begin{proof}
		On the stopped set, the normal-penalty and tangential-friction quadratic
		forms depend locally Lipschitz-continuously on the graph, the stochastic
		flow, and the finite-dimensional fluid and shell variables. Therefore the
		decomposition
		\eqref{eq:splitting_remainder_decomposition}--%
		\eqref{eq:splitting_remainder_tangent},
		together with
		\eqref{eq:N_zeta_shift_rate},
		\eqref{eq:N_time_translation_v_moment},
		and \eqref{eq:N_flow_shift_rate}, gives
		\[
		\mathbb E
		\sup_{0\le t\le T}
		|\mathfrak R_{N,\delta}^{\rm split}(t)|^p
		\le
		C\bigl(h^p+h^{p\vartheta}\bigr)
		\le
		Ch^{p\vartheta}.
		\]
		The $h^\vartheta$ term is the contribution of the stochastic-flow time
		shift; no $O(h)$ estimate for the flow coefficients is used.
	\end{proof}

	\subsection{Compactness, representation, and martingale structure}
	
	To pass the shell martingale to the limit, it is useful to keep it as an
	explicit component of the joint law. Fix
	$
	0<\kappa_0<\kappa_1<\frac12.
	$
	We use the path space
	\[
	\mathcal X
	=
	\mathcal X_\varrho
	\times
	\mathcal X_{\mathbf v}
	\times
	\mathcal X_\zeta
	\times
	\mathcal X_{\zeta^*}
	\times
	\mathcal X_w
	\times
	\mathcal X_\phi
	\times
	\mathcal X_W
	\times
	\mathcal X_M,
	\]
	where
	\begin{align*}		& \mathcal X_\varrho
		:=
		C_w([0,T];L^\beta(\mathcal O_\alpha))
		\cap
		L^4((0,T)\times\mathcal O_\alpha),
		\;\;
		\mathcal X_{\mathbf v}
		:=
		L^2(0,T;X_m^f), \\
		& \mathcal X_\zeta
		:=
		C([0,T];X_m^{st}),
		\;\;
		\mathcal X_{\zeta^*}
		:=
		C([0,T];X_m^{st}),
		\;\;
		\mathcal X_w
		:=
		L^2(0,T;X_m^{st}), \\
		& \mathcal X_\phi
		:=
		C^{\kappa_0}
		([0,T];C^2(\overline{\mathcal O}_\alpha;
		\overline{\mathcal O}_\alpha)),
		\;\;
		\mathcal X_W
		:=
		C([0,T];\mathbb R^K),
		\;\;
		\mathcal X_M
		:=
		C([0,T];\mathbb R). 
	\end{align*}

	Let $\mu_N$ be the joint law of
	$
	\left(
	\varrho_N,
	\mathbf v_N,
	\zeta_N,
	\zeta_N^*,
	w_N,
	\phi,
	W,
	M_N^{\rm sh}
	\right)
	$
	on $\mathcal X$.

	\begin{lemma}[Tightness]
		
		The family $\{\mu_N\}_{N\in\mathbb N}$ is tight on $\mathcal X$.
	\end{lemma}
	
	\begin{proof}
		Since $m$ is fixed, the structure variables are compact in the
		finite-dimensional spaces. Indeed,
		\eqref{eq:N_zeta_strong_time_bound} yields tightness of
		$\zeta_N$ and $\zeta_N^*$ in $C([0,T];X_m^{st})$, while
		\eqref{eq:N_basic_bounds} and \eqref{eq:N_w_time_bound} yield
		tightness of $w_N$ in $L^2(0,T;X_m^{st})$.
		
		Likewise, \eqref{eq:N_basic_bounds} together with the time-translation
		estimate \eqref{eq:N_time_translation_v_moment} gives tightness of
		$\mathbf v_N$ in $L^2(0,T;X_m^f)$.
		
		At fixed $\epsilon>0$, the continuity equation is uniformly parabolic.
		The energy bounds therefore give tightness of $\varrho_N$ in
		\[
		C_w([0,T];L^\beta(\mathcal O_\alpha))
		\cap
		L^4((0,T)\times\mathcal O_\alpha),
		\]
		while the stronger fixed-$\epsilon$ parabolic compactness will be used after the representation.
		
		For the stochastic flow, choose
		$\kappa_0<\kappa_1<1/2$. Standard stochastic-flow estimates and the
		compact embedding
		\[
		C^{\kappa_1}([0,T];C^3)
		\subset
		C^{\kappa_0}([0,T];C^2)
		\]
		give tightness of $\phi$ and $\phi^{-1}$. The tightness of $W$ in
		$C([0,T];\mathbb R^K)$ is standard. Moreover,
		\eqref{eq:N_stochastic_bounds}, together with the corresponding increment
		estimate, yields tightness of $M_N^{\rm sh}$ in $C([0,T];\mathbb R)$.
		Combining the componentwise estimates proves the claim.
	\end{proof}

	We next identify the represented Wiener process and the limiting martingale structure.
		Since the density component contains a weak topology, we use the
	Jakubowski--Skorokhod representation theorem
	\cite{Jakubowski1997}.
	
	\begin{lemma}[Representation]
		\label{lem:skorokhod_N}
		There exist a probability space
		$
		(\widetilde\Omega,\widetilde{\mathcal F},
		\widetilde{\mathbb P}),
		$
		random variables
		$
		\widetilde{\mathcal U}_N
		=
		(
		\widetilde\varrho_N,
		\widetilde{\mathbf v}_N,
		\widetilde\zeta_N,
		\widetilde\zeta_N^*,
		\widetilde w_N,
		\widetilde\phi_N,
		\widetilde W_N,
		\widetilde M_N^{\rm sh}
		),
		$
		and
		$
		\mathcal U
		=
		(
		\varrho,
		\mathbf v,
		\zeta,
		\zeta^*,
		w,
		\phi,
		W,
		M^{\rm sh}
		)
		$
		such that
		$
		\mathcal L(\widetilde{\mathcal U}_N)=\mu_N
		$
		and
		\[
		\widetilde{\mathcal U}_N
		\to
		\mathcal U
		\qquad
		\widetilde{\mathbb P}\text{-a.s. in }\mathcal X.
		\]
	\end{lemma}
	
	\begin{proof}
		The representation theorem gives almost sure convergence in $\mathcal X$.
		The stronger convergences follow from the compactness estimates in the
		previous subsection. In particular, because $X_m^{st}$ and $X_m^f$ are
		finite-dimensional, convergence in the corresponding path spaces implies
		convergence in every spatial Sobolev norm on those spaces. The strong
		density convergence follows from the fixed-$\epsilon$ parabolic stability
		estimate.
	\end{proof}
	
	From now on we work on the new probability space and omit the tildes.
	Define
	\[
	\chi_N:=\phi_N|_\Gamma,
	\qquad
	\chi:=\phi|_\Gamma,\;\;
	\widehat\eta_N
	:=
	\zeta_N\circ\chi_N^{-1},
	\qquad
	\widehat\eta
	:=
	\zeta\circ\chi^{-1}.
	\]
	By Lemma~\ref{lem:skorokhod_N} and the finite-dimensional compactness,
	\begin{equation*}		\widehat\eta_N
		\to
		\widehat\eta
		\qquad
		\text{in }C([0,T];H^2(\Gamma)),
		\quad
		\widetilde{\mathbb P}\text{-a.s.}
	\end{equation*}
	Consequently,
	\begin{align}& \label{eq:N_G_conv}
		\mathcal G_k(\widehat\eta_N)
		\to
		\mathcal G_k(\widehat\eta)\;\;
		\text{and }\;\;
		\mathcal R_k(\widehat\eta_N)
		\to
		\mathcal R_k(\widehat\eta)\;\;
		\text{in }C([0,T]).  
	\end{align}
	
	The next lemma records the stochastic information that must be retained
	after the representation. Equal laws alone are not used as a substitute
	for adaptedness; instead we pass the martingale characterizations.
	
	\begin{lemma}[Identification of the limiting shell martingale]
		\label{lem:N_martingale_identification}
		Let $\widetilde{\mathfrak F}_t$ be the
		usual $\widetilde{\mathbb P}$-augmentation of the natural filtration generated by
		\[
		(\varrho,\mathbf v,\zeta,\zeta^*,w,\phi,W,M^{\rm sh})
		\]
		up to time $t$. Then $W$ is a $K$-dimensional
		$(\widetilde{\mathfrak F}_t)$-Wiener process, $\phi$ is the stochastic flow
		generated by $W$, and
		\begin{equation}
			\label{eq:N_limit_shell_martingale}
			M^{\rm sh}(t)
			=
			\sum_{k=1}^K
			\int_0^t
			\mathcal G_k(\widehat\eta(s))\,\dd W_k(s).
		\end{equation}
	\end{lemma}
	
	\begin{proof}
		For every $N$, the copied processes satisfy the same martingale
		characterizations as the original approximation. In particular, for
		$0\le s<t\le T$ and every bounded continuous functional $F$ of the paths
		up to time $s$,
		\[
		\mathbb E
		\left[
		F\bigl(W_N(t)-W_N(s)\bigr)
		\right]
		=0,
		\]
		and the analogous identities hold for the quadratic covariations. Passing
		to the limit using Lemma~\ref{lem:skorokhod_N}, uniform integrability, and the same
		identities for $|W_N|^2-t$ proves that $W$ remains Wiener with respect to
		the canonical limiting filtration. Since the stochastic flow equation has
		a pathwise unique strong solution for the smooth transport fields, the
		flow component is the solution generated by this $W$.
		
		Similarly, $M_N^{\rm sh}$ is a continuous martingale with
		\begin{align*}			& \langle M_N^{\rm sh}\rangle_t
			=
			\sum_{k=1}^K
			\int_0^t
			|\mathcal G_k(\widehat\eta_N)|^2\,\dd s,\;\;
			\langle M_N^{\rm sh},W_{N,j}\rangle_t
			=
			\int_0^t
			\mathcal G_j(\widehat\eta_N)\,\dd s. 
		\end{align*}
		The convergence in Lemma~\ref{lem:skorokhod_N}, together with
		\eqref{eq:N_G_conv} and \eqref{eq:N_stochastic_bounds}, allows passage to the corresponding
		martingale and covariation identities. The martingale representation
		characterization then gives \eqref{eq:N_limit_shell_martingale}.
	\end{proof}

	\subsection{Passage to the limit as \texorpdfstring{$N\to\infty$}{N to infinity}}
	
	The representation provides almost-sure convergence in precisely the
	topologies selected above. The bulk nonlinearities are then harmless at
	fixed $m,\epsilon,l,\delta$: the parabolic density is strongly compact and
	the velocity is finite dimensional. The interface terms require one
	additional observation. Their current and lagged versions differ only
	through time translations of the solution and through the
	$C_t^\vartheta$ stochastic-flow coefficients; the quantitative splitting
	estimate shows that these differences vanish. Consequently the two
	half-sized penalty and friction contributions recombine into the single
	coupled interface forms of the continuous-time approximation.

	For the $N\to\infty$ limit define the stopping time by
	\begin{equation*}		\tau^\zeta
		:=
		T\wedge
		\inf
		\left\{
		t>0:
		\inf_\Gamma(1+\zeta(t))\le\alpha
		\ \text{or}\
		\|\zeta(t)\|_{H^s(\Gamma)}\ge\frac1\alpha
		\right\}.
	\end{equation*}
	and set
	$
	\widehat\eta:=\zeta\circ\chi^{-1}
	$
	and
	\begin{equation*}		M^{\rm sh}(t)
		:=
		\sum_{k=1}^K
		\int_0^t
		\mathcal G_k(\widehat\eta(s))\,\dd W_k(s).
	\end{equation*}
	Define
	\begin{align}
		& \mathcal E(t)
		:={}
		\frac12
		\int_{\mathcal O_\alpha}
		(\epsilon+[\varrho]_l)|\mathbf v|^2\,\dd\bm x
		+
		\frac12\|w(t)\|_{L^2(\Gamma)}^2
		+
		\frac12
		\|\Delta_\Gamma^\chi\zeta(t)\|_{L^2(\Gamma)}^2
		\notag \\
		& +
		\int_{\mathcal O_\alpha}
		\frac{a}{\gamma-1}\varrho^\gamma\,\dd\bm x
		+
		\int_{\mathcal O_\alpha}
		\frac{\delta}{\beta-1}\varrho^\beta\,\dd\bm x ,
		\label{eq:N_limit_energy_def} 
	\end{align}
	and
	\begin{align}
		& \mathcal D(t)
		:={}
		\int_0^t
		\int_{\mathcal O_\alpha}
		\mathbb S_\kappa^{\zeta^*}
		(\nabla^\phi\mathbf v):
		\nabla^\phi\mathbf v
		\,\dd\bm x\dd s
		+
		\nu_s
		\int_0^t
		\|\nabla_\Gamma^\chi w\|_{L^2(\Gamma)}^2\,\dd s
		\notag \\
		& +
		\epsilon
		\int_0^t
		\int_{\mathcal O_\alpha}
		\left(
		a\gamma\varrho^{\gamma-2}
		+
		\delta\beta\varrho^{\beta-2}
		\right)
		|\nabla^\phi\varrho|^2
		\,\dd\bm x\dd s
		\notag \\
		& +
		\frac1\delta
		\int_0^t
		\int_{T^\delta_{\zeta^*}}
		\left|
		(\mathbf v-w\mathbf e_z)
		\cdot\mathbf n^{\zeta^*,\phi}
		\right|^2
		\,\dd\bm x\dd s
		\notag \\
		& +
		\iota
		\int_0^t
		\int_{\Gamma_{\zeta^*}}
		J_\Gamma^\phi[\zeta^*]
		\left|
		\left(
		\mathbf v
		-
		w\mathbf e_z\circ\Phi_{\zeta^*}^{-1}
		\right)_{\tau^{\zeta^*,\phi}}
		\right|^2
		\,\dd S\dd s .
		\label{eq:N_limit_dissipation_def} 
	\end{align}
	
	\begin{proposition}[Limit $N\to\infty$]
		\label{thm:limit_N}
		The limiting variables
		$
		(
		\varrho,\mathbf v,\zeta,\zeta^*,w,\phi,W
		)
		$
		satisfy, $\widetilde{\mathbb P}$-almost surely, the following properties.
		
		\begin{enumerate}[label=(\arabic*)]
			
			\item
			We have
			\begin{equation*}				\partial_t\zeta=w
				\qquad\text{on }(0,T)\times\Gamma.
			\end{equation*}
			
			Then
			\[
			\zeta^*(t)=\zeta(t)
			\qquad
			\text{for every }t<\tau^\zeta,
			\]
			and
			\[
			\widetilde{\mathbb P}(\tau^\zeta>0)=1.
			\]
			
			\item
			For every
			$\varphi\in C^\infty(\overline{\mathcal O}_\alpha)$,
			\begin{align}
				& \int_{\mathcal O_\alpha}
				\varrho(t)\varphi\,\dd\bm x
				=
				\int_{\mathcal O_\alpha}
				\varrho_{0,\delta,\epsilon}\varphi\,\dd\bm x
				+
				\int_0^t
				\int_{\mathcal O_\alpha}
				\varrho\mathbf v\cdot\nabla^\phi\varphi
				\,\dd\bm x\dd s
				\notag \\
				& \quad
				-
				\epsilon
				\int_0^t
				\int_{\mathcal O_\alpha}
				\nabla^\phi\varrho\cdot\nabla^\phi\varphi
				\,\dd\bm x\dd s.
				\label{eq:N_limit_cont} 
			\end{align}
			
			\item
			For every
			\[
			\mathbf q\in X_m^f,
			\qquad
			\psi\in X_m^{st},
			\]
			set
			$\mathbf n_*:=\mathbf n^{\zeta^*,\phi}$,
			$\tau_*:=\tau^{\zeta^*,\phi}$, and
			$J_*:=J_\Gamma^\phi[\zeta^*]$. Then the coupled
			momentum--structure identity reads
			\begin{align}
				& \label{eq:N_limit_coupled}
				\int_{\mathcal O_\alpha}(\epsilon+[\varrho]_l)\mathbf v(t)\cdot\mathbf q\,\dd\bm x
				+\int_\Gamma w(t)\psi\,\dd\bm y\\
				& =\int_{\mathcal O_\alpha}(\epsilon+[\varrho_{0,\delta,\epsilon}]_l)\mathbf v_{0,\delta,\epsilon}\cdot\mathbf q\,\dd\bm x
				+\int_\Gamma w_0\psi\,\dd\bm y \notag \\
				& \quad+\int_0^t\!\!\int_{\mathcal O_\alpha}\Bigl([\varrho]_l\mathbf v\otimes\mathbf v:\nabla^\phi\mathbf q
				+(a\varrho^\gamma+\delta\varrho^\beta)\Div^\phi\mathbf q\notag \\
				& \qquad\qquad-\mathbb S_\kappa^{\zeta^*}(\nabla^\phi\mathbf v):\nabla^\phi\mathbf q
				-\frac\epsilon2[\nabla^\phi\varrho]_l\cdot\nabla^\phi\mathbf v\cdot\mathbf q\Bigr)\,\dd\bm x\dd s \notag \\
				& \quad-\int_0^t\!\!\int_\Gamma\Bigl(\Delta_\Gamma^\chi\zeta\,\Delta_\Gamma^\chi\psi
				+\nu_s\nabla_\Gamma^\chi w\cdot\nabla_\Gamma^\chi\psi\Bigr)\,\dd\bm y\dd s \notag \\
				& \quad-\frac1\delta\int_0^t\!\!\int_{T^\delta_{\zeta^*}}
				((\mathbf v-w\mathbf e_z)\cdot\mathbf n_*) ((\mathbf q-\psi\mathbf e_z)\cdot\mathbf n_*)\,\dd\bm x\dd s \notag \\
				& \quad-\iota\int_0^t\!\!\int_{\Gamma_{\zeta^*}}J_* \left(\mathbf v-w\mathbf e_z\circ\Phi_{\zeta^*}^{-1}\right)_{\tau_*} \cdot\left(\mathbf q-\psi\mathbf e_z\circ\Phi_{\zeta^*}^{-1}\right)_{\tau_*}\,\dd S\dd s. \notag 
			\end{align}
			\item
			The energy defined above satisfies
			\begin{align}
				& \mathcal E(t)+\mathcal D(t)
				\le{}
				\mathcal E(0)
				+
				M^{\rm sh}(t)
				+
				\frac12
				\sum_{k=1}^K
				\int_0^t
				\mathcal R_k(\widehat\eta(s))\,\dd s
				\label{eq:N_limit_energy} 
			\end{align}
			for every $t\in[0,T]$, after choosing the standard lower-semicontinuous
			representative. Moreover, for every $p\ge1$,
			\begin{equation}
				\label{eq:N_limit_energy_moment}
				\widetilde{\mathbb E}
				\left[
				\sup_{0\le t\le T}\mathcal E(t)^p
				+
				\mathcal D(T)^p
				\right]
				\le
				C_{p,T,m,\epsilon,l,\delta}.
			\end{equation}
			
		\end{enumerate}
	\end{proposition}

	\begin{proof}
		Lemma~\ref{lem:skorokhod_N} and the fixed-parameter compactness give
		\[
		\varrho_N\to\varrho
		\quad\text{in }L^4((0,T)\times\mathcal O_\alpha),
		\qquad
		\mathbf v_N\to\mathbf v
		\quad\text{in }L^2(0,T;X_m^f),
		\]
		and
		\[
		(\zeta_N,\zeta_N^*)\to(\zeta,\zeta^*)
		\quad\text{in }C([0,T];X_m^{st})^2,
		\qquad
		\phi_N\to\phi
		\quad\text{in }C^{\kappa_0}_tC^2_x.
		\]
		Since
		$l>0$,
		\[
		[\varrho_N]_l\to[\varrho]_l
		\quad\text{in }L^4(0,T;C^r(\overline{\mathcal O}_\alpha))
		\quad\text{for every fixed }r\ge0.
		\] Hence the
		continuity equation, $\partial_t\zeta=w$, and all bulk terms pass to the
		limit.

		Let $\mathcal R_N^{\rm int}$ be the difference between the two half-step
		interface forms and their continuous-time counterparts. Then
		\[
		\mathcal R_N^{\rm int}\longrightarrow0
		\]
		by \eqref{eq:N_zeta_shift_rate},
		\eqref{eq:N_time_translation_v_moment}, the
		$C_t^\vartheta$ convergence of the flow coefficients, and
		Lemma~\ref{lem:N_splitting_remainder}. Thus the limiting interface terms
		are precisely those in \eqref{eq:N_limit_coupled}.

		Lemma~\ref{lem:N_martingale_identification} gives
		\[
		M^{\rm sh}(t)
		=
		\sum_{k=1}^K
		\int_0^t
		\mathcal G_k(\widehat\eta(s))\,\dd W_k(s),
		\]
		and \eqref{eq:N_G_conv} gives the It\^o correction. Lower
		semicontinuity and Fatou's lemma yield
		\eqref{eq:N_limit_energy}--\eqref{eq:N_limit_energy_moment}. Finally,
		\[
		\zeta_N\to\zeta
		\quad\text{in }C([0,T];H^s(\Gamma)),
		\qquad s>\frac32,
		\]
		together with \eqref{eq:initial_localization_margin}, gives
		$\zeta^*=\zeta$ on $[0,\tau^\zeta)$ and
		$\widetilde{\mathbb P}(\tau^\zeta>0)=1$.
	\end{proof}

	\medskip

\section{Galerkin limit \texorpdfstring{$m\to\infty$}{m to infinity}}
	\label{s6}
	
	We now let $m\to\infty$ with $\epsilon,l,\delta>0$ fixed.  The artificial
	diffusion still gives strong compactness of the density, while the mollified
	inertial coefficient is spatially regular.  This is sufficient to identify the
	bulk nonlinearities and the Navier term.  We do not identify the quadratic
	shell-energy martingale at this intermediate level; only its uniform moment
	bound is retained, and the final shell energy is reconstructed in
	Appendix~\ref{app:shell_ito}.
	
	\subsection{Uniform estimates independent of \texorpdfstring{$m$}{m}}
	
	The estimates of Proposition~\ref{thm:limit_N} are uniform with respect to the
	Galerkin dimension.
	
	\begin{lemma}[Uniform estimates]
		
		For every $p\ge1$ there is a constant $C$, independent of $m$, with
		\[
		C=C(p,T,\epsilon,l,\delta,\alpha).
		\]
		Here $\mathcal E_m$ and $\mathcal D_m$ are the
		energy and dissipation defined in
		\eqref{eq:N_limit_energy_def}--\eqref{eq:N_limit_dissipation_def},
		evaluated at the $m$-level variables after the $N\to\infty$ passage.
		Then
		\begin{equation}
			\label{eq:m_energy_moment}
			\sup_m
			\widetilde{\mathbb E}
			\left[
			\sup_{0\le t\le T}\mathcal E_m(t)^p
			+
			\mathcal D_m(T)^p
			\right]
			\le C.
		\end{equation}
	\end{lemma}
	
	\begin{proof}
		All estimates follow from
		\eqref{eq:N_limit_energy_moment} and the definitions
		\eqref{eq:N_limit_energy_def}--\eqref{eq:N_limit_dissipation_def}.
		The extended stress is kept in the tensorial form
		\[
		\mathbb S_\kappa^{\zeta_m^*}(\nabla^\phi\mathbf v_m)
		=
		\mu_\kappa^{\zeta_m^*}
		\left(
		\nabla^\phi\mathbf v_m
		+
		(\nabla^\phi\mathbf v_m)^\top
		\right)
		+
		\lambda_\kappa^{\zeta_m^*}
		\Div^\phi\mathbf v_m\,\mathbb I,
		\]
		so that the coercivity assumption
		$
		\lambda_\kappa^{\zeta_m^*}
		+
		\frac23\mu_\kappa^{\zeta_m^*}\ge0
		$
		and Korn's inequality give \eqref{eq:m_energy_moment}. The penalty bound is
		only a bound on the \emph{normal} mismatch; no full-vector penalty is used.
	\end{proof}
	
	Since
	$
	\partial_t\zeta_m=w_m,
	$
	we also have
	\begin{equation}
		\label{eq:m_zeta_time}
		\sup_m
		\widetilde{\mathbb E}
		\|\zeta_m\|_{W^{1,2}(0,T;H^1(\Gamma))}^{2p}
		\le C.
	\end{equation}
	
	The equations provide the time regularity needed for compactness of the
	velocity and of the material shell velocity.
	
	\begin{lemma}[Time regularity]
		
		There exist $r_f>3/2$ and a constant
		$C=C(\epsilon,l,\delta)$ independent of $m$ such that
		\begin{align}
			& \sup_m
			\widetilde{\mathbb E}
			\|\partial_t w_m\|_{L^2(0,T;H^{-2}(\Gamma))}^2
			\le C,
			\label{eq:m_w_time} \\
			& \sup_m
			\widetilde{\mathbb E}
			\|\partial_t\mathbf v_m\|_{L^1(0,T;H^{-r_f}(\mathcal O_\alpha))}^2
			\le C.
			\label{eq:m_v_time} 
		\end{align}
		Consequently, for every $\sigma\in(0,1/2)$,
		\begin{align}
			& \{\mathbf v_m\}_m
			\quad\text{is relatively compact in}\quad
			L^2(0,T;H^{1-\sigma}(\mathcal O_\alpha)),
			\label{eq:m_v_compact} \\
			& \{w_m\}_m
			\quad\text{is relatively compact in}\quad
			L^2(0,T;L^2(\Gamma)).
			\label{eq:m_w_compact} 
		\end{align}
	\end{lemma}
	
	\begin{proof}
		The shell equation is tested against arbitrary
		$\psi\in H^2(\Gamma)$. The bending term is bounded by
		$\|\zeta_m\|_{H^2}\|\psi\|_{H^2}$, the damping term by
		$\|w_m\|_{H^1}\|\psi\|_{H^1}$, and the normal penalty term by
		Cauchy--Schwarz inequality together with \eqref{eq:m_energy_moment}. The Navier
		friction term is controlled by \eqref{eq:m_energy_moment} and the trace
		theorem. Since $\delta>0$ is fixed, this yields
		\eqref{eq:m_w_time}.
		
		For the fluid equation, write the time derivative of
		$(\epsilon+[\varrho_m]_l)\mathbf v_m$ in a sufficiently negative Sobolev
		space. The convective term is controlled by
		\eqref{eq:m_energy_moment}, the pressure by
		the fixed artificial pressure, the stress by
		\eqref{eq:m_energy_moment}, and the interface terms by
		\eqref{eq:m_energy_moment}.
		The time derivative of the mollified density is controlled by the
		continuity equation. This gives \eqref{eq:m_v_time}.
		
		The compactness conclusions follow from the Aubin--Lions--Simon theorem.
		Using
		\[
		H^1(\mathcal O_\alpha)
		\subset
		H^{1-\sigma}(\mathcal O_\alpha)
		\hookrightarrow
		H^{-r_f}(\mathcal O_\alpha),\;\;
		H^1(\Gamma)
		\subset
		L^2(\Gamma)
		\hookrightarrow
		H^{-2}(\Gamma),
		\]
		we get \eqref{eq:m_v_compact}-\eqref{eq:m_w_compact}.
	\end{proof}
	
	A useful consequence of \eqref{eq:m_v_compact} is compactness of traces.
	If $\sigma\in(0,1/2)$, then
	\[
	H^{1-\sigma}(\mathcal O_\alpha)
	\to
	H^{1/2-\sigma}(\partial\mathcal O_\alpha)
	\]
	continuously. After pulling the moving interfaces back to the reference
	torus by the graph parametrizations, the traces of $\mathbf v_m$ therefore
	converge strongly in the spaces needed for the Navier friction term.

	\subsection{Compactness and representation}
	
	Fix
	$
	s\in(3/2,2),
	\sigma\in(0,1/2),
	0<\kappa_0<\kappa_1<\frac12.
	$
	We use the path space
	\[
	\mathcal X^{(m)}
	=
	\mathcal X_\varrho
	\times
	\mathcal X_{\mathbf v}
	\times
	\mathcal X_\zeta
	\times
	\mathcal X_{\zeta^*}
	\times
	\mathcal X_w
	\times
	\mathcal X_\phi
	\times
	\mathcal X_W,
	\]
	where
	\begin{align*}		& \mathcal X_\varrho
		:=
		C_w([0,T];L^\beta(\mathcal O_\alpha))
		\cap
		L^4((0,T)\times\mathcal O_\alpha), \\
		& \mathcal X_{\mathbf v}
		:=
		\bigl(L^2(0,T;H^1(\mathcal O_\alpha)),w\bigr)
		\cap
		L^2(0,T;H^{1-\sigma}(\mathcal O_\alpha)), \\
		& \mathcal X_\zeta
		:=
		C([0,T];H^s(\Gamma))
		\cap
		\bigl(L^\infty(0,T;H^2(\Gamma)),w^*\bigr), \\
		& \mathcal X_{\zeta^*}
		:=
		C([0,T];H^s(\Gamma))
		\cap
		\bigl(L^\infty(0,T;H^2(\Gamma)),w^*\bigr), \\
		& \mathcal X_w
		:=
		\bigl(L^2(0,T;H^1(\Gamma)),w\bigr)
		\cap
		L^2(0,T;L^2(\Gamma)), \\
		& \mathcal X_\phi
		:=
		C^{\kappa_0}
		([0,T];C^2(\overline{\mathcal O}_\alpha;
		\overline{\mathcal O}_\alpha)),\;\;
		\mathcal X_W
		:=
		C([0,T];\mathbb R^K). 
	\end{align*}
	
	Let $\mu_m$ be the joint law of
	$
	(
	\varrho_m,\mathbf v_m,\zeta_m,\zeta_m^*,w_m,\phi,W
	)
	$
	on $\mathcal X^{(m)}$.
	
	\begin{lemma}[Tightness]
		
		The family
		$
		\{\mu_m\}_{m\in\mathbb N}
		$
		is tight on $\mathcal X^{(m)}$.
	\end{lemma}
	
	\begin{proof}
		The bounds
		\eqref{eq:m_energy_moment}, \eqref{eq:m_zeta_time}, and the compact
		embedding
		\[
		L^\infty(0,T;H^2(\Gamma))
		\cap
		W^{1,2}(0,T;H^1(\Gamma))
		\subset
		C([0,T];H^s(\Gamma)),
		\qquad s<2,
		\]
		give tightness of $\zeta_m$ and $\zeta_m^*$. The stopped displacement is
		handled by the same localization argument as in Section~\ref{s4}.
		
		Tightness of $w_m$ follows from
		\eqref{eq:m_energy_moment}, \eqref{eq:m_w_time}, and
		\eqref{eq:m_w_compact}. Tightness of $\mathbf v_m$ follows from
		\eqref{eq:m_energy_moment}, \eqref{eq:m_v_time}, and
		\eqref{eq:m_v_compact}.
		
		For the density,
		\[
		\partial_t\varrho_m
		=
		-\Div^\phi(\varrho_m\mathbf v_m)
		+
		\epsilon\Delta^\phi\varrho_m.
		\]
		Since $\epsilon>0$ is fixed, the density dissipation and the time derivative
		bound imply, by Aubin--Lions,
		\[
		\varrho_m
		\quad\text{is relatively compact in}\quad
		L^4((0,T)\times\mathcal O_\alpha).
		\]
		Together with \eqref{eq:m_energy_moment}, this yields tightness in
		$\mathcal X_\varrho$.
		
		The flow component is independent of $m$ before representation. Its
		tightness in $\mathcal X_\phi$ follows from the stronger
		$C^{\kappa_1}_tC^3_x$ moment bounds and the compact embedding into
		$C^{\kappa_0}_tC^2_x$. Tightness of the Wiener process is standard.
	\end{proof}

	We next apply the Jakubowski--Skorokhod representation to the tight family.
		Since weak and weak-$*$ topologies occur in $\mathcal X^{(m)}$, the space
	is not Polish. We therefore apply the Jakubowski--Skorokhod representation
	theorem.
	
	\begin{lemma}[Jakubowski--Skorokhod representation]
		\label{lem:skorokhod_m}
		There exist a probability space
		\[
		(\widehat\Omega,\widehat{\mathcal F},\widehat{\mathbb P}),
		\]
		random variables
		\[
		\widehat{\mathcal U}_m
		=
		(
		\widehat\varrho_m,
		\widehat{\mathbf v}_m,
		\widehat\zeta_m,
		\widehat\zeta_m^*,
		\widehat w_m,
		\widehat\phi_m,
		\widehat W_m
		),
		\qquad
		\mathcal U
		=
		(
		\varrho,\mathbf v,\zeta,\zeta^*,w,\phi,W
		),
		\]
		such that
		$
		\mathcal L(\widehat{\mathcal U}_m)=\mu_m
		$
		and
		\[
		\widehat{\mathcal U}_m
		\to
		\mathcal U
		\qquad
		\widehat{\mathbb P}\text{-a.s. in }\mathcal X^{(m)}.
		\]
	\end{lemma}
	
	From now on we work on the new probability space and omit the hats.
	Equal laws alone do not automatically provide the required limiting
	filtration. As in Section~\ref{s4}, one passes the canonical martingale
	characterizations of the Wiener process and of the stochastic flow.
	
	\begin{lemma}[Wiener process and stochastic flow after representation]
		
		Let $\widehat{\mathfrak F}_t$ be the
		usual $\widehat{\mathbb P}$-augmentation of the natural filtration generated by
		\[
		(\varrho,\mathbf v,\zeta,\zeta^*,w,\phi,W)
		\]
		up to time $t$. Then $W$ is a $K$-dimensional
		$(\widehat{\mathfrak F}_t)$-Wiener process and $\phi$ satisfies
		\eqref{phiv}. In particular,
		\[
		\chi=\phi|_\Gamma
		\]
		satisfies the surface-flow equation \eqref{eq:surface_flow}.
	\end{lemma}
	
	\begin{proof}
		For the represented approximations, the martingale identities for
		increments of $W_m$, for $W_{m,i}W_{m,j}-\delta_{ij}t$, and the weak
		formulation of the stochastic-flow equation depend only on their joint
		laws. The almost sure convergence
		Lemma~\ref{lem:skorokhod_m} and uniform moment bounds allow
		passage to these identities. L\'evy's characterization then gives the
		Wiener property. Pathwise uniqueness for the smooth stochastic-flow
		equation identifies $\phi$ as the flow generated by $W$.
	\end{proof}

	\subsection{Passage to the limit as \texorpdfstring{$m\to\infty$}{m to infinity}}
	
	We first record convergence of the Galerkin projections. For
	\[
	\mathbf q\in C_c^\infty(\overline{\mathcal O}_\alpha;\mathbb R^3),
	\qquad
	\psi\in C^\infty(\Gamma),
	\]
	set
	\[
	\mathbf q_m:=P_m^f\mathbf q,
	\qquad
	\psi_m:=P_m^{st}\psi.
	\]
	By the construction of the Galerkin bases in
	Section~\ref{sec:approx},
	\begin{equation}
		\label{eq:m_projection_conv}
		\mathbf q_m\to\mathbf q
		\quad\text{strongly in }H^{s_f}(\mathcal O_\alpha),
		\qquad
		\psi_m\to\psi
		\quad\text{strongly in }H^2(\Gamma).
	\end{equation}
	
	\begin{proposition}[Limit $m\to\infty$]
		\label{thm:limit_m}
		The represented limit $(\varrho,\mathbf v,\zeta,\zeta^*,w,\phi,W)$ satisfies
		$\partial_t\zeta=w$, the stopping relation of Proposition~\ref{thm:limit_N}, and
		$\widehat{\mathbb P}(\tau^\zeta>0)=1$. The continuity identity
		\eqref{eq:N_limit_cont} and the coupled identity \eqref{eq:N_limit_coupled}
		remain valid with the limiting variables, now for arbitrary smooth tests
		\[
		\mathbf q\in C_c^\infty([0,T]\times\overline{\mathcal O}_\alpha;\mathbb R^3),
		\qquad \psi\in C_c^\infty([0,T]\times\Gamma).
		\]
		If $\mathcal E^{(m)}$ and $\mathcal D^{(m)}$ denote
		\eqref{eq:N_limit_energy_def}--\eqref{eq:N_limit_dissipation_def} with the
		limiting variables, then for every $p\ge1$,
		\begin{equation}
			\label{eq:m_limit_energy_moment}
			\widehat{\mathbb E}\!\left[
			\sup_{0\le t\le T}\mathcal E^{(m)}(t)^p+\mathcal D^{(m)}(T)^p
			\right]\le C_{p,T,\epsilon,l,\delta}.
		\end{equation}
	\end{proposition}

	\begin{proof}
		For smooth $(\mathbf q,\psi)$ take
		\[
		\mathbf q_m=P_m^f\mathbf q,
		\qquad
		\psi_m=P_m^{st}\psi.
		\]
		By \eqref{eq:m_projection_conv} and the represented compactness,
		\[
		\mathbf q_m\to\mathbf q
		\quad\text{in }H^{s_f}(\mathcal O_\alpha),
		\qquad
		\psi_m\to\psi
		\quad\text{in }H^2(\Gamma),
		\]
		\[
		\mathbf v_m\to\mathbf v
		\quad\text{in }L^2_tH^{1-\sigma}_x,
		\qquad
		[\varrho_m]_l\to[\varrho]_l
		\quad\text{in }L^4(0,T;C^r(\overline{\mathcal O}_\alpha))
		\quad\text{for every fixed }r\ge0,
		\]
		while
		\[
		\nabla^{\phi_m}\mathbf v_m
		\rightharpoonup
		\nabla^\phi\mathbf v
		\quad\text{in }L^2.
		\]
		Thus the continuity equation and the bulk terms in
		\eqref{eq:N_limit_coupled} pass to the limit.

		Let $\mathcal P_m^{\rm nor}$ and $\mathcal N_m^{\rm slip}$ denote the
		normal-penalty and Navier-friction forms. The graph and trace compactness
		give
		\[
		\mathcal P_m^{\rm nor}\to\mathcal P^{\rm nor},
		\qquad
		\mathcal N_m^{\rm slip}\to\mathcal N^{\rm slip},
		\]
		while the shell terms pass weakly in the energy spaces. Hence the coupled
		identity holds for arbitrary smooth tests.

		By lower semicontinuity and Fatou's lemma,
		\[
		\widehat{\mathbb E}
		\left[
		\sup_{t\le T}\mathcal E^{(m)}(t)^p
		+\mathcal D^{(m)}(T)^p
		\right]
		\le C_{p,T,\epsilon,l,\delta},
		\]
		which is \eqref{eq:m_limit_energy_moment}. The $C_tH^s$ convergence and
		\eqref{eq:initial_localization_margin} yield the stopping relation and
		$\widehat{\mathbb P}(\tau^\zeta>0)=1$.
	\end{proof}

	\medskip

\section{Artificial-viscosity limit \texorpdfstring{$\epsilon\to0$}{epsilon to 0}}
	\label{s7}
	
	This section removes the artificial density diffusion while keeping the
	spatial mollification and artificial pressure fixed. We let
	$\epsilon\to0$ with $l,\delta>0$ fixed and write
	\[
	p_\delta(z)=az^\gamma+\delta z^\beta,\qquad \beta>\max\{6,\gamma\}.
	\]
	The loss of parabolic density compactness is handled by the effective viscous
	flux.  Because the pullback coefficients are only H\"older continuous in time,
	we localize bounded-flow events and freeze the coefficients on short
	deterministic intervals; the required quantitative estimates are proved below
	in Subsection~\ref{app:effective_flux}.
	
	\subsection{Uniform estimates}
	
	The Galerkin-limit estimates of Section~\ref{s6} are uniform in
	$\epsilon$.
	
	\begin{lemma}[Uniform estimates independent of $\epsilon$]
		
		For every $p\ge1$ there exists
		$
		C=C(p,T,l,\delta,\alpha)
		$
		independent of $\epsilon$.  The quantities
		$\mathcal E_\epsilon$ and $\mathcal D_\epsilon$ are the energy and
		dissipation in \eqref{eq:N_limit_energy_def}--\eqref{eq:N_limit_dissipation_def}
		after the Galerkin limit, with $\epsilon$ still present.  Then
		\begin{align}
			& \sup_{\epsilon>0}
			\mathbb E_\epsilon
			\left[
			\sup_{0\le t\le T}
			\mathcal E_\epsilon(t)^p
			+
			\mathcal D_\epsilon(T)^p
			\right]
			\le C,
			\label{eq:eps_energy_moment} 
		\end{align}
	\end{lemma}
	
	\begin{proof}
		This is the estimate
		\eqref{eq:m_limit_energy_moment} obtained after the Galerkin limit,
		together with the definitions of the limiting energy and dissipation.
		Notice in particular that
		\eqref{eq:eps_energy_moment} controls only the normal mismatch. There is no
		estimate of the form
		$
		\delta^{-1}
		\int
		|\mathbf v_\epsilon-w_\epsilon\mathbf e_z|^2,
		$
		which would be incompatible with the Navier-slip model.
	\end{proof}
	
	Since
	$
	\partial_t\zeta_\epsilon=w_\epsilon,
	$
	we also have
	\begin{equation}
		\label{eq:eps_zeta_time}
		\sup_{\epsilon>0}
		\mathbb E_\epsilon
		\|\zeta_\epsilon\|_
		{W^{1,2}(0,T;H^1(\Gamma))}^{2p}
		\le C.
	\end{equation}

	\subsection{Higher integrability and frozen-coefficient preparation}
	\label{app:effective_flux}
	
	The pressure gain is needed in the compactness of the momentum and in the
	effective-viscous-flux argument below. Since the pullback coefficients are
	only H\"older continuous in time, we first localize on bounded-flow events
	and freeze the geometry on short deterministic intervals. This yields the
	uniform higher-integrability estimate before the tightness lemma is invoked.

	We first localize the stochastic flow and freeze the pullback operators on deterministic time intervals.

	Fix
	$
	0<\vartheta<\frac12
	$
	and set
	\begin{equation*}		\Lambda
		:=
		1+
		\|\phi\|_{C^\vartheta([0,T];C^2(\overline {\mathcal O_\alpha}))}
		+
		\|\phi^{-1}\|_{C^\vartheta([0,T];C^2(\overline {\mathcal O_\alpha}))}.
	\end{equation*}
	The stochastic-flow estimates give
	\begin{equation*}		\mathbb E[\Lambda^r]<\infty
		\qquad\text{for every }1\le r<\infty .
	\end{equation*}
	
	For $R\ge1$ let
	\[
	\Omega_R:=\{\Lambda\le R\}.
	\]
	Given $\xi\in(0,1)$, choose a deterministic partition
	\[
	0=t_0<t_1<\cdots<t_M=T
	\]
	fine enough that, on $\Omega_R$,
	\begin{equation}
		\label{app:eq_freezing_bound}
		\sup_{t\in[t_i,t_{i+1}]}
		\left(
		\|A_\phi(t)-A_\phi(t_i)\|_{C^1(\overline {\mathcal O_\alpha})}
		+
		\|\nabla\phi(t)-\nabla\phi(t_i)\|_{C^1(\overline {\mathcal O_\alpha})}
		\right)
		\le\xi ,
	\end{equation}
	where
	\[
	A_\phi(t,x)
	=
	\nabla\phi^{-1}(t,\phi(t,x)).
	\]
	The partition is deterministic for fixed $R,\xi$, so no anticipative
	random time is introduced.  At the end of the argument we let
	$\xi\to0$ and then $R\to\infty$.
	
	We first record the artificial-diffusion estimate used below.
	
	\begin{lemma}[Artificial-diffusion estimate]
		\label{app:lem_unweighted_diffusion}
		Let $\varrho_\epsilon$ solve
		\[
		\partial_t\varrho_\epsilon
		+
		\Div^\phi(\varrho_\epsilon\mathbf v_\epsilon)
		=
		\epsilon\Delta^\phi\varrho_\epsilon .
		\]
		Then, for every $R\ge1$,
		\begin{equation}
			\label{app:eq_unweighted_diffusion}
			\sup_{\epsilon>0}
			\mathbb E
			\left[
			\mathbf 1_{\Omega_R}
			\epsilon
			\int_0^T
			\|\nabla^\phi\varrho_\epsilon\|_{L^2({\mathcal O_\alpha})}^2
			\,\dd t
			\right]
			\le C_R .
		\end{equation}
		Consequently,
		\begin{equation}
			\label{app:eq_epsilon_grad_vanish}
			\epsilon\nabla^\phi\varrho_\epsilon
			\to0
			\qquad
			\text{in }L^2((0,T)\times {\mathcal O_\alpha})
			\text{ in probability}.
		\end{equation}
	\end{lemma}
	
	\begin{proof}
		Testing the continuity equation by $\varrho_\epsilon$ and using volume
		preservation gives
		\[
		\frac12\frac{\dd}{\dd t}
		\|\varrho_\epsilon\|_{L^2({\mathcal O_\alpha})}^2
		+
		\epsilon
		\|\nabla^\phi\varrho_\epsilon\|_{L^2({\mathcal O_\alpha})}^2
		=
		-\frac12
		\int_{\mathcal O_\alpha}
		\varrho_\epsilon^2
		\Div^\phi\mathbf v_\epsilon\,\dd x .
		\]
		On $\Omega_R$, the pullback norms are uniformly equivalent to the
		standard ones.  Since $\beta>4$, the right-hand side is controlled by
		the energy bounds, which yields
		\eqref{app:eq_unweighted_diffusion}.  The convergence
		\eqref{app:eq_epsilon_grad_vanish} follows immediately.
	\end{proof}
	
	On $I_i=(t_i,t_{i+1})$ write
	$
	\phi_i:=\phi(t_i).
	$
	Let
	\[
	\mathcal B:
	L_0^q({\mathcal O_\alpha})\to W_0^{1,q}({\mathcal O_\alpha};\mathbb R^3)
	\]
	be a classical Bogovski\u{\i} operator and define the frozen transformed
	operator
	\begin{equation*}		\mathcal B_i^\phi[f]
		:=
		\left[
		\mathcal B(f\circ\phi_i^{-1})
		\right]\circ\phi_i .
	\end{equation*}
	By volume preservation and the classical Bogovski\u{\i} estimates,
	\begin{equation*}		\Div^{\phi_i}\mathcal B_i^\phi[f]=f,
	\end{equation*}
	and
	\begin{equation*}		\|\mathcal B_i^\phi[f]\|_{W^{1,q}({\mathcal O_\alpha})}
		\le C(R,q)\|f\|_{L^q({\mathcal O_\alpha})}
		\qquad
		\text{on }\Omega_R .
	\end{equation*}
	We also use the corresponding negative-norm estimate
	\begin{equation}
		\label{app:eq_Bi_negative}
		\left\|
		\mathcal B_i^\phi
		\bigl[\Div^{\phi_i}\mathbf F\bigr]
		\right\|_{L^q({\mathcal O_\alpha})}
		\le
		C(R,q)\|\mathbf F\|_{L^q({\mathcal O_\alpha})}.
	\end{equation}
	
	Density-dependent Bogovski\u{\i} tests are regularized in time before
	being inserted into the weak momentum equation.  Since
	$\mathcal B_i^\phi$ is frozen on $I_i$, the continuity equation together
	with \eqref{app:eq_Bi_negative} gives the required time-derivative
	bounds.  The only additional term is produced by
	$\Div^\phi-\Div^{\phi_i}$, and is $O_R(\xi)$ by
	\eqref{app:eq_freezing_bound}.

	We next derive the pressure gain uniformly in $\epsilon$ at fixed $\delta$.

	\begin{proposition}[Higher integrability at fixed $\delta$]
		\label{app:prop_higher_integrability}
		There exists $C_\delta<\infty$, independent of $\epsilon$, such that
		\begin{equation}
			\label{app:eq_higher_integrability}
			\sup_{\epsilon>0}
			\mathbb E
			\int_0^T\!\!\int_{\mathcal O_\alpha}
			\left(
			a\varrho_\epsilon^{\gamma+1}
			+
			\delta\varrho_\epsilon^{\beta+1}
			\right)
			\,\dd x\,\dd t
			\le
			C_\delta .
		\end{equation}
	\end{proposition}
	
	\begin{proof}
		Fix $R\ge1$ and a frozen interval $I_i=(t_i,t_{i+1})$. Set
		\[
		g_\epsilon:=\varrho_\epsilon-
		\langle\varrho_\epsilon\rangle_{\mathcal O_\alpha},
		\qquad
		\mathbf q_{\epsilon,i}^\tau
		:=\chi_i\mathcal B_i^\phi[g_\epsilon^\tau],
		\]
		where $g_\epsilon^\tau$ is a backward time regularization. By
		\eqref{app:eq_Bi_negative},
		\[
		\sup_{\tau>0}
		\|\mathbf q_{\epsilon,i}^\tau\|_{L^q(I_i;W^{1,q})}
		\le C_R\|g_\epsilon\|_{L^q(I_i\times\mathcal O_\alpha)}.
		\]
		The pressure contribution is
		\begin{align}
		\int_{I_i}\!\!\int_{\mathcal O_\alpha}
		\chi_i p_\delta(\varrho_\epsilon)g_\epsilon\,\dd x\,\dd t
		&=
		\int_{I_i}\!\!\int_{\mathcal O_\alpha}
		\chi_i\left(a\varrho_\epsilon^{\gamma+1}
		+\delta\varrho_\epsilon^{\beta+1}\right)\,\dd x\,\dd t
		\notag\\
		&-
		\int_{I_i}\chi_i
		\langle\varrho_\epsilon\rangle_{\mathcal O_\alpha}
		\left(\int_{\mathcal O_\alpha}p_\delta(\varrho_\epsilon)\,\dd x\right)\dd t .
		\label{app:eq_higher_pressure_split}
		\end{align}
		The last term is bounded by $C_{R,\delta}$. Let
		$\mathcal R_{\epsilon,i}^\tau$ be the sum of the convective, viscous,
		interface, time-derivative, artificial-diffusion, and frozen-coefficient
		remainders. The energy and trace bounds,
		\eqref{app:eq_Bi_negative}, Lemma~\ref{app:lem_unweighted_diffusion}, and
		\eqref{app:eq_freezing_bound} give
		\begin{align}
		\mathbb E\!\left[
		\mathbf1_{\Omega_R}|\mathcal R_{\epsilon,i}^\tau|\right]
		&\le C_{R,\delta}
		+C_R\xi\,
		\mathbb E\!\left[
		\mathbf1_{\Omega_R}
		\int_{I_i}\!\!\int_{\mathcal O_\alpha}
		\left(a\varrho_\epsilon^{\gamma+1}
		+\delta\varrho_\epsilon^{\beta+1}\right)\,\dd x\,\dd t
		\right]
		+o_\tau(1).
		\label{app:eq_higher_remainder_bound}
		\end{align}
		Letting $\tau\to0$ and summing over the deterministic partition yield
		\begin{align}
		&\mathbb E\!\left[
		\mathbf1_{\Omega_R}
		\int_0^T\!\!\int_{\mathcal O_\alpha}
		\left(a\varrho_\epsilon^{\gamma+1}
		+\delta\varrho_\epsilon^{\beta+1}\right)\,\dd x\,\dd t
		\right]
		\notag\\
		&\qquad\le C_{R,\delta}
		+C_R\xi\,
		\mathbb E\!\left[
		\mathbf1_{\Omega_R}
		\int_0^T\!\!\int_{\mathcal O_\alpha}
		\left(a\varrho_\epsilon^{\gamma+1}
		+\delta\varrho_\epsilon^{\beta+1}\right)\,\dd x\,\dd t
		\right].
		\label{app:eq_higher_absorption}
		\end{align}
		Choose $\xi$ so that $C_R\xi<1/2$ and absorb the last term. Since the
		frozen constants grow at most polynomially in $R$ and
		$\mathbb E[\Lambda^q]<\infty$ for every finite $q$, a dyadic decomposition
		removes the localization and gives \eqref{app:eq_higher_integrability}.
	\end{proof}

	We next establish tightness and apply the Jakubowski--Skorokhod representation.
		The estimate \eqref{app:eq_higher_integrability}, together with the energy
	bounds and the equations, supplies the integrability used in the momentum
	time-regularity estimate. We use the path space
	
	\[
	\mathcal Z
	=
	\mathcal Z_\varrho
	\times
	\mathcal Z_{\mathbf v}
	\times
	\mathcal Z_{\mathbf m}
	\times
	\mathcal Z_\zeta
	\times
	\mathcal Z_{\zeta^*}
	\times
	\mathcal Z_w
	\times
	\mathcal Z_\phi
	\times
	\mathcal Z_W,
	\]
	where
	\begin{align*}        & \mathcal Z_\varrho
		:=
		C_w([0,T];L^\beta(\mathcal O_\alpha)), \;\; \mathcal Z_{\mathbf v}
		:=
		\bigl(
		L^2(0,T;H^1(\mathcal O_\alpha;\mathbb R^3)),w
		\bigr), \\
		& \mathcal Z_{\mathbf m}
		:=
		C_w([0,T];W^{-r,2}(\mathcal O_\alpha;\mathbb R^3)),
		\; r>\frac32, \\
		& \mathcal Z_\zeta
		:=
		C([0,T];H^s(\Gamma))
		\cap
		\bigl(
		L^\infty(0,T;H^2(\Gamma)),w^*
		\bigr), \\
		& \mathcal Z_{\zeta^*}
		:=
		C([0,T];H^s(\Gamma))
		\cap
		\bigl(
		L^\infty(0,T;H^2(\Gamma)),w^*
		\bigr), \\
		& \mathcal Z_w
		:=
		\bigl(
		L^2(0,T;H^1(\Gamma)),w
		\bigr)
		\cap
		L^2(0,T;L^2(\Gamma)), \\
		& \mathcal Z_\phi
		:=
		C^{\vartheta_0}
		([0,T];C^2(\overline{\mathcal O}_\alpha;
		\overline{\mathcal O}_\alpha)),
		\;
		0<\vartheta_0<\vartheta,\;\;\mathcal Z_W
		:=
		C([0,T];\mathbb R^K). 
	\end{align*}
	
	For the momentum component we take
	\[
	\mathbf m_\epsilon:=\varrho_\epsilon\mathbf v_\epsilon.
	\]
	The continuity equation gives
	\[
	\partial_t\varrho_\epsilon
	=
	-\Div^\phi\mathbf m_\epsilon
	+
	\epsilon\Delta^\phi\varrho_\epsilon,
	\]
	and the energy bounds give tightness of $\mathbf m_\epsilon$ in a
	sufficiently negative path space.
	
	\begin{lemma}[Tightness]
		
		The joint laws of
		$
		(
		\varrho_\epsilon,
		\mathbf v_\epsilon,
		\mathbf m_\epsilon,
		\zeta_\epsilon,
		\zeta_\epsilon^*,
		w_\epsilon,
		\phi,
		W
		)
		$
		are tight on $\mathcal Z$.
	\end{lemma}
	
	\begin{proof}
		The energy bound \eqref{eq:eps_energy_moment} and the continuity equation give tightness in $\mathcal Z_\varrho$. The velocity is tight in the weak
		$L^2_tH^1_x$ topology by \eqref{eq:eps_energy_moment}. The momentum estimate, the pressure gain \eqref{app:eq_higher_integrability},
		and the time regularity obtained from the momentum equation give tightness
		in $\mathcal Z_{\mathbf m}$.
		
		For $\zeta_\epsilon$ and $\zeta_\epsilon^*$ we use
		\eqref{eq:eps_energy_moment}, \eqref{eq:eps_zeta_time}, and the compact
		embedding into $C_tH^s_y$ for $s<2$. The shell equation and
		\eqref{eq:eps_energy_moment} give the time compactness needed for $w_\epsilon$.
		The stochastic flow is tight in
		$C^{\vartheta_0}_tC^2_x$ by the stronger
		$C^\vartheta_tC^3_x$ moment estimates, and tightness of $W$ is standard.
	\end{proof}
	
	Define the joint variable
	\[
	\mathbf U_\epsilon
	:=
	(\varrho_\epsilon,\mathbf v_\epsilon,\mathbf m_\epsilon,
	\zeta_\epsilon,\zeta_\epsilon^*,w_\epsilon,\phi_\epsilon,W_\epsilon),
	\qquad
	\mathbf U
	:=
	(\varrho,\mathbf v,\mathbf m,\zeta,\zeta^*,w,\phi,W).
	\]
	Applying the Jakubowski--Skorokhod representation theorem, we obtain a
	new probability space
	$(\widehat\Omega,\widehat{\mathcal F},\widehat{\mathbb P})$
	and copies, again denoted without hats, such that
	\begin{equation*}		\mathbf U_\epsilon\to\mathbf U
		\qquad\widehat{\mathbb P}\text{-a.s. in }\mathcal Z.
	\end{equation*}
	Thus all component convergences are read directly from the definition of
	$\mathcal Z$: in particular, the density converges in $C_wL^\beta$, the velocity converges weakly in
	$L^2H^1$, the momentum converges in $C_wW^{-r,2}$, the two graph
	variables converge strongly in $C_tH^s$, and the flow and Wiener process
	converge in their strong path topologies.  As in
	Sections~\ref{s4}--\ref{s6}, the canonical martingale
	characterizations pass to the represented variables.  Hence $W$ is a
	Wiener process for the limiting canonical filtration and $\phi$ is the
	stochastic flow generated by $W$.
	
	The product is identified by the continuity equation:
	\begin{equation*}		\mathbf m=\varrho\mathbf v.
	\end{equation*}
	Indeed, the compactness of the momentum, the weak convergence of the
	velocity, and the strong compactness of the mollified density
	$[\varrho_\epsilon]_l$ give the standard Lions product identification at
	fixed $l>0$.

	\subsection{Effective viscous flux and density compactness}

	Let $K\subset {\mathcal O_\alpha}$ and choose
	\[
	\vartheta_0,\vartheta_1\in C_c^\infty({\mathcal O_\alpha}),
	\qquad
	\vartheta_0=1\ \text{on }K,
	\qquad
	\vartheta_1=1\ \text{on }\supp\vartheta_0 .
	\]
	For a frozen time $t_i$ define
	\begin{equation*}		\mathcal A_i[f]
		:=
		\left[
		\nabla_y(-\Delta_y)^{-1}
		\widetilde{
			(\vartheta_1f)\circ\phi_i^{-1}
		}
		\right]\circ\phi_i ,
	\end{equation*}
	where the tilde denotes extension by zero to $\mathbb R^3$.
	The usual Calder\'on--Zygmund estimates hold uniformly on
	$\Omega_R$.
	
	We shall also use the classical compact Riesz-commutator property in the
	frozen coordinates.  Namely, if
	\[
	b_n\to b
	\quad\text{in }C^1(\overline {\mathcal O_\alpha}),
	\qquad
	F_n\rightharpoonup F
	\quad\text{in }L^2({\mathcal O_\alpha}),
	\]
	with $b_n$ uniformly bounded in $C^2$, then
	\begin{equation}
		\label{app:eq_Riesz_commutator_compact}
		[
		\mathcal R_{ab}^{\,i},b_n
		]F_n
		\to
		[
		\mathcal R_{ab}^{\,i},b
		]F
		\qquad
		\text{strongly in }L^q(K),
		\quad 1\le q<6.
	\end{equation}
	This is the standard local Calder\'on--Zygmund commutator theorem after
	the volume-preserving change of variables $y=\phi_i(x)$.
	
	At fixed $\delta$, the regularized viscosity coefficients satisfy
	\begin{equation}
		\label{app:eq_viscosity_strong_C1}
		\mu_\kappa^{\zeta_\epsilon^*}
		\to
		\mu_\kappa^{\zeta^*},
		\qquad
		\lambda_\kappa^{\zeta_\epsilon^*}
		\to
		\lambda_\kappa^{\zeta^*}
		\quad
		\text{in }C([0,T];C^1(\overline {\mathcal O_\alpha})),
	\end{equation}
	and
	\begin{equation*}		\sup_\epsilon
		\left(
		\|\mu_\kappa^{\zeta_\epsilon^*}\|_{C_tC_x^2}
		+
		\|\lambda_\kappa^{\zeta_\epsilon^*}\|_{C_tC_x^2}
		\right)
		\le
		C_\delta .
	\end{equation*}
	
	Set
	$
	\mu_\epsilon:=\mu_\kappa^{\zeta_\epsilon^*},
	\lambda_\epsilon:=\lambda_\kappa^{\zeta_\epsilon^*},
	$
	and
	\[
	\mathbf S_\epsilon
	:=
	\mu_\epsilon
	\left(
	\nabla^{\phi_\epsilon}\mathbf v_\epsilon
	+
	(\nabla^{\phi_\epsilon}\mathbf v_\epsilon)^\top
	\right)
	+
	\lambda_\epsilon
	\Div^{\phi_\epsilon}\mathbf v_\epsilon\,\mathbb I .
	\]
	
	For
	\[
	T_k(z):=\min\{z,k\},
	\qquad
	\mathbf Q_{\epsilon,i,k}
	:=
	\mathcal A_i[T_k(\varrho_\epsilon)],
	\]
	the coefficient-freezing errors satisfy, on $\Omega_R$, the following
	representative estimates.  Writing
	$B_\epsilon=A_{\phi_\epsilon}$ and $B_i=A_{\phi_i}$, the freezing bound
	gives $\|B_\epsilon-B_i\|_{C^1}\le C_R\xi$.  Since
	$T_k(\varrho_\epsilon)$ is bounded and the frozen potentials satisfy the
	uniform Calder\'on--Zygmund estimates,
	\[
	\left|
	\int_{I_i}\!\!\int_{\mathcal O_\alpha}
	\mathbf S_\epsilon:
	\bigl(\nabla^{\phi_\epsilon}-\nabla^{\phi_i}\bigr)
	(\psi\mathbf Q_{\epsilon,i,k})
	\right|
	\le
	C_{R,k,\psi}\xi
	\left(1+\|\mathbf v_\epsilon\|_{L^2(I_i;H^1)}^2\right),
	\]
	\[
	\left|
	\int_{I_i}\!\!\int_{\mathcal O_\alpha}
	[\varrho_\epsilon]_l\mathbf v_\epsilon\otimes\mathbf v_\epsilon:
	\bigl(\nabla^{\phi_\epsilon}-\nabla^{\phi_i}\bigr)
	(\psi\mathbf Q_{\epsilon,i,k})
	\right|
	\le
	C_{R,k,\psi}\xi
	\left(1+\|\mathbf v_\epsilon\|_{L^2(I_i;H^1)}^2\right),
	\]
	and the analogous divergence error satisfies
	\[
	\left|
	\int_{I_i}\!\!\int_{\mathcal O_\alpha}
	p_\delta(\varrho_\epsilon)
	\bigl(\Div^{\phi_\epsilon}-\Div^{\phi_i}\bigr)
	(\psi\mathbf Q_{\epsilon,i,k})
	\right|
	\le C_{R,k,\psi}\xi
	\|p_\delta(\varrho_\epsilon)\|_{L^1(I_i\times {\mathcal O_\alpha})}.
	\]
	The terms in which the continuity equation is used to differentiate the
	density-dependent test are estimated in the same way with
	\eqref{app:eq_Bi_negative}.  Consequently, on $\Omega_R$,
	\begin{equation}
		\label{app:eq_flux_freezing_bound}
		|\mathfrak E_{\epsilon,i,k}^{\rm fr}|
		\le
		C_{R,k,\psi}\xi
		\left(
		1+
		\|\mathbf v_\epsilon\|_{L^2(I_i;H^1({\mathcal O_\alpha}))}^2
		\right),
	\end{equation}
	where $\mathfrak E_{\epsilon,i,k}^{\rm fr}$ denotes collectively the
	errors generated by
	\[
	\nabla^{\phi_\epsilon}-\nabla^{\phi_i},
	\qquad
	\Div^{\phi_\epsilon}-\Div^{\phi_i}.
	\]
	Thus all freezing errors vanish after $\xi\to0$.
	
	The time-derivative and convective terms are treated by the standard
	Lions commutator argument in the frozen coordinates.  At fixed $l>0$, spatial convolution converts the negative-Sobolev time
	control furnished by the continuity equation into equicontinuity in every
	fixed $C^m$ norm.  Indeed, after convolution,
	$\partial_t[\varrho_\epsilon]_l$ is bounded in
	$L^1(0,T;C^m(\overline {\mathcal O_\alpha}))$ with a constant depending on $l$ but not on
	$\epsilon$; the artificial-diffusion part vanishes by
	Lemma~\ref{app:lem_unweighted_diffusion}.  Arzel\`a--Ascoli and the weak
	density convergence therefore give
	\[
	[\varrho_\epsilon]_l
	\to
	[\varrho]_l
	\qquad
	\text{in }C([0,T];C^m(\overline {\mathcal O_\alpha}))
	\]
	for every fixed integer $m\ge0$.  Hence the mollified density is a
	strongly convergent multiplier in the inertial commutator identities.
	
	\begin{proposition}[Local weak continuity of the effective viscous flux]
		\label{app:prop_effective_flux}
		Let
		\[
		K\subset {\mathcal O_\alpha},
		\qquad
		\psi\in C_c^\infty((0,T)\times K),
		\qquad
		k\in\mathbb N .
		\]
		Then, almost surely along the represented subsequence,
		\begin{align}
			& \lim_{\epsilon\to0}
			\int_0^T\!\!\int_{\mathcal O_\alpha}
			\psi
			\left[
			p_\delta(\varrho_\epsilon)
			-
			(\lambda_\epsilon+2\mu_\epsilon)
			\Div^{\phi_\epsilon}\mathbf v_\epsilon
			\right]
			T_k(\varrho_\epsilon)
			\,\dd x\,\dd t
			\notag \\
			& \qquad=
			\int_0^T\!\!\int_{\mathcal O_\alpha}
			\psi
			\left[
			\overline{p_\delta(\varrho)}
			-
			(\lambda+2\mu)
			\Div^\phi\mathbf v
			\right]
			\overline{T_k(\varrho)}
			\,\dd x\,\dd t,
			\label{app:eq_effective_flux} 
		\end{align}
		where
		$
		\mu=\mu_\kappa^{\zeta^*},
		\lambda=\lambda_\kappa^{\zeta^*}.
		$
	\end{proposition}
	
	\begin{proof}
		Fix $R$ and $\xi$. On each frozen interval $I_i$, use a time-regularized
		test of
		\[
		\psi\,\mathcal A_i[T_k(\varrho_\epsilon)].
		\]
		The frozen pressure--viscosity term gives the effective-flux expression.
		The remaining remainders satisfy
		\[
		\mathcal R_{\epsilon,i}^{\rm visc}\to0,
		\qquad
		\mathcal R_{\epsilon,i}^{\rm diff}\to0,
		\qquad
		\mathcal R_{\epsilon,i}^{\rm inert}\to0,
		\]
		by \eqref{app:eq_Riesz_commutator_compact}--%
		\eqref{app:eq_viscosity_strong_C1},
		Lemma~\ref{app:lem_unweighted_diffusion}, and the Lions commutator
		argument. At fixed $\delta$, the interface terms converge by strong geometry
		and trace compactness, while
		\[
		|\mathcal R_{\epsilon,i}^{\rm fr}|
		\le C_{R,k,\psi}\xi
		\]
		by \eqref{app:eq_flux_freezing_bound}. Therefore
		\[
		\limsup_{\epsilon\to0}
		\left|
		\mathrm{LHS}\bigl(\eqref{app:eq_effective_flux}\bigr)
		-
		\mathrm{RHS}\bigl(\eqref{app:eq_effective_flux}\bigr)
		\right|
		\le C_{R,k,\psi}\xi .
		\]
		Letting $\xi\to0$ and then $R\to\infty$ gives
		\eqref{app:eq_effective_flux}.
	\end{proof}

	We next combine renormalization with the effective-viscous-flux identity to obtain strong convergence of the density.

	At the artificial-viscosity level the continuity equation satisfies the
	usual renormalized identity.  Namely, if
	$b\in C^2([0,\infty))$ has bounded $b'$ and compactly supported $b''$,
	then
	\begin{align*}		& \partial_t b(\varrho_\epsilon)
		+
		\Div^\phi
		\left(
		b(\varrho_\epsilon)\mathbf v_\epsilon
		\right)
		+
		\left(
		b'(\varrho_\epsilon)\varrho_\epsilon
		-
		b(\varrho_\epsilon)
		\right)
		\Div^\phi\mathbf v_\epsilon
		\notag \\
		& =
		\epsilon
		\Delta^\phi b(\varrho_\epsilon)
		-
		\epsilon
		b''(\varrho_\epsilon)
		|\nabla^\phi\varrho_\epsilon|^2 
	\end{align*}
	in distributions. At fixed $\epsilon>0$, this follows from the standard
	parabolic renormalization argument by spatial regularization and passage to
	the limit.
	
	For $q>1$ define
	\begin{equation*}		\operatorname{osc}_q
		[\varrho_\epsilon\to\varrho]
		:=
		\sup_{k\ge1}
		\limsup_{\epsilon\to0}
		\int_0^T\!\!\int_{\mathcal O_\alpha}
		|T_k(\varrho_\epsilon)-T_k(\varrho)|^q
		\,\dd x\,\dd t .
	\end{equation*}
	
	\begin{proposition}[Oscillation-defect bound]
		\label{app:prop_oscillation_defect}
		At fixed $\delta>0$,
		\begin{equation}
			\label{app:eq_osc_finite}
			\operatorname{osc}_{\beta+1}
			[\varrho_\epsilon\to\varrho]
			<\infty
			\qquad
			\text{almost surely}.
		\end{equation}
		Consequently, the limit continuity equation is renormalized in the
		DiPerna--Lions sense.
	\end{proposition}
	
	\begin{proof}
		Combine the weak continuity of the effective viscous flux
		\eqref{app:eq_effective_flux} with the renormalized equations for the
		truncations $T_k$.  The monotonicity of
		$p_\delta(z)=az^\gamma+\delta z^\beta$ gives the usual nonnegative
		defect, while the freezing and commutator errors have already been
		removed.  The classical Lions--Feireisl oscillation-defect estimate then
		gives \eqref{app:eq_osc_finite}; see
		\cite{Lions1998,Feireisl2001}.
	\end{proof}
	
	\begin{proposition}[Strong convergence and pressure identification]
		\label{app:thm_strong_density}
		Almost surely,
		\begin{equation}
			\label{app:eq_rho_strong}
			\varrho_\epsilon
			\to
			\varrho
			\qquad
			\text{strongly in }
			L^1((0,T)\times {\mathcal O_\alpha}).
		\end{equation}
		Moreover,
		\begin{align}
			& \varrho_\epsilon^\gamma
			\to
			\varrho^\gamma,\;\; \varrho_\epsilon^\beta
			\to
			\varrho^\beta,
			\qquad\text{strongly in }L^1((0,T)\times {\mathcal O_\alpha}),
			\label{app:eq_beta_strong} 
		\end{align}
		and hence
		\begin{equation*}			\overline{p_\delta(\varrho)}
			=
			a\varrho^\gamma+\delta\varrho^\beta .
		\end{equation*}
	\end{proposition}
	
	\begin{proof}
		For $k\ge1$ set
		\[
		T_k(z):=\min\{z,k\},
		\qquad
		L_k(z):=z\int_1^z\frac{T_k(s)}{s^2}\,\dd s,
		\qquad
		zL_k'(z)-L_k(z)=T_k(z).
		\]
		The renormalized equation for $L_k(\varrho_\epsilon)$ contains
		\[
		-\epsilon L_k''(\varrho_\epsilon)
		|\nabla^\phi\varrho_\epsilon|^2\le0,
		\]
		and all remaining artificial-diffusion terms vanish by
		Lemma~\ref{app:lem_unweighted_diffusion}. Thus the only possible defect is
		$T_k(\varrho_\epsilon)\Div^\phi\mathbf v_\epsilon$.
		Proposition~\ref{app:prop_effective_flux} and
		\[
		\bigl(p_\delta(z_1)-p_\delta(z_2)\bigr)
		\bigl(T_k(z_1)-T_k(z_2)\bigr)\ge0
		\]
		give \eqref{app:eq_osc_finite}. Letting $k\to\infty$ yields
		\[
		\overline{\varrho\log\varrho}
		=\varrho\log\varrho
		\qquad\text{a.e.},
		\]
		and hence
		\[
		\varrho_\epsilon\to\varrho
		\qquad\text{a.e. on }(0,T)\times\mathcal O_\alpha.
		\]
		Proposition~\ref{app:prop_higher_integrability} gives uniform integrability of
		$
		\varrho_\epsilon,
		\varrho_\epsilon^\gamma,
			\varrho_\epsilon^\beta.
		$
		Vitali's theorem gives \eqref{app:eq_rho_strong}--%
		\eqref{app:eq_beta_strong} and therefore
		\[
		\overline{p_\delta(\varrho)}
		=a\varrho^\gamma+\delta\varrho^\beta.
		\]
	\end{proof}

	\subsection{Limit passage and pressure identification}
	
	After the Skorokhod representation, all terms except the pressure can be
	passed to the limit directly.  We write
	\begin{equation}
		\label{eq:pressure_bar}
		p_\delta(\varrho_\epsilon)
		\rightharpoonup
		\overline{p_\delta(\varrho)}
		\qquad
		\text{in }
		L^{\frac{\beta+1}{\beta}}
		((0,T)\times\mathcal O_\alpha).
	\end{equation}
	
	\begin{proposition}[Preliminary limit $\epsilon\to0$]
		\label{thm:limit_eps_preliminary}
		The limiting variables satisfy, $\widehat{\mathbb P}$-almost surely,
		\[
		\partial_t\zeta=w,
		\qquad
		\zeta^*=\zeta
		\quad\text{on }[0,\tau^\zeta),
		\]
		and the continuity equation
		\begin{equation*}			\int_{\mathcal O_\alpha}
			\varrho(t)\varphi\,\dd\bm x
			=
			\int_{\mathcal O_\alpha}
			\varrho_{0,\delta}\varphi\,\dd\bm x
			+
			\int_0^t
			\int_{\mathcal O_\alpha}
			\varrho\mathbf v\cdot\nabla^\phi\varphi
			\,\dd\bm x\dd s
		\end{equation*}
		for every smooth $\varphi$.
		
		Moreover, the coupled momentum--structure formulation obtained in
		Proposition~\ref{thm:limit_m} passes to the limit with
		$
		\epsilon=0,
		p_\delta(\varrho)
		\ \text{replaced by}\
		\overline{p_\delta(\varrho)}.
		$
		In particular, the exact transformed Newtonian stress
		$
		\mathbb S_\kappa^{\zeta^*}(\nabla^\phi\mathbf v)
		$
		is retained, the interface penalty remains normal-only, and the Navier
		friction term retains the full transformed normal, tangential projection,
		and surface Jacobian.
		
		Finally,
		\begin{equation*}			\widehat{\mathbb E}
			\left[
			\sup_{0\le t\le T}
			\mathcal E_\delta(t)^p
			+
			\mathcal D_\delta(T)^p
			\right]
			\le
			C_{p,T,l,\delta}.
		\end{equation*}
	\end{proposition}
	
	\begin{proof}
		The purpose of this preliminary passage is to isolate the pressure defect
		before trying to identify it. The artificial diffusion in the continuity
		equation and its companion correction in the momentum equation vanish by
		Lemma~\ref{app:lem_unweighted_diffusion} and the energy estimate. Since
		$l$ and $\delta$ are fixed, the mollified density occurring in the
		inertial terms is compact, while the flow and graph coefficients converge
		strongly. Hence the convective, viscous, shell, penalty, and Navier terms
		pass by strong--weak convergence in their natural dualities. No such
		argument is available for $p_\delta(\varrho_\epsilon)$ under weak density
		convergence, so it is deliberately retained as
		$\overline{p_\delta(\varrho)}$ in \eqref{eq:pressure_bar}. The
		energy-moment bound follows from Fatou's lemma and weak lower
		semicontinuity. The next step identifies precisely this single remaining
		defect.
	\end{proof}

	The only nonlinear defect left by the preliminary passage is the weak
	pressure limit $\overline{p_\delta(\varrho)}$. Propositions
	\ref{app:prop_effective_flux} and \ref{app:prop_oscillation_defect}, followed
	by Proposition~\ref{app:thm_strong_density}, give
	\[
	\varrho_\epsilon\to\varrho
	\quad\text{strongly in }L^1((0,T)\times\mathcal O_\alpha),
	\]
	together with strong $L^1$ convergence of $\varrho_\epsilon^\gamma$ and
	$\varrho_\epsilon^\beta$. Hence
	\begin{equation}
		\label{eq:eps_pressure_identified}
		\overline{p_\delta(\varrho)}
		=p_\delta(\varrho)=a\varrho^\gamma+\delta\varrho^\beta.
	\end{equation}
	
	We now collect the identified limits and record the $\epsilon\to0$ system.
	
	Combining Proposition~\ref{thm:limit_eps_preliminary} with
	\eqref{eq:eps_pressure_identified} gives the desired limit statement.
	
	\begin{proposition}[Limit $\epsilon\to0$]
		\label{thm:limit_eps}
		The limiting variables
		$
		(
		\varrho,\mathbf v,\zeta,\zeta^*,w,\phi,W
		)
		$
		satisfy, $\widehat{\mathbb P}$-almost surely,
		
		\begin{enumerate}[label=(\arabic*)]
			
			\item
			\[
			\partial_t\zeta=w,
			\qquad
			\zeta^*=\zeta
			\quad\text{on }[0,\tau^\zeta).
			\]
			
			\item
			For every
			$\varphi\in C^\infty(\overline{\mathcal O}_\alpha)$,
			\[
			\int_{\mathcal O_\alpha}
			\varrho(t)\varphi\,\dd\bm x
			=
			\int_{\mathcal O_\alpha}
			\varrho_{0,\delta}\varphi\,\dd\bm x
			+
			\int_0^t
			\int_{\mathcal O_\alpha}
			\varrho\mathbf v\cdot\nabla^\phi\varphi
			\,\dd\bm x\dd s.
			\]
			
			\item
			The coupled momentum--structure identity is
			\begin{align*}				& 
				\int_{\mathcal O_\alpha}[\varrho]_l\mathbf v(t)\cdot\mathbf q(t)\,\dd\bm x+\int_\Gamma w(t)\psi(t)\,\dd\bm y
				=\int_{\mathcal O_\alpha}[\varrho_{0,\delta}]_l\mathbf v_{0,\delta}\cdot\mathbf q(0)\,\dd\bm x+\int_\Gamma w_0\psi(0)\,\dd\bm y \notag \\
				& \quad+\int_0^t\!\!\int_{\mathcal O_\alpha}\Bigl([\varrho]_l\mathbf v\cdot\partial_t\mathbf q
				+[\varrho]_l\mathbf v\otimes\mathbf v:\nabla^\phi\mathbf q+p_\delta(\varrho)\Div^\phi\mathbf q
				-\mathbb S_\kappa^{\zeta^*}(\nabla^\phi\mathbf v):\nabla^\phi\mathbf q\Bigr)\,\dd\bm x\dd s \notag \\
				& \quad+\int_0^t\!\!\int_\Gamma\Bigl(w\partial_t\psi-\Delta_\Gamma^\chi\zeta\,\Delta_\Gamma^\chi\psi
				-\nu_s\nabla_\Gamma^\chi w\cdot\nabla_\Gamma^\chi\psi\Bigr)\,\dd\bm y\dd s \notag \\
				& \quad-\frac1\delta\int_0^t\!\!\int_{T^\delta_{\zeta^*}}
				((\mathbf v-w\mathbf e_z)\cdot\mathbf n^{\zeta^*,\phi}) ((\mathbf q-\psi\mathbf e_z)\cdot\mathbf n^{\zeta^*,\phi})\,\dd\bm x\dd s \notag \\
				& \quad-\iota\int_0^t\!\!\int_{\Gamma_{\zeta^*}}J_\Gamma^\phi[\zeta^*] \left(\mathbf v-w\mathbf e_z\circ\Phi_{\zeta^*}^{-1}\right)_{\tau^{\zeta^*,\phi}} \cdot\left(\mathbf q-\psi\mathbf e_z\circ\Phi_{\zeta^*}^{-1}\right)_{\tau^{\zeta^*,\phi}}\,\dd S\dd s. 
			\end{align*}
			for all smooth admissible test pairs $(\mathbf q,\psi)$.
			
			\item
			For every $p\ge1$,
			\[
			\widehat{\mathbb E}\left[
			\sup_{0\le t\le T}\mathcal E_\delta(t)^p
			+\mathcal D_\delta(T)^p\right]
			\le C_{p,T,l,\delta}.
			\]
			The normal-penalty and Navier-friction bounds are the corresponding
			terms of $\mathcal D_\delta$.
		\end{enumerate}
	\end{proposition}

	\medskip

	\section{Spatial-regularization limit \texorpdfstring{$l\to0$}{l to 0}}
	
	This section removes the spatial mollification in the inertial coefficient at fixed $\delta>0$, identifies the physical mass flux, and then passes the effective-viscous-flux argument to the limit $l\to0$.

	We next let $l\to0$ with $\delta>0$ fixed.  The only $l$-dependence is in the
	inertial factors $[\varrho_l]_l\mathbf v_l$ and
	$[\varrho_l]_l\mathbf v_l\otimes\mathbf v_l$, whereas the continuity equation
	contains the physical mass flux $\varrho_l\mathbf v_l$.  The corresponding
	Friedrichs commutators vanish in a negative Sobolev space; the remaining
	pressure argument is the same localized effective-flux mechanism as in the
	previous section.  The mollifier estimates and commutator argument are given
	below in Subsection~\ref{app:spatial_reg}.
	
	\subsection{Uniform estimates and compactness}
	
	Although the statement of Proposition~\ref{thm:limit_eps} allows the constant
	to depend on $l$, inspection of the energy proof shows that the
	approximation-level energy constants are uniform in the mollification
	scale.  Indeed, spatial convolution is $L^q$-contractive and all
	mollified inertial terms were introduced in the energy-consistent form.
	Hence the following estimate is uniform in $l$.
	
	\begin{lemma}[Uniform $l$-independent estimates]
		
		For every $p\ge1$ there exists
		$C=C(p,T,\delta,\alpha)$, independent of $l$, such that
		\begin{equation}
			\label{eq:l_energy_moment}
			\sup_{l>0}
			\mathbb E_l\left[
			\sup_{0\le t\le T}\mathcal E_l(t)^p
			+
			\mathcal D_l(T)^p
			\right]
			\le C.
		\end{equation}
		Here $\mathcal E_l$ and $\mathcal D_l$ are the quantities in
		\eqref{eq:N_limit_energy_def}--\eqref{eq:N_limit_dissipation_def} after
		the $\epsilon\to0$ passage, so the artificial-density diffusion terms
		are absent while the artificial pressure, exterior viscosity, normal
		penalty, and Navier friction are retained.  In particular,
		\eqref{eq:l_energy_moment} controls uniformly in $l$ the density in
		$L^\infty_tL^\beta_x$, the velocity in $L^2_tH^1_x$, the shell variables
		in their energy spaces, and both interface dissipations.
	\end{lemma}
	
	\begin{proof}
		The bound follows from the  energy estimate after the
		$\epsilon\to0$ passage.  The only $l$-dependent operation is spatial
		convolution.  Since
		\[
		\|[\varrho_l]_l\|_{L^\beta}
		\le
		\|\varrho_l\|_{L^\beta},
		\]
		all constants in the energy estimate are independent of $l$.  The
		interface terms retain the normal-only penalty and the Navier friction
		dissipation.
	\end{proof}
	
	Since $\partial_t\zeta_l=w_l$, the estimate
	\eqref{eq:l_energy_moment} also yields
	$\sup_l\mathbb E_l\|\zeta_l\|_{W^{1,2}(0,T;H^1(\Gamma))}^{2p}\le C$.
	
	We first quantify the mollification defect and identify the momentum variable used in the $l\to0$ passage.
		Set
	\[
	\mathbf m_l
	:=
	[\varrho_l]_l\mathbf v_l,
	\qquad
	\mathbf j_l
	:=
	\varrho_l\mathbf v_l.
	\]
	The two quantities need not coincide for fixed $l$, but their difference
	vanishes in a negative Sobolev norm.

	\begin{lemma}[Vanishing mollification defect]
		\label{lem:l_momentum_defect}
		Let
		$
		r_\beta
		:=
		\frac{2\beta}{\beta+2}>1.
		$ Then, we have
		\begin{equation*}			\mathbf m_l-\mathbf j_l
			\to0
			\qquad
			\text{in }
			L^2(0,T;W^{-1,r_\beta}(\mathcal O_\alpha))
		\end{equation*}
		in probability, and almost surely after the Skorokhod representation.
	\end{lemma}
	
	\begin{proof}
		The point is to separate a genuine Friedrichs commutator from the ordinary
		mollifier error:
		\[
		\mathbf m_l-\mathbf j_l
		=
		\bigl([\varrho_l]_l\mathbf v_l-\mathcal J_l(\varrho_l\mathbf v_l)\bigr)
		+
		\bigl(\mathcal J_l(\varrho_l\mathbf v_l)-\varrho_l\mathbf v_l\bigr).
		\]
		The first term is $O(l)$ in $L^{r_\beta}$ by the standard Friedrichs
		commutator estimate, using $\varrho_l\in L^\beta$ and
		$\nabla\mathbf v_l\in L^2$. The second is $O(l)$ in
		$W^{-1,r_\beta}$ because mollification approximates the identity in that
		negative norm. The energy bounds make both estimates integrable in time,
		which proves the stated convergence.
	\end{proof}

	The following pressure gain is uniform in $l$ and is used in the momentum
	time-regularity estimate entering the tightness argument. Its proof uses
	the scale-invariant contraction
	$\|[\varrho_l]_l\|_{L^\beta}\le\|\varrho_l\|_{L^\beta}$ rather than
	an $l$-dependent $L^\infty$ estimate.
	
	\begin{proposition}[Higher integrability uniform in $l$]
		\label{app:prop_l_higher_integrability}
		There exists $C_\delta$, independent of $l$, such that
		\begin{equation}
			\label{app:eq_l_higher}
			\sup_{l>0}
			\mathbb E
			\int_0^T\!\!\int_{\mathcal O_\alpha}
			\left(
			a\varrho_l^{\gamma+1}
			+
			\delta\varrho_l^{\beta+1}
			\right)
			\,\dd x\,\dd t
			\le C_\delta .
		\end{equation}
	\end{proposition}
	
	\begin{proof}
		Use the frozen Bogovski\u{\i} test from
		Proposition~\ref{app:prop_higher_integrability}.
		The only modified inertial coefficient is $[\varrho_l]_l$.
		Since $\mathcal J_l$ is an $L^\beta$ contraction, all estimates are
		uniform in $l$, and \eqref{app:eq_l_higher} follows.
	\end{proof}

	We next establish tightness and apply the Jakubowski--Skorokhod representation.
	
	Let
	\[
	s_\beta:=\frac{6\beta}{\beta+6}>2,
	\qquad
	p_\beta:=\frac{2\beta}{\beta+1}>\frac65.
	\]
	The energy bounds imply
	\[
	\mathbf j_l
	\ \text{bounded in }\
	L^2(0,T;L^{s_\beta}(\mathcal O_\alpha)),
	\]
	whereas
	\[
	\mathbf m_l
	\ \text{bounded in }\
	L^\infty(0,T;L^{p_\beta}(\mathcal O_\alpha)).
	\]
	The momentum equation, together with the pressure gain
	\eqref{app:eq_l_higher}, yields a bound for
	$\partial_t\mathbf m_l$ in a sufficiently negative Sobolev space,
	uniformly in $l$.  Since
	\[
	L^{p_\beta}(\mathcal O_\alpha)
	\subset
	H^{-1}(\mathcal O_\alpha)
	\]
	for $\beta>4$, Simon's compactness theorem yields
	\begin{equation}
		\label{eq:l_m_compact}
		\{\mathbf m_l\}
		\quad\text{relatively compact in}\quad
		L^2(0,T;H^{-1}(\mathcal O_\alpha)).
	\end{equation}

	We use the path space
	\[
	\mathcal W
	=
	\mathcal W_\varrho
	\times
	\mathcal W_{\mathbf v}
	\times
	\mathcal W_{\mathbf m}
	\times
	\mathcal W_{\mathbf j}
	\times
	\mathcal W_\zeta
	\times
	\mathcal W_{\zeta^*}
	\times
	\mathcal W_w
	\times
	\mathcal W_\phi
	\times
	\mathcal W_W,
	\]
	where
	\begin{align*}        & \mathcal W_\varrho
		:=
		C_w([0,T];L^\beta(\mathcal O_\alpha)), \;\; \mathcal W_{\mathbf v}
		:=
		\bigl(
		L^2(0,T;H^1(\mathcal O_\alpha)),w
		\bigr), \;\; \mathcal W_{\mathbf m}
		:=
		L^2(0,T;H^{-1}(\mathcal O_\alpha)), \\
		& \mathcal W_{\mathbf j}
		:=
		\bigl(
		L^2(0,T;L^{s_\beta}(\mathcal O_\alpha)),w
		\bigr), \;\; \mathcal W_\zeta
		:=
		C([0,T];H^s(\Gamma))
		\cap
		\bigl(
		L^\infty(0,T;H^2(\Gamma)),w^*
		\bigr), \\
		& \mathcal W_{\zeta^*}
		:=
		C([0,T];H^s(\Gamma))
		\cap
		\bigl(
		L^\infty(0,T;H^2(\Gamma)),w^*
		\bigr), \\
		& \mathcal W_w
		:=
		\bigl(
		L^2(0,T;H^1(\Gamma)),w
		\bigr)
		\cap
		L^2(0,T;L^2(\Gamma)), \\
		& \mathcal W_\phi
		:=
		C^{\vartheta_0}
		([0,T];C^2(\overline{\mathcal O}_\alpha;
		\overline{\mathcal O}_\alpha)), \;\; \mathcal W_W
		:=
		C([0,T];\mathbb R^K), 
	\end{align*}
	with
	$
	s\in(3/2,2),
	0<\vartheta_0<1/2.
	$
	
	\begin{lemma}[Tightness]
		
		The joint laws of
		$
		(
		\varrho_l,
		\mathbf v_l,
		\mathbf m_l,
		\mathbf j_l,
		\zeta_l,
		\zeta_l^*,
		w_l,
		\phi,
		W
		)
		$
		are tight on $\mathcal W$.
	\end{lemma}
	
	\begin{proof}
		The density tightness follows from \eqref{eq:l_energy_moment} and
		\[
		\partial_t\varrho_l
		=
		-\Div^\phi\mathbf j_l.
		\]
		The velocity is tight in the weak $L^2_tH^1_x$ topology.
		The momentum compactness is
		\eqref{eq:l_m_compact}; the physical mass flux is weakly tight in
		$L^2_tL^{s_\beta}_x$.  The structure and stochastic-flow components are
		handled exactly as in Sections~\ref{s6}--\ref{s7}.
	\end{proof}
	
	Define
	\[
	\mathbf U_l
	:=
	(\varrho_l,\mathbf v_l,\mathbf m_l,\mathbf j_l,
	\zeta_l,\zeta_l^*,w_l,\phi_l,W_l),
	\qquad
	\mathbf U
	:=
	(\varrho,\mathbf v,\mathbf m,\mathbf j,
	\zeta,\zeta^*,w,\phi,W).
	\]
	After a Jakubowski--Skorokhod representation and extraction of a
	subsequence,
	\[
	\mathbf U_l\to\mathbf U
	\qquad\text{almost surely in }\mathcal W.
	\]
	We record only the three component convergences used explicitly below:
	\begin{equation}
		\label{eq:l_m_strong}
		\mathbf m_l\to\mathbf m
		\qquad\text{in }L^2(0,T;H^{-1}(\mathcal O_\alpha)),
	\end{equation}
	\begin{equation*}		\zeta_l^*\to\zeta^*
		\qquad\text{in }C([0,T];H^s(\Gamma)),
	\end{equation*}
	\begin{equation*}		\phi_l\to\phi
		\qquad\text{in }C^{\vartheta_0}([0,T];C^2).
	\end{equation*}
	All remaining component convergences are part of the convergence in
	$\mathcal W$ and will be invoked without separate numbering.
	The canonical martingale characterization is passed as in the previous
	sections, so $W$ remains Wiener and $\phi$ is the stochastic flow generated
	by $W$.
	
	\begin{lemma}[Identification of momentum and convection]
		\label{lem:l_momentum_identification}
		We have
		\begin{equation*}			\mathbf j=\mathbf m=\varrho\mathbf v,
		\end{equation*}
		and
		\begin{equation}
			\label{eq:l_convection_identification}
			\mathbf m_l\otimes\mathbf v_l
			\rightharpoonup
			\varrho\mathbf v\otimes\mathbf v
			\qquad
			\text{in }\mathcal D'((0,T)\times\mathcal O_\alpha).
		\end{equation}
	\end{lemma}
	
	\begin{proof}
		The point is to compare two quantities before passing to the nonlinear
		product. The continuity equation gives time control of the physical mass
		flux $\mathbf j_l=\varrho_l\mathbf v_l$, while the momentum equation gives
		time control of $\mathbf m_l=[\varrho_l]_l\mathbf v_l$. Together with the
		energy bounds, this yields compactness of $\mathbf m_l$ in a negative
		Sobolev space. On the other hand, the basic mollifier estimate shows that
		$[\varrho_l]_l-\varrho_l$ is $O(l)$ in a negative norm; paired with the
		$H^1$ velocity bound this is exactly the defect controlled in
		Lemma~\ref{lem:l_momentum_defect}. Hence $\mathbf m_l$ and
		$\mathbf j_l$ have the same distributional limit $\varrho\mathbf v$.
		The strong negative-Sobolev convergence of the momentum can then be paired
		with the weak $H^1$ convergence of $\mathbf v_l$ to identify the
		convective tensor, which proves
		\eqref{eq:l_convection_identification}.
	\end{proof}

	\subsection{Spatial-regularization commutators and effective flux}
	\label{app:spatial_reg}
	
	Here $\delta>0$ is fixed and
	\[
	\mathcal J_l f=[f]_l
	\]
	denotes convolution after a fixed bounded extension from
	$\mathcal O_\alpha$ to $\mathbb R^3$.
	The localization, frozen transformed operators, and effective-flux
	argument are those of Subsection~\ref{app:effective_flux}.
	The only additional issue is the replacement of the mollified inertial
	flux
	\[
	\mathbf m_l:=(\mathcal J_l\varrho_l)\mathbf v_l
	\]
	by the physical mass flux
	\[
	\mathbf j_l:=\varrho_l\mathbf v_l .
	\]
	
	We use the standard mollifier estimates
	\begin{align}
		& \|\mathcal J_l f\|_{L^p}
		\le C\|f\|_{L^p}, \notag \\
		& \|\mathcal J_l f-f\|_{W^{-1,p}}
		\le Cl\|f\|_{L^p},
		\label{app:eq_l_mollifier_negative} 
	\end{align}
	and the Friedrichs commutator bound
	\begin{equation}
		\label{app:eq_l_Friedrichs}
		\|
		(\mathcal J_l f)g-\mathcal J_l(fg)
		\|_{L^r}
		\le
		Cl\|f\|_{L^p}\|\nabla g\|_{L^q},
		\qquad
		\frac1r=\frac1p+\frac1q\le1.
	\end{equation}
	
	Set
	$
	r_\beta:=\frac{2\beta}{\beta+2}.
	$
	Then the energy bounds and
	\eqref{app:eq_l_mollifier_negative}--\eqref{app:eq_l_Friedrichs} give
	\begin{equation}
		\label{app:eq_l_momentum_defect}
		\|\mathbf m_l-\mathbf j_l\|_
		{L^2(0,T;W^{-1,r_\beta}(\mathcal O_\alpha))}
		\to0
	\end{equation}
	in probability as $l\to0$.

	Let
	$
	p_\beta:=\frac{2\beta}{\beta+1}.
	$
	The kinetic-energy estimate yields
	\begin{equation*}		\sup_{l>0}
		\|\mathbf m_l\|_{L^\infty(0,T;L^{p_\beta}(\mathcal O_\alpha))}
		\le C.
	\end{equation*}
	Since $\beta>4$, one has $p_\beta>6/5$ and hence
	\[
	L^{p_\beta}(\mathcal O_\alpha)
	\subset
	H^{-1}(\mathcal O_\alpha).
	\]
	Moreover, the momentum equation and
	\eqref{app:eq_l_higher} give, for some $R_0>3$,
	\begin{equation*}		\sup_{l>0}
		\mathbb E
		\|
		\partial_t\mathbf m_l
		\|_{L^1(0,T;W^{-R_0,1}(\mathcal O_\alpha))}
		<\infty.
	\end{equation*}
	Therefore, by Simon's compactness theorem,
	\begin{equation*}		\{\mathbf m_l\}_l
		\quad\text{is relatively compact in}\quad
		L^2(0,T;H^{-1}(\mathcal O_\alpha)).
	\end{equation*}
	
	After the Skorokhod representation, assume
	\[
	\varrho_l\to\varrho
	\quad\text{in }C_w([0,T];L^\beta),
	\qquad
	\mathbf v_l\rightharpoonup\mathbf v
	\quad\text{in }L^2(0,T;H^1),
	\]
	and
	\[
	\mathbf m_l\to\mathbf m
	\quad\text{strongly in }L^2(0,T;H^{-1}).
	\]
	The continuity equation gives compactness of $\varrho_l$ in a negative
	Sobolev space.  Together with
	\eqref{app:eq_l_momentum_defect}, this identifies
	\begin{equation*}		\mathbf m=\varrho\mathbf v.
	\end{equation*}
	Consequently,
	\begin{equation*}		\mathbf m_l\otimes\mathbf v_l
		\rightharpoonup
		\varrho\mathbf v\otimes\mathbf v
		\qquad
		\text{in }\mathcal D'.
	\end{equation*}

	We now remove the inertial mollification in the effective-viscous-flux argument and then recover strong convergence of the density.
	
	The only modification of the effective-flux argument is the inertial
	commutator caused by the factor $\mathcal J_l\varrho_l$.
	By \eqref{app:eq_l_Friedrichs} and
	\eqref{app:eq_l_momentum_defect}, this error vanishes in the negative
	Sobolev spaces required by the frozen local-potential test.  Hence the
	argument of Proposition~\ref{app:prop_effective_flux} applies without
	further change.
	
	\begin{proposition}[Weak continuity of the effective viscous flux]
		\label{app:prop_l_effective_flux}
		For every
		\[
		K\subset\mathcal O_\alpha,
		\qquad
		\psi\in C_c^\infty((0,T)\times K),
		\qquad
		k\in\mathbb N,
		\]
		one has
		\begin{align}
			& \lim_{l\to0}
			\int_0^T\!\!\int_{\mathcal O_\alpha}
			\psi
			\left[
			p_\delta(\varrho_l)
			-
			\bigl(
			\lambda_\kappa^{\zeta_l^*}
			+
			2\mu_\kappa^{\zeta_l^*}
			\bigr)
			\Div^{\phi_l}\mathbf v_l
			\right]
			T_k(\varrho_l)
			\,\dd x\,\dd t
			\notag \\
			& \qquad=
			\int_0^T\!\!\int_{\mathcal O_\alpha}
			\psi
			\left[
			\overline{p_\delta(\varrho)}
			-
			\bigl(
			\lambda_\kappa^{\zeta^*}
			+
			2\mu_\kappa^{\zeta^*}
			\bigr)
			\Div^\phi\mathbf v
			\right]
			\overline{T_k(\varrho)}
			\,\dd x\,\dd t .
			\label{app:eq_l_effective_flux} 
		\end{align}
	\end{proposition}
	
	\begin{proof}
		Fix a bounded-flow localization and the deterministic frozen partition of
		Subsection~\ref{app:effective_flux}. The pressure, viscous, interface, and
		freezing remainders satisfy the estimates of
		Proposition~\ref{app:prop_effective_flux}. The only new term is the inertial
		mollification remainder $\mathcal R_l^{\rm in}$, and
		\[
		|\mathcal R_l^{\rm in}|
		\le C
		\left(
		\|\mathcal J_l(\varrho_l\mathbf v_l)-\varrho_l\mathbf v_l\|_{L^2_tH^{-1}_x}
		+
		\|\mathbf m_l-\varrho_l\mathbf v_l\|_{L^2_tH^{-1}_x}
		\right)
		\longrightarrow0
		\]
		by \eqref{app:eq_l_Friedrichs} and
		\eqref{app:eq_l_momentum_defect}. Thus the frozen effective-flux identity
		passes to $l\to0$; removing the freezing scale and the bounded-flow
		localization yields \eqref{app:eq_l_effective_flux}.
	\end{proof}
	
	At the $l$-level the density satisfies the renormalized continuity
	equation
	\[
	\partial_t b(\varrho_l)
	+
	\Div^{\phi_l}
	\bigl(
	b(\varrho_l)\mathbf v_l
	\bigr)
	+
	\bigl(
	b'(\varrho_l)\varrho_l-b(\varrho_l)
	\bigr)
	\Div^{\phi_l}\mathbf v_l
	=
	0
	\]
	for the usual class of renormalizations.
	
	\begin{proposition}[Strong convergence of the density as $l\to0$]
		\label{app:thm_l_strong_density}
		Almost surely,
		\begin{equation*}			\varrho_l
			\to
			\varrho
			\qquad
			\text{strongly in }
			L^1((0,T)\times\mathcal O_\alpha).
		\end{equation*}
		Moreover,
		\[
		\varrho_l\to\varrho
		\qquad
		\text{strongly in }
		L^p((0,T)\times\mathcal O_\alpha)
		\quad
		\text{for every }1\le p<\beta+1,
		\]
		and
		\begin{equation}
			\label{app:eq_l_pressure_identification}
			\overline{p_\delta(\varrho)}
			=
			a\varrho^\gamma+\delta\varrho^\beta.
		\end{equation}
	\end{proposition}
	
	\begin{proof}
		From \eqref{app:eq_l_effective_flux}, the renormalized continuity equation,
		and monotonicity of $p_\delta$,
		\[
		\operatorname{osc}_{\beta+1}[\varrho_l\to\varrho]<\infty,
		\qquad
		\overline{\varrho\log\varrho}=\varrho\log\varrho .
		\]
		Consequently,
		\[
		\varrho_l\to\varrho
		\qquad\text{a.e. and strongly in }L^1.
		\]
		The bound \eqref{app:eq_l_higher} gives uniform integrability in every
		$L^p$, $1\le p<\beta+1$. Vitali's theorem yields the stated $L^p$
		convergence and \eqref{app:eq_l_pressure_identification}.
	\end{proof}
	
	\subsection{Limit passage and pressure identification}
	
	Propositions~\ref{app:prop_l_effective_flux} and
	\ref{app:thm_l_strong_density} give, almost surely,
	\[
	\varrho_l\to\varrho
	\quad\text{strongly in }L^1((0,T)\times\mathcal O_\alpha),
	\]
	and in $L^p$ for every $1\le p<\beta+1$. In particular,
	\[
	\overline{p_\delta(\varrho)}
	=p_\delta(\varrho)
	=a\varrho^\gamma+\delta\varrho^\beta .
	\]
	The remaining inertial, viscous, shell and interface terms pass by
	\eqref{eq:l_m_strong}--\eqref{eq:l_convection_identification} together
	with the strong convergence of the graph and flow coefficients.

	
	\begin{proposition}[Limit $l\to0$]
		
		The limiting variables
		$
		(
		\varrho,\mathbf v,\zeta,\zeta^*,w,\phi,W
		)
		$
		satisfy, almost surely,
		
		\begin{enumerate}[label=(\arabic*)]
			
			\item
			\[
			\partial_t\zeta=w,
			\qquad
			\zeta^*=\zeta
			\quad\text{on }[0,\tau^\zeta).
			\]
			
			\item
			For every smooth scalar test function $\varphi$,
			\begin{equation*}				\int_{\mathcal O_\alpha}
				\varrho(t)\varphi\,\dd\bm x
				=
				\int_{\mathcal O_\alpha}
				\varrho_{0,\delta}\varphi\,\dd\bm x
				+
				\int_0^t
				\int_{\mathcal O_\alpha}
				\varrho\mathbf v\cdot\nabla^\phi\varphi
				\,\dd\bm x\dd s.
			\end{equation*}
			
			\item
			For every smooth admissible test pair
			$(\mathbf q,\psi)$,
			\begin{align*}				& \int_{\mathcal O_\alpha}\varrho\mathbf v(t)\cdot\mathbf q(t)\,\dd\bm x+\int_\Gamma w(t)\psi(t)\,\dd\bm y
				=\int_{\mathcal O_\alpha}\varrho_{0,\delta}\mathbf v_{0,\delta}\cdot\mathbf q(0)\,\dd\bm x+\int_\Gamma w_0\psi(0)\,\dd\bm y \notag \\
				& \quad+\int_0^t\!\!\int_{\mathcal O_\alpha}\Bigl(\varrho\mathbf v\cdot\partial_t\mathbf q
				+\varrho\mathbf v\otimes\mathbf v:\nabla^\phi\mathbf q+p_\delta(\varrho)\Div^\phi\mathbf q
				-\mathbb S_\kappa^{\zeta^*}(\nabla^\phi\mathbf v):\nabla^\phi\mathbf q\Bigr)\,\dd\bm x\dd s \notag \\
				& \quad+\int_0^t\!\!\int_\Gamma\Bigl(w\partial_t\psi-\Delta_\Gamma^\chi\zeta\,\Delta_\Gamma^\chi\psi
				-\nu_s\nabla_\Gamma^\chi w\cdot\nabla_\Gamma^\chi\psi\Bigr)\,\dd\bm y\dd s \notag \\
				& \quad-\frac1\delta\int_0^t\!\!\int_{T^\delta_{\zeta^*}}
				((\mathbf v-w\mathbf e_z)\cdot\mathbf n^{\zeta^*,\phi}) ((\mathbf q-\psi\mathbf e_z)\cdot\mathbf n^{\zeta^*,\phi})\,\dd\bm x\dd s \notag \\
				& \quad-\iota\int_0^t\!\!\int_{\Gamma_{\zeta^*}}J_\Gamma^\phi[\zeta^*] \left(\mathbf v-w\mathbf e_z\circ\Phi_{\zeta^*}^{-1}\right)_{\tau^{\zeta^*,\phi}}
				\cdot\left(\mathbf q-\psi\mathbf e_z\circ\Phi_{\zeta^*}^{-1}\right)_{\tau^{\zeta^*,\phi}}\,\dd S\dd s. 
			\end{align*}
			
			\item
			For every $p\ge1$,
			\begin{equation}
				\label{eq:l_limit_energy_moment}
				\mathbb E\left[
				\sup_{0\le t\le T}\mathcal E_\delta(t)^p
				+\mathcal D_\delta(T)^p\right]
				\le C_{p,T,\delta}.
			\end{equation}
			Here the kinetic part is
			$\frac12\int_{\mathcal O_\alpha}\varrho|\mathbf v|^2\,\dd\bm x$;
			the normal-penalty and Navier-friction bounds are contained in
			$\mathcal D_\delta$.
		\end{enumerate}
	\end{proposition}
	
	\begin{proof}
		The only remaining point is the kinetic energy.  By Proposition~\ref{app:thm_l_strong_density},
		\[
		[\varrho_l]_l
		\to
		\varrho
		\qquad
		\text{strongly in }L^1,
		\]
		while
		\[
		\mathbf m_l
		=
		[\varrho_l]_l\mathbf v_l
		\to
		\varrho\mathbf v
		\]
		in the momentum topology.  The convex functional
		\[
		(\rho,\mathbf m)
		\longmapsto
		\begin{cases}
			|\mathbf m|^2/\rho,&\rho>0,\\
			0,&\rho=0,\ \mathbf m=0,\\
			+\infty,&\text{otherwise},
		\end{cases}
		\]
		is lower semicontinuous.  Hence
		\[
		\int
		\varrho|\mathbf v|^2
		\le
		\liminf_{l\to0}
		\int
		[\varrho_l]_l|\mathbf v_l|^2.
		\]
		All other energy and dissipation terms are lower semicontinuous under the
		convergences already established, so
		\eqref{eq:l_limit_energy_moment} follows from
		\eqref{eq:l_energy_moment} and Fatou's lemma.
	\end{proof}

	\medskip

\section{Artificial-pressure and penalty limit \texorpdfstring{$\delta\to0$}{delta to 0}}
	\label{s9}

	This final approximation section removes the artificial pressure, the exterior-viscosity regularization, and the normal interface penalty, and identifies the limiting moving-domain equations term by term.
	The artificial-pressure exponent
	$\beta>\max\{6,\gamma\}$ and $a_\beta=\frac12-\frac1\beta$ are fixed in
	\eqref{eq:global_beta_choice}. The exterior-viscosity scaling is needed only
	at this final stage. Choose $p_{\rm ext}\in(3/2,\gamma)$ and set
	\[
	m_{\rm ext}:=\min\left\{a_\beta^2,\,
	a_\beta\left(1-\frac1\gamma\right)\right\},\qquad
	\vartheta_{\rm ext}:=
	\frac{\frac1{p_{\rm ext}}-\frac1\gamma}{1-\frac1\gamma}.
	\]
	Then fix
	\begin{equation}
		\label{eq:global_nu_choice}
		0<\nu_\kappa<
		\nu_*:=\min\left\{a_\beta,\,m_{\rm ext}\vartheta_{\rm ext}\right\},
		\qquad
		\kappa_\delta:=\delta^{\nu_\kappa},\qquad
		h_\delta:=\delta^{a_\beta}.
	\end{equation}
	The restriction $\beta>6$ is used in the exterior boundary-pressure estimate,
	whereas the second restriction in \eqref{eq:global_nu_choice} enters the
	exterior convective estimates in Subsection~\ref{app:delta_limit}. 
	The exterior viscosity now degenerates, the artificial pressure
	must vanish, and the normal penalty becomes singular. The proof combines the
	exterior-mass estimate, normal-trace recovery, pressure equi-integrability and
	local effective flux, and geometry-adapted tests. 

	\subsection{Uniform estimates and compactness}
	
	The stochastic energy estimate obtained after the $l\to0$ passage is
	uniform in $\delta$.
	
	\begin{lemma}[Uniform estimates independent of $\delta$]
		\label{lem:delta_uniform_bounds}
		
		For every $p\ge1$ there is a constant
		$
		C=C(p,T,\alpha)
		$
		independent of $\delta$ such that
		\begin{align}
			& \sup_{\delta>0}
			\mathbb E_\delta
			\left[
			\sup_{0\le t\le T}
			\left(
			\|\sqrt{\varrho_\delta}\mathbf v_\delta\|_{L^2}^2
			+
			\|\varrho_\delta\|_{L^\gamma}^{\gamma}
			+
			\delta\|\varrho_\delta\|_{L^\beta}^{\beta}
			+
			\|\zeta_\delta\|_{H^2(\Gamma)}^2
			+
			\|w_\delta\|_{L^2(\Gamma)}^2
			\right)^p
			\right]
			\le C, \notag \\
			& \sup_{\delta>0}
			\mathbb E_\delta
			\left[
			\left(
			\int_0^T
			\int_{\mathcal O_\alpha}
			\mathbb S_{\kappa_\delta}^{\zeta_\delta^*}
			(\nabla^\phi\mathbf v_\delta):
			\nabla^\phi\mathbf v_\delta
			\,\dd\bm x\dd t
			\right)^p
			\right]
			\le C, \notag \\
			& \sup_{\delta>0}
			\mathbb E_\delta
			\left[
			\|w_\delta\|_{L^2(0,T;H^1(\Gamma))}^{2p}
			\right]
			\le C, \notag \\
			& \sup_{\delta>0}
			\mathbb E_\delta
			\left[
			\left(
			\frac1\delta
			\int_0^T
			\int_{T^\delta_{\zeta_\delta^*}}
			\left|
			(
			\mathbf v_\delta-w_\delta\mathbf e_z
			)
			\cdot
			\mathbf n^{\zeta_\delta^*,\phi}
			\right|^2
			\,\dd\bm x\dd t
			\right)^p
			\right]
			\le C,
			\label{eq:delta_normal_penalty} \\
			& \sup_{\delta>0}
			\mathbb E_\delta
			\left[
			\left(
			\iota
			\int_0^T
			\int_{\Gamma_{\zeta_\delta^*}}
			J_\Gamma^\phi[\zeta_\delta^*]
			\left|
			\left(
			\mathbf v_\delta
			-
			w_\delta\mathbf e_z\circ\Phi_{\zeta_\delta^*}^{-1}
			\right)_{\tau^{\zeta_\delta^*,\phi}}
			\right|^2
			\,\dd S\dd t
			\right)^p
			\right]
			\le C. \notag 
		\end{align}
		Furthermore,
		\begin{equation*}			\sup_{\delta>0}
			\mathbb E_\delta
			\|\zeta_\delta\|_{W^{1,2}(0,T;H^1(\Gamma))}^{2p}
			\le C.
		\end{equation*}
	\end{lemma}
	
	\begin{proof}
		These estimates are the $\delta$-uniform part of
		\eqref{eq:l_limit_energy_moment}.  On the physical fluid region the
		extended viscosities equal the physical constants $\mu,\lambda$; outside,
		the coercivity is weighted by $\kappa_\delta$.  The interface penalty in
		\eqref{eq:delta_normal_penalty} is the normal-only penalty inherited from
		the approximation scheme.
	\end{proof}
	
	For the final $\delta$-limit we do not require a separate strong
	time-compactness estimate for $w_\delta$. The energy bounds give
	\[
	w_\delta\quad\text{bounded in }L^2(0,T;H^1(\Gamma))
	\cap L^\infty(0,T;L^2(\Gamma)),
	\]
	and we retain the corresponding weak/weak-star topology in the
	Jakubowski representation. This is sufficient for the shell terms and for
	the later trace identification.
	
	We use a Jakubowski path space for
	$
	\varrho_\delta,
	\varrho_\delta\mathbf v_\delta,
	\zeta_\delta,
	\zeta_\delta^*,
	w_\delta,
	\phi,
	W,
	$
	with the weak or weakly continuous topologies dictated by the uniform
	estimates and the equations. Local velocity compactness is not built into
	this representation; it is derived after the geometry has converged, in
	Proposition~\ref{app:prop_delta_time_compactness}.
	
	\begin{lemma}[Tightness and representation]
		\label{lem:delta_tightness}
		The corresponding joint laws are tight.  Set
		\[
		\mathbf U_\delta
		:=
		(\varrho_\delta,\varrho_\delta\mathbf v_\delta,
		\zeta_\delta,\zeta_\delta^*,w_\delta,\phi_\delta,W_\delta),
		\]
		and let $\mathbf U$ denote the corresponding limiting tuple.  After a
		Jakubowski--Skorokhod representation and subsequence extraction,
		$\mathbf U_\delta\to\mathbf U$ almost surely in the joint path space
		described above.  We record the component convergences used later:
		\begin{equation}
			\label{eq:delta_zeta_strong}
			\zeta_\delta\to\zeta
			\qquad\text{in }C([0,T];H^s(\Gamma)),\qquad s\in(3/2,2),
		\end{equation}
		\begin{equation}
			\label{eq:delta_w_weak}
			w_\delta\rightharpoonup w
			\qquad\text{in }L^2(0,T;H^1(\Gamma)),
		\end{equation}
		and
		\[
		\varrho_\delta\to\varrho\quad\text{in }C_w([0,T];L^\gamma(\mathcal O_\alpha)),
		\qquad
		\varrho_\delta\mathbf v_\delta\rightharpoonup\varrho\mathbf v
		\quad\text{in }\mathcal D'.
		\]
		The limiting process $W$ is a Wiener process for the canonical limiting
		filtration, and $\phi$ is the stochastic flow generated by $W$.
	\end{lemma}

	\begin{proof}
		The energy estimate, the continuity equation, and the momentum equation
		give tightness of the density and momentum in the standard weak and
		negative-Sobolev topologies. The shell energy bound and
		$\partial_t\zeta_\delta=w_\delta$ give compactness of $\zeta_\delta$
		in $C([0,T];H^s(\Gamma))$, $s<2$. The bounds for $w_\delta$ give
		tightness in the weak $L^2_tH^1_\Gamma$ and weak-star
		$L^\infty_tL^2_\Gamma$ topologies. The flow estimates give tightness of $\phi_\delta$,
		while tightness of $W_\delta$ is standard. These bounds close the product
		Jakubowski space without using local $H^1$ compactness of the velocity.
		The latter is recovered only after representation, once the limiting graph
		is available; see Proposition~\ref{app:prop_delta_time_compactness}. The
		Wiener property and the stochastic-flow equation follow by passing the
		canonical martingale characterizations as in the preceding limits.
	\end{proof}
	
	From now on we work on the represented probability space and omit hats.

	\subsection{Geometry and traces for \texorpdfstring{$H^2$}{H2} graph interfaces}
	\label{app:geometry}
	
	For fixed $M,h_-,h_+>0$, define
	\begin{equation*}		\mathfrak G(M,h_-,h_+)
		:=
		\left\{
		\zeta\in H^2(\Gamma):
		\|\zeta\|_{H^2}\le M,\quad
		h_-\le1+\zeta\le h_+
		\right\},
	\end{equation*}
	where $0<h_-<h_+<2+\alpha^{-1}$. On bounded-flow localization events we
	also assume
	\begin{equation}
		\label{app:eq_D_flow_class}
		\|\phi\|_{C^2(\overline{\mathcal O}_\alpha)}
		+
		\|\phi^{-1}\|_{C^2(\overline{\mathcal O}_\alpha)}
		\le R.
	\end{equation}
	All constants below are uniform on these classes.

	\subsubsection{Graph geometry}
	
	For $\zeta\in\mathfrak G(M,h_-,h_+)$ we use the graph notation
	$\Phi_\zeta$, $\mathbf N^\zeta$, $J_\zeta$, and $\mathbf n^\zeta$
	introduced in Subsection~\ref{sec:physical_model}.
	
	\begin{lemma}[Compactness of the graph geometry]
		\label{app:lem_graph_compactness}
		Let
		$\zeta_n,\zeta\in\mathfrak G(M,h_-,h_+)$ and assume, for some
		$s\in(3/2,2)$,
		\begin{equation}
			\label{app:eq_D_zeta_Hs_conv}
			\zeta_n\to\zeta
			\qquad
			\text{strongly in }H^s(\Gamma).
		\end{equation}
		Then
		\begin{align*}			& \zeta_n\to\zeta
			\qquad\text{uniformly on }\Gamma, \\
			& \nabla_\Gamma\zeta_n
			\to\nabla_\Gamma\zeta, \;\; J_{\zeta_n}
			\to J_\zeta, \;\;\mathbf n^{\zeta_n}
			\to \mathbf n^\zeta,
			\;\;\text{strongly in }L^q(\Gamma),
		\end{align*}
		with $1\le q<\infty$. Moreover,
		\begin{equation*}			\sup_n
			\left(
			\|J_{\zeta_n}\|_{L^q}
			+
			\|\mathbf n^{\zeta_n}\|_{L^\infty}
			\right)
			\le C(M,q).
		\end{equation*}
	\end{lemma}
	
	\begin{proof}
		By interpolation with the uniform $H^2$ bound,
		$\zeta_n\to\zeta$ strongly in $H^{s'}(\Gamma)$ for every
		$s<s'<2$.  The two-dimensional Sobolev embeddings then give uniform
		convergence and strong convergence in $W^{1,q}$ for every finite $q$.
		The remaining statements follow from the continuity of
		$
		p\mapsto\sqrt{1+|p|^2}$ and $
		p\mapsto\frac{(-p,1)}{\sqrt{1+|p|^2}}.
		$
	\end{proof}

	\begin{lemma}[Uniform transformed geometry]
		\label{app:lem_transformed_geometry}
		Under \eqref{app:eq_D_flow_class},
	$			0<c_R
			\le
			J_\Gamma^\phi[\zeta]
			\le
			C_R
		$
		for   $\zeta\in\mathfrak G(M,h_-,h_+)$.  If
		$\phi_n\to\phi$ in $C^1(\overline{\mathcal O}_\alpha)$ and
		\eqref{app:eq_D_zeta_Hs_conv} holds, then for every finite $q$,
		\begin{align*}			& n^{\zeta_n,\phi_n}
			\to
			n^{\zeta,\phi},
			\;\; J_\Gamma^{\phi_n}[\zeta_n]
			\to
			J_\Gamma^\phi[\zeta], \;\; P_\tau^{\zeta_n,\phi_n}
			\to
			P_\tau^{\zeta,\phi},
			\qquad\text{strongly in }L^q(\Gamma). 
		\end{align*}
	\end{lemma}
	
	\begin{proof}
		The singular values of $\operatorname{cof}\nabla\phi$ are uniformly
		bounded above and below under \eqref{app:eq_D_flow_class}.  The result
		therefore follows from Lemma~\ref{app:lem_graph_compactness}.
	\end{proof}
	
	Set
	\begin{equation*}		\mathcal J^{\zeta,\phi}
		:=
		J_\zeta J_\Gamma^\phi[\zeta].
	\end{equation*}
	Then
	\begin{equation*}		\int_{\Gamma_\zeta}
		J_\Gamma^\phi[\zeta]F\,\dd S
		=
		\int_\Gamma
		\mathcal J^{\zeta,\phi}
		(F\circ\Phi_\zeta)\,\dd\bm y.
	\end{equation*}
	
	\begin{corollary}[Convergence of surface weights]
		
		Under the assumptions of
		Lemma~\ref{app:lem_transformed_geometry},
		\begin{equation}
			\label{app:eq_D_total_weight_conv}
			\mathcal J^{\zeta_n,\phi_n}
			\to
			\mathcal J^{\zeta,\phi}
			\qquad
			\text{strongly in }L^q(\Gamma)
		\end{equation}
		for every finite $q$.
	\end{corollary}

	\subsubsection{Moving graph traces}
	
	For smooth $u$ define
	\begin{equation*}		\Tr_\zeta u(\bm y)
		:=
		u(\bm y,1+\zeta(\bm y)).
	\end{equation*}
	
	The one-dimensional trace theorem in the vertical variable gives,
	uniformly in $\zeta\in\mathfrak G(M,h_-,h_+)$,
	\begin{equation*}		\|\Tr_\zeta u\|_{L^2(\Gamma)}
		\le
		C_\theta\|u\|_{H^\theta(\mathcal O_\alpha)},
		\qquad
		\theta>\frac12.
	\end{equation*}
	Moreover, for every $2\le q<4$,
	\begin{equation*}		\|\Tr_\zeta u\|_{L^q(\Gamma)}
		\le
		C_q\|u\|_{H^1(\mathcal O_\alpha)}.
	\end{equation*}
	Interpolation yields, for every $\theta>1/2$, some $q_\theta>2$ such
	that
	\begin{equation}
		\label{app:eq_D_fractional_trace_Lq}
		\|\Tr_\zeta u\|_{L^q(\Gamma)}
		\le
		C_{\theta,q}\|u\|_{H^\theta(\mathcal O_\alpha)},
		\qquad
		2\le q<q_\theta.
	\end{equation}
	
	Since $J_\zeta$ is bounded in every finite $L^p(\Gamma)$, these estimates
	also imply
	\begin{equation*}		\int_{\Gamma_\zeta}|u|^2\,\dd S
		\le
		C(M,\theta)\|u\|_{H^\theta(\mathcal O_\alpha)}^2,
		\qquad
		\theta>\frac12,
	\end{equation*}
	and likewise with the additional factor
	$J_\Gamma^\phi[\zeta]$.
	
	\begin{proposition}[Continuity of moving graph traces]
		\label{app:prop_moving_trace_convergence}
		
		Let $\theta>1/2$, assume
		\eqref{app:eq_D_zeta_Hs_conv}, and let
		\[
		u_n\to u
		\qquad
		\text{strongly in }H^\theta(\mathcal O_\alpha).
		\]
		Then, for some $q>2$,
		\begin{equation*}			\Tr_{\zeta_n}u_n
			\to
			\Tr_\zeta u
			\qquad
			\text{strongly in }L^q(\Gamma).
		\end{equation*}
		If also $\phi_n\to\phi$ in $C^1(\overline{\mathcal O}_\alpha)$, then
		\begin{equation*}			\int_\Gamma
			\mathcal J^{\zeta_n,\phi_n}
			\left|
			\Tr_{\zeta_n}u_n-\Tr_\zeta u
			\right|^2
			\,\dd\bm y
			\to0.
		\end{equation*}
	\end{proposition}
	
	\begin{proof}
		Split
		\[
		\Tr_{\zeta_n}u_n-\Tr_\zeta u
		=
		\Tr_{\zeta_n}(u_n-u)
		+
		\bigl(\Tr_{\zeta_n}u-\Tr_\zeta u\bigr).
		\]
		The first term is controlled by
		\eqref{app:eq_D_fractional_trace_Lq}; the second follows by approximation
		of $u$ by smooth functions and the uniform convergence of $\zeta_n$.
		The weighted convergence follows from
		\eqref{app:eq_D_total_weight_conv}.
	\end{proof}
	
	Consequently,
	\begin{equation*}		P_\tau^{\zeta_n,\phi_n}
		\Tr_{\zeta_n}u_n
		\to
		P_\tau^{\zeta,\phi}
		\Tr_\zeta u
	\end{equation*}
	strongly in the corresponding weighted $L^2$ interface space.

	\subsubsection{Subcritical estimates on graph subdomains}
	
	\begin{lemma}[Uniform subcritical Korn--Sobolev estimate]
		\label{app:lem_subgraph_subcritical_korn}
		Let
		$\zeta\in\mathfrak G(M,h_-,h_+)$ and
		$\mathbf u\in H^1(\mathcal O_\zeta;\mathbb R^3)$ have zero trace on the
		fixed bottom.  For every $q\in(1,2)$,
		\begin{equation}
			\label{app:eq_subcritical_korn}
			\|\mathbf u\|_{W^{1,q}(\mathcal O_\zeta)}
			\le
			C
			\|\mathbb D\mathbf u\|_{L^2(\mathcal O_\zeta)},
			\qquad
			\mathbb D\mathbf u
			:=
			\frac12(\nabla\mathbf u+\nabla\mathbf u^\top).
		\end{equation}
		Consequently, with
		$q^\#=3q/(3-q)$,
		\begin{equation}
			\label{app:eq_subcritical_korn_sobolev}
			\|\mathbf u\|_{L^{q^\#}(\mathcal O_\zeta)}
			\le
			C
			\|\mathbb D\mathbf u\|_{L^2(\mathcal O_\zeta)}.
		\end{equation}
		In particular, every exponent $2\le r<6$ is admissible by choosing
		$q<2$ sufficiently close to $2$.
	\end{lemma}
	
	\begin{proof}
		Fix $q<2$.  By the two-dimensional embedding
		$H^2(\Gamma)\hookrightarrow C^{0,\alpha}(\Gamma)$ for every
		$\alpha<1$, the graphs in $\mathfrak G(M,h_-,h_+)$ have a uniform
		$C^{0,\alpha}$ character, uniform height bounds, and a uniform finite
		covering by graph charts.  We choose $\alpha<1$ sufficiently close to
		one, depending only on $q$.  The weighted Korn inequality on H\"older
		domains \cite{AcostaDuranLombardi2006} then gives a bound for
		$\nabla\mathbf u$ with a boundary-distance weight whose constant depends
		only on these uniform graph characteristics.  The reciprocal weight is
		uniformly integrable to the exponent required by H\"older's inequality;
		hence, for some $p\in(q,2)$,
		\[
		\|\nabla\mathbf u\|_{L^q(\mathcal O_\zeta)}
		\le
		C
		\left(
		\|\mathbb D\mathbf u\|_{L^2(\mathcal O_\zeta)}
		+
		\|\mathbf u\|_{L^p(\mathcal O_\zeta)}
		\right),
		\]
		with $C$ independent of the particular graph.
		
		It remains to remove the lower-order term uniformly.  Otherwise there
		would exist $\zeta_n\in\mathfrak G(M,h_-,h_+)$ and
		$\mathbf u_n$ with zero bottom trace such that
		$\|\mathbf u_n\|_{W^{1,q}}=1$ while
		$\|\mathbb D\mathbf u_n\|_2\to0$.  After extraction,
		$\zeta_n\to\zeta$ uniformly and strongly in $W^{1,r}$ for every finite
		$r$. It is enough to treat $q>6/5$, since smaller exponents follow by
		H\"older on the uniformly bounded domains. Choose $q_0\in(6/5,q)$ and
		$r<\infty$ so that $1/q_0=1/q+1/r$. Pulling back by
		$(\bm y,s)\mapsto(\bm y,s(1+\zeta_n(\bm y)))$, the uniform
		$L^r$ bound for $\nabla_\Gamma\zeta_n$ and the normalized
		$W^{1,q}$ bound give a uniform $W^{1,q_0}$ bound on the fixed cylinder.
		Since $q_0^\#>2$, Rellich compactness gives strong convergence in the
		$L^p$ space appearing above. The limit has zero symmetric gradient and
		zero trace on the fixed bottom, hence is the zero rigid motion,
		contradicting the normalization. This proves
		\eqref{app:eq_subcritical_korn}; the fixed-cylinder Sobolev embedding
		gives \eqref{app:eq_subcritical_korn_sobolev}.
	\end{proof}

	\subsubsection{Tubular neighborhoods and normal traces}

	For $h>0$ define
	\begin{equation*}		T_\zeta^h
		:=
		\left\{
		(\bm y,z)\in\mathcal O_\alpha:
		0<z-1-\zeta(\bm y)<h
		\right\},
	\end{equation*}
	with parametrization
	\begin{equation*}		\Psi_{\zeta,h}(\bm y,r)
		=
		(\bm y,1+\zeta(\bm y)+r),
		\qquad
		(\bm y,r)\in\Gamma\times(0,h).
	\end{equation*}
	Its Jacobian equals one, hence
	\begin{equation*}		|T_\zeta^h|
		=
		h|\Gamma|.
	\end{equation*}
	
	For $a$ on $\Gamma$, define its vertical extension by
	\begin{equation*}		(\mathcal E_\zeta^h a)
		\bigl(\Psi_{\zeta,h}(\bm y,r)\bigr)
		=
		a(\bm y).
	\end{equation*}
	Then
	\begin{equation*}		\|\mathcal E_\zeta^h a\|_{L^q(T_\zeta^h)}
		=
		h^{1/q}
		\|a\|_{L^q(\Gamma)}.
	\end{equation*}
	
	For measurable sets $A,B\subset\mathcal O_\alpha$, we write
	$A\triangle B:=(A\setminus B)\cup(B\setminus A)$ for their symmetric
	difference.
	If $\zeta_n\to\zeta$ uniformly and $h_n\to h\ge0$, then
	\begin{align*}		& |\mathcal O_{\zeta_n}\triangle\mathcal O_\zeta|
		\le
		\|\zeta_n-\zeta\|_{L^1(\Gamma)}
		\to0, \\
		& |T_{\zeta_n}^{h_n}\triangle T_\zeta^h|
		\le
		C
		\left(
		\|\zeta_n-\zeta\|_{L^1(\Gamma)}
		+
		|h_n-h|
		\right)
		\to0. 
	\end{align*}
	
	\begin{lemma}[Vertical trace inequality]
		
		Let
		$F,\partial_zF\in L^2(T_\zeta^h)$.  Then
		\begin{equation}
			\label{app:eq_D_vertical_tube_trace}
			\|F|_{\Gamma_\zeta}\|_{L^2(\Gamma)}^2
			\le
			\frac{2}{h}
			\|F\|_{L^2(T_\zeta^h)}^2
			+
			2h
			\|\partial_zF\|_{L^2(T_\zeta^h)}^2.
		\end{equation}
		More generally, for every nonnegative measurable $b=b(\bm y)$,
		\begin{align}
			& \int_\Gamma
			b|F|_{\Gamma_\zeta}|^2\,\dd\bm y
			\le{}
			\frac{2}{h}
			\int_{T_\zeta^h}
			(\mathcal E_\zeta^h b)|F|^2\,\dd\bm x
			+
			2h
			\int_{T_\zeta^h}
			(\mathcal E_\zeta^h b)|\partial_zF|^2\,\dd\bm x.
			\label{app:eq_D_weighted_vertical_tube_trace} 
		\end{align}
	\end{lemma}
	
	\begin{proof}
		For almost every $\bm y$, apply the one-dimensional fundamental theorem
		of calculus on $(0,h)$, square, average in the vertical variable, and
		use Cauchy--Schwarz inequality.
	\end{proof}
	
	\begin{corollary}[Normal trace from a tubular penalty]
		\label{app:cor_normal_penalty_trace}
		Let
		\[
		F_\delta
		:=
		(\mathbf v_\delta-w_\delta\mathbf e_z)
		\cdot
		\mathbf n^{\zeta_\delta,\phi_\delta}
		\qquad
		\text{in }T_{\zeta_\delta}^{h_\delta},
		\]
		where
		$\mathbf n^{\zeta_\delta,\phi_\delta}$ and $w_\delta$
		are extended constantly along the vertical fibers. Assume
		\begin{align*}
			\frac1\delta
			\int_0^T
			\|F_\delta\|_{L^2(T_{\zeta_\delta}^{h_\delta})}^2
			\,\dd t
			&\le C,\;\;
			\kappa_\delta
			\int_0^T
			\|\nabla\mathbf v_\delta\|_
			{L^2(T_{\zeta_\delta}^{h_\delta})}^2
			\,\dd t
			\le C.
		\end{align*}
		Then
		\begin{equation}
			\int_0^T
			\left\|
			F_\delta|_{\Gamma_{\zeta_\delta}}
			\right\|_{L^2(\Gamma)}^2
			\,\dd t
			\le
			C
			\left(
			\frac{\delta}{h_\delta}
			+
			\frac{h_\delta}{\kappa_\delta}
			\right).
			\label{app:eq_D_penalty_trace}
		\end{equation}
	\end{corollary}
	
	\begin{proof}
		Since the vertical extensions satisfy
		\[
		\partial_z w_\delta=0,
		\qquad
		\partial_z
		\mathbf n^{\zeta_\delta,\phi_\delta}=0,
		\]
		we have
		\[
		\partial_zF_\delta
		=
		(\partial_z\mathbf v_\delta)
		\cdot
		\mathbf n^{\zeta_\delta,\phi_\delta},
		\qquad
		|\partial_zF_\delta|
		\le
		|\nabla\mathbf v_\delta|.
		\]
		By \eqref{app:eq_D_vertical_tube_trace},
		\begin{align*}
			\int_0^T
			\left\|
			F_\delta|_{\Gamma_{\zeta_\delta}}
			\right\|_{L^2(\Gamma)}^2
			\,\dd t
			&\le
			\frac{2}{h_\delta}
			\int_0^T
			\|F_\delta\|_{L^2(T_{\zeta_\delta}^{h_\delta})}^2
			\,\dd t+
			2h_\delta
			\int_0^T
			\|\partial_zF_\delta\|_
			{L^2(T_{\zeta_\delta}^{h_\delta})}^2
			\,\dd t
			\\
			&\le
			C
			\left(
			\frac{\delta}{h_\delta}
			+
			\frac{h_\delta}{\kappa_\delta}
			\right).
		\end{align*}
The proof is complete.
	\end{proof}
	
	The estimate is stated in the reference measure $\dd\bm y$.
	Surface-measure estimates follow from
	\eqref{app:eq_D_weighted_vertical_tube_trace}; this distinction is
	relevant because $J_\zeta$ need not be uniformly bounded for an
	$H^2$ graph.

	\subsubsection{Graph-normal multipliers}
	
	Define
	\begin{equation*}		d_\zeta(\bm y,z)
		:=
		z-1-\zeta(\bm y).
	\end{equation*}
	Since $\det\nabla\phi=1$,
	\begin{equation*}		\nabla^\phi d_\zeta
		=
		\operatorname{cof}(\nabla\phi)\mathbf N^\zeta,
	\end{equation*}
	and therefore
	\begin{equation*}		\nabla^\phi d_\zeta
		=
		J_\zeta
		J_\Gamma^\phi[\zeta]
		n^{\zeta,\phi}.
	\end{equation*}
	Because the flow is horizontal,
	\begin{equation}
		\label{app:eq_D_level_gradient_lower}
		|\nabla^\phi d_\zeta|
		\ge1.
	\end{equation}
	
	For a frozen flow $\phi_*$ set
	\begin{equation*}		\mathbf B_{\zeta,\phi_*}
		:=
		\frac{
			\nabla^{\phi_*}d_\zeta
		}{
			|\nabla^{\phi_*}d_\zeta|^2
		}.
	\end{equation*}
	
	\begin{lemma}[Graph-normal multiplier]
		\label{app:lem_graph_normal_multiplier}
		
		For $\zeta\in\mathfrak G(M,h_-,h_+)$,
		\begin{align*}			& \mathbf B_{\zeta,\phi_*}
			\cdot
			\nabla^{\phi_*}d_\zeta
			=1, \\
			& \|\mathbf B_{\zeta,\phi_*}\|_{L^\infty(\mathcal O_\alpha)}
			+
			\|\mathbf B_{\zeta,\phi_*}\|_{W^{1,2}(\mathcal O_\alpha)}
			\le
			C(M,R). 
		\end{align*}
		If
		$\partial_t\zeta=w\in L^2(0,T;H^1(\Gamma))$, then
		\begin{equation*}			\|\partial_t\mathbf B_{\zeta,\phi_*}\|_{L^2(\mathcal O_\alpha)}
			\le
			C(R)
			\|\nabla_\Gamma w\|_{L^2(\Gamma)}.
		\end{equation*}
	\end{lemma}
	
	\begin{proof}
		The first identity is immediate.  By
		\eqref{app:eq_D_level_gradient_lower}, the map
		$p\mapsto p/|p|^2$ and its first derivative are uniformly bounded on
		the relevant range.  Spatial differentiation then uses
		$\nabla_\Gamma^2\zeta\in L^2$, while time differentiation uses
		$\partial_t\nabla_\Gamma\zeta=\nabla_\Gamma w$.
	\end{proof}
	
	For every $\Psi\in W^{1,\infty}(\mathbb R)$,
	\begin{equation}
		\label{app:eq_D_boundary_div_identity}
		\Div^{\phi_*}
		\left(
		\Psi(d_\zeta)\mathbf B_{\zeta,\phi_*}
		\right)
		=
		\Psi'(d_\zeta)
		+
		\Psi(d_\zeta)
		\Div^{\phi_*}\mathbf B_{\zeta,\phi_*},
	\end{equation}
	with
	\begin{equation}
		\label{app:eq_D_boundary_remainder_L2}
		\|
		\Div^{\phi_*}\mathbf B_{\zeta,\phi_*}
		\|_{L^2(\mathcal O_\alpha)}
		\le
		C(M,R).
	\end{equation}
	This is the multiplier used in the boundary-pressure argument: the
	first term generates the positive layer contribution and the second is
	an $L^2$ remainder.
	
	Since
	$\Psi(d_\zeta)\mathbf B_{\zeta,\phi_*}\in W^{1,2}(\mathcal O_\alpha)$, it can be
	approximated strongly in $W^{1,2}$ by deterministic spatial
	mollification.  In particular, the approximating fields can be chosen
	progressively measurable and their transformed divergences converge
	strongly to the right-hand side of
	\eqref{app:eq_D_boundary_div_identity}.

	\subsubsection{Interior localization}
	
	\begin{lemma}[Fixed cylinders inside converging graph domains]
		\label{lem:fixed_cylinders_graph}
		Assume
		\[
		\zeta_n\to\zeta
		\qquad
		\text{uniformly on }[0,T]\times\Gamma.
		\]
		If
		$
		K
		\subset
		\{
		(t,\bm x):
		0<z<1+\zeta(t,\bm y)
		\},
		$
		then $K$ can be covered by finitely many product cylinders
		$I_j\times U_j$ with smooth fixed $U_j\subset\mathcal O_\alpha$ such that, for all
		sufficiently large $n$,
		\[
		U_j
		\subset
		\mathcal O_{\zeta_n(t)}
		\qquad
		\text{for every }t\in I_j.
		\]
		Hence standard local Sobolev, Korn, trace, and Rellich estimates can be
		applied on the $U_j$ with constants independent of $n$.
	\end{lemma}
	
	\begin{proof}
		A compact subset of the limiting fluid region has positive vertical
		distance from both the fixed bottom and the moving graph.  Uniform
		convergence of $\zeta_n$ preserves half of this distance for all large
		$n$.
	\end{proof}

		\subsection{Technical estimates for the final \texorpdfstring{$\delta$}{delta}-limit}
	\label{app:delta_limit}
	
	We use the exponent choice $(\beta,a_\beta)$ from
	\eqref{eq:global_beta_choice} and the final exterior-viscosity/tubular
	scaling $(\nu_\kappa,\kappa_\delta,h_\delta)$ from
	\eqref{eq:global_nu_choice}.  We set only
	\[
	D_\delta(t):=\mathcal O_{\zeta_\delta^*(t)},
	\qquad
	T_\delta(t):=T^\delta_{\zeta_\delta^*(t)}.
	\]
	For $\delta\to0$ we prove exterior-mass decay, recovery of the normal
	trace, pressure compactness, and convergence of the nonlinear terms.

	\subsubsection{Velocity estimates, exterior support, and normal trace}
	\label{app:subsec_delta_tightness}
	
	For $r>0$ define
	\[
	D_\delta^r(t)
	:=
	\left\{
	(\bm y,z):
	r<z<1+\zeta_\delta^*(t,\bm y)-r
	\right\}.
	\]
	For every fixed $r>0$, the corresponding limiting trimmed domain is
	eventually compactly contained in $D_\delta^{r/2}(t)$.
	
	\begin{proposition}[Subcritical and local velocity estimates]
		\label{app:prop_delta_time_compactness}
		For every $q\in(1,2)$ and $p\ge1$,
		\begin{equation}
			\label{app:eq_delta_subcritical_W1q}
			\sup_{\delta>0}
			\mathbb E
			\|\mathbf v_\delta\|_{L^2(0,T;W^{1,q}(D_\delta(t)))}^{2p}
			\le C_{p,q}.
		\end{equation}
		Moreover, for every compact space--time cylinder
		\[
		Q\subset\{(t,\bm x):\bm x\in\mathcal O_{\zeta^*(t)}\},
		\]
		there exists $\delta_Q>0$ such that
		\begin{equation}
			\label{app:eq_delta_local_H1}
			\sup_{0<\delta<\delta_Q}
			\mathbb E\|\mathbf v_\delta\|_{L^2_tH^1_x(Q)}^{2p}
			\le C_{p,Q}.
		\end{equation}
	\end{proposition}
	
	\begin{proof}
		Set
		\[
		\widetilde D_\delta(t):=\phi_\delta(t,D_\delta(t)),
		\qquad
		\mathbf u_\delta:=\mathbf v_\delta\circ\phi_\delta^{-1}.
		\]
		On bounded-flow events,
		\[
		\|\mathbf v_\delta(t)\|_{W^{1,q}(D_\delta(t))}
		\le C
		\|\mathbf u_\delta(t)\|_{W^{1,q}(\widetilde D_\delta(t))},
		\qquad q<2.
		\]
		Since the viscosities are physical on $\widetilde D_\delta(t)$ and
		$\mathbf u_\delta|_{\Gamma_b}=0$,
		Lemma~\ref{app:lem_subgraph_subcritical_korn} gives
		\[
		\|\mathbf u_\delta(t)\|_{W^{1,q}(\widetilde D_\delta(t))}
		\le
		C
		\|\mathbb D\mathbf u_\delta(t)\|_{L^2(\widetilde D_\delta(t))}.
		\]
		Therefore
		\[
		\mathbb E
		\|\mathbf v_\delta\|_{L^2(0,T;W^{1,q}(D_\delta(t)))}^{2p}
		\le
		C_{p,q}\,
		\mathbb E
		\left(
		\int_0^T\!\!\int_{\widetilde D_\delta(t)}
		|\mathbb D\mathbf u_\delta|^2
		\right)^p
		\le C_{p,q},
		\]
		which proves \eqref{app:eq_delta_subcritical_W1q}.
		
		Let
		\[
		Q\subset
		\{(t,\bm x):\bm x\in\mathcal O_{\zeta^*(t)}\}.
		\]
		By \eqref{eq:delta_zeta_strong} and
		Lemma~\ref{lem:fixed_cylinders_graph},
		\[
		Q\subset\bigcup_{j=1}^J(I_j\times U_j),
		\qquad
		U_j\subset D_\delta(t)
		\quad(t\in I_j)
		\]
		for $0<\delta<\delta_Q$. Writing
		$\widetilde U_{\delta,j}(t):=\phi_\delta(t,U_j)$,
		uniform local Korn inequalities yield
		\[
		\|\mathbf v_\delta\|_{L^2(I_j;H^1(U_j))}
		\le
		C_Q
		\left(
		\|\mathbb D\mathbf u_\delta\|_
		{L^2(I_j\times\widetilde U_{\delta,j})}
		+
		\|\mathbf v_\delta\|_{L^2(I_j;L^2(U_j))}
		\right).
		\]
		Choose $q_0\in(6/5,2)$. Then
		$W^{1,q_0}\hookrightarrow L^2$ and
		\eqref{app:eq_delta_subcritical_W1q} give
		\[
		\sup_{0<\delta<\delta_Q}
		\mathbb E
		\|\mathbf v_\delta\|_{L^2_tL^2_x(Q)}^{2p}
		\le C_{p,Q}.
		\]
		Hence
		\[
		\sup_{0<\delta<\delta_Q}
		\mathbb E
		\|\mathbf v_\delta\|_{L^2_tH^1_x(Q)}^{2p}
		\le C_{p,Q},
		\]
		which is \eqref{app:eq_delta_local_H1}. Along a diagonal subsequence,
		\[
		\mathbf v_\delta\rightharpoonup\mathbf v
		\qquad
		\text{in }L^2_tH^1_x(Q).
		\]
	The proof is complete.
	\end{proof}
	
	Along the same subsequence,
	\[
	\varrho_\delta\mathbf v_\delta
	\rightharpoonup
	\varrho\mathbf v
	\qquad
	\text{in }\mathcal D'.
	\]
	
	Set
	\[
	g_\delta(t,\bm y,z)
	:=
	z-1-\zeta_\delta^*(t,\bm y),
	\]
	and choose a scalar cutoff $\Xi\in C^\infty(\mathbb R;[0,1])$ such that
	\[
	\Xi=0 \quad\text{on }(-\infty,0],
	\qquad
	\Xi=1 \quad\text{on }[1,\infty).
	\]
	Define
	$
	\Xi_\delta
	:=
	\Xi\left(\frac{g_\delta}{h_\delta}\right).
	$
	
	\begin{proposition}[Exterior mass estimate]
		\label{app:prop_exterior_mass}
		If $\varrho_{0,\delta}=0$ outside $\mathcal O_{\zeta_0}$, then
		\begin{equation}
			\label{app:eq_exterior_mass}
			\mathbb E
			\sup_{0\le t\le T}
			\int_{\mathcal O_\alpha}
			\varrho_\delta(t)\Xi_\delta(t)\,\dd\bm x
			\le
			C\delta^{a_\beta^2}.
		\end{equation}
	\end{proposition}
	
	\begin{proof}
		Since the stochastic flow is horizontal,
		\[
		\partial_t g_\delta
		+
		w_\delta\mathbf e_z\cdot\nabla^{\phi_\delta}g_\delta
		=0.
		\]
		Testing the renormalized continuity equation with $\Xi_\delta$ gives
		\[
		\frac{\dd}{\dd t}
		\int_{\mathcal O_\alpha}
		\varrho_\delta\Xi_\delta\,\dd\bm x
		=
		\int_{\mathcal O_\alpha}
		\varrho_\delta
		(\mathbf v_\delta-w_\delta\mathbf e_z)
		\cdot\nabla^{\phi_\delta}\Xi_\delta\,\dd\bm x.
		\]
		Moreover,
		\[
		\supp\nabla^{\phi_\delta}\Xi_\delta\subset T_\delta,
		\qquad
		\|\varrho_\delta\|_{L^\infty_tL^\beta_x}
		\le C\delta^{-1/\beta},
		\]
		\[
		\|
		(\mathbf v_\delta-w_\delta\mathbf e_z)
		\cdot\mathbf n^{\zeta_\delta^*,\phi_\delta}
		\|_{L^2(T_\delta)}
		\le C\delta^{1/2},
		\qquad
		|T_\delta|\le Ch_\delta.
		\]
		Hence
		\[
		\mathbb E
		\left|
		\int_0^T\!\!\int_{\mathcal O_\alpha}
		\varrho_\delta
		(\mathbf v_\delta-w_\delta\mathbf e_z)
		\cdot\nabla^{\phi_\delta}\Xi_\delta
		\,\dd\bm x\,\dd t
		\right|
		\le C\delta^{a_\beta^2}.
		\]
		Since
		\[
		\int_{\mathcal O_\alpha}
		\varrho_{0,\delta}\Xi_\delta(0)\,\dd\bm x=0,
		\]
		\eqref{app:eq_exterior_mass} follows.
	\end{proof}
	
	In particular, the limiting density is supported in the physical
	fluid region.
	
	\begin{proposition}[Normal trace from the tubular penalty]
		\label{app:prop_normal_trace}
		Assume \eqref{eq:global_nu_choice}. Then
		\begin{equation*}			\mathbb E
			\int_0^T\!\!\int_\Gamma
			\left|
			\left[
			(\mathbf v_\delta-w_\delta\mathbf e_z)
			\cdot
			\mathbf n^{\zeta_\delta^*,\phi_\delta}
			\right]
			\circ\Phi_{\zeta_\delta^*}
			\right|^2
			\,\dd\bm y\,\dd t
			\le
			C
			\left(
			\frac{\delta}{h_\delta}
			+
			\frac{h_\delta}{\kappa_\delta}
			\right).
		\end{equation*}
		Consequently,
		\[
		\left(
		\mathbf v
		-
		w\mathbf e_z\circ\Phi_{\zeta^*}^{-1}
		\right)
		\cdot
		\mathbf n^{\zeta^*,\phi}
		=
		0
		\qquad
		\text{on }\Gamma_{\zeta^*}
		\]
		in the trace sense.
	\end{proposition}
	
	\begin{proof}
		Set
		\[
		F_\delta
		:=
		(\mathbf v_\delta-w_\delta\mathbf e_z)
		\cdot
		\mathbf n^{\zeta_\delta^*,\phi_\delta}.
		\]
		The tubular trace estimate gives
		\[
		\|\Tr_{\Gamma_{\zeta_\delta^*}}F_\delta\|_{L^2(\Gamma)}^2
		\le
		C
		\left(
		\frac1{h_\delta}\|F_\delta\|_{L^2(T_\delta)}^2
		+
		h_\delta
		\|\nabla\mathbf v_\delta\|_{L^2(T_\delta)}^2
		\right).
		\]
		Using the normal penalty and exterior viscous coercivity,
		\[
		\mathbb E
		\int_0^T
		\|\Tr_{\Gamma_{\zeta_\delta^*}}F_\delta\|_{L^2(\Gamma)}^2\,\dd t
		\le
		C
		\left(
		\frac{\delta}{h_\delta}
		+
		\frac{h_\delta}{\kappa_\delta}
		\right)
		\longrightarrow0
		\]
		by \eqref{eq:global_nu_choice}. The trace convergence and the strong
		convergence of the geometry give the asserted limiting normal condition.
	\end{proof}

	\subsubsection{Pressure tightness and artificial-pressure decay}

	\begin{proposition}[Interior pressure gain]
		\label{app:prop_delta_interior_pressure}
		There exists $\Theta>0$ such that, for every $r>0$, there is
		$\delta_r>0$ with
		\begin{equation}
			\label{app:eq_delta_interior_pressure}
			\sup_{0<\delta<\delta_r}
			\mathbb E
			\int_0^T\!\!\int_{D_\delta^r(t)}
			\left(
			\varrho_\delta^{\gamma+\Theta}
			+
			\delta\varrho_\delta^{\beta+\Theta}
			\right)
			\,\dd\bm x\,\dd t
			\le C_r.
		\end{equation}
	\end{proposition}
	
	\begin{proof}
		Fix $r>0$ and choose a cutoff supported in
		$D_\delta^{r/2}(t)$ and equal to one on $D_\delta^r(t)$.
		The localized frozen Bogovski\u{\i} construction used in
		Proposition~\ref{app:prop_higher_integrability} is admissible on this
		support. There
		\[
		\mathbb S_{\kappa_\delta}^{\zeta_\delta^*}
		=
		\mathbb S,
		\qquad
		\text{penalty}=0,
		\qquad
		\text{Navier term}=0.
		\]
		The termwise estimates of
		Proposition~\ref{app:prop_higher_integrability} therefore give
		\[
		\sup_{0<\delta<\delta_r}
		\mathbb E
		\int_0^T\!\!\int_{D_\delta^r(t)}
		\left(
		\varrho_\delta^{\gamma+\Theta}
		+
		\delta\varrho_\delta^{\beta+\Theta}
		\right)
		\,\dd\bm x\,\dd t
		\le C_r,
		\]
		which is \eqref{app:eq_delta_interior_pressure}.
	\end{proof}
	
	Let
	\[
	d_\delta
	:=
	z-1-\zeta_\delta^*(t,\bm y),
	\qquad
	L_{\delta,r}^-:=\{-r<d_\delta<0\},
	\qquad
	L_{\delta,r}^+:=\{0<d_\delta<r\}.
	\]
	
	Choose $q_K<2$ sufficiently close to $2$ and set
	\[
	q_*:=\frac{3q_K}{3-q_K},
	\qquad
	\frac1\gamma+\frac2{q_*}<1,
	\qquad
	\frac1{p_*}
	:=
	1-\frac1\gamma-\frac2{q_*}>0.
	\]
	Then
	\begin{equation}
		\label{app:eq_delta_subgraph_Sobolev}
		\sup_{\delta>0}
		\mathbb E
		\|\mathbf v_\delta\|_{L^2(0,T;L^{q_*}(D_\delta(t)))}^2
		\le C,
	\end{equation}
	and we set
	\begin{equation*}		\sigma_\gamma
		:=
		\min
		\left\{
		\frac12,\,
		\frac1{p_*},\,
		\frac{\gamma-1}{2\gamma}
		\right\}
		>0.
	\end{equation*}
	
	\begin{proposition}[Boundary pressure tightness]
		\label{app:prop_delta_boundary_pressure}
		There exists $\sigma_\gamma>0$ such that
		\begin{align}
			& \sup_{\delta>0}
			\mathbb E
			\int_0^T\!\!\int_{L_{\delta,r}^-}
			p_\delta(\varrho_\delta)
			\,\dd\bm x\,\dd t
			\le
			Cr^{\sigma_\gamma},
			\label{app:eq_delta_boundary_pressure_inside} \\
			& \limsup_{\delta\to0}
			\mathbb E
			\int_0^T\!\!\int_{L_{\delta,r}^+}
			p_\delta(\varrho_\delta)
			\,\dd\bm x\,\dd t
			\le
			Cr^{\sigma_\gamma}.
			\label{app:eq_delta_boundary_pressure_outside} 
		\end{align}
		The same estimate holds near the fixed bottom. Consequently, for every
		$\varepsilon>0$ there exists $r_\varepsilon>0$ such that
		\begin{equation}
			\label{app:eq_delta_boundary_pressure}
			\limsup_{\delta\to0}
			\mathbb E
			\int_0^T\!\!\int_{
				\{|d_\delta|<r_\varepsilon\}
				\cup
				\{0<z<r_\varepsilon\}
			}
			p_\delta(\varrho_\delta)
			\,\dd\bm x\,\dd t
			<
			\varepsilon .
		\end{equation}
	\end{proposition}
	
	\begin{proof}
		For the interior layer choose
		\[
		\mathbf q_{\delta,r}^-
		=
		\vartheta(t)\Psi_r^-(d_\delta)
		\mathbf B_{\zeta_\delta^*,\phi_i},
		\]
		with
		\[
		\supp\Psi_r^-\subset(-2r,0),
		\qquad
		0\le\Psi_r^-\le1,
		\qquad
		(\Psi_r^-)'=r^{-1}
		\quad\text{on }(-r,0).
		\]
		By \eqref{app:eq_D_boundary_div_identity},
		\[
		\Div^{\phi_i}\mathbf q_{\delta,r}^-
		=
		r^{-1}\vartheta
		+
		\mathcal R_{\delta,r}^-,
		\qquad
		\|\mathcal R_{\delta,r}^-\|_{L^2}
		\le C
		\quad\text{on }L_{\delta,r}^-,
		\]
		up to the frozen-coefficient error, which vanishes with the localization
		scale. After multiplication of the weak momentum identity by $r$, the
		time contribution $\mathcal T_{\delta,r}$ satisfies
		\[
		|\mathcal T_{\delta,r}|
		\le
		Cr^{(\gamma-1)/(2\gamma)},
		\]
		\[
		r\left|
		\int_{L_{\delta,2r}^-}
		\varrho_\delta
		\mathbf v_\delta\otimes\mathbf v_\delta:
		\nabla\mathbf q_{\delta,r}^-
		\right|
		\le
		Cr^{1/p_*},
		\]
		and
		\[
		r\left|
		\int_{L_{\delta,2r}^-}
		\mathbb S(\nabla^{\phi_\delta}\mathbf v_\delta):
		\nabla^{\phi_\delta}\mathbf q_{\delta,r}^-
		\right|
		\le
		Cr^{1/2}.
		\]
		The remainder $\mathcal R_{\delta,r}^-$ is controlled by the same
		estimates. Hence
		\[
		\sup_{\delta>0}
		\mathbb E
		\int_0^T\!\!\int_{L_{\delta,r}^-}
		p_\delta(\varrho_\delta)
		\,\dd\bm x\,\dd t
		\le
		C
		\left(
		r^{(\gamma-1)/(2\gamma)}
		+
		r^{1/p_*}
		+
		r^{1/2}
		\right)
		\le
		Cr^{\sigma_\gamma},
		\]
		which proves
		\eqref{app:eq_delta_boundary_pressure_inside}.
		
		For the exterior layer,
		Proposition~\ref{app:prop_exterior_mass} and H\"older's inequality give
		\[
		\mathbb E\sup_{t\le T}
		\int_{\mathcal O_\alpha\setminus D_\delta(t)}
		\varrho_\delta\,\dd\bm x
		\le
		C
		\left(
		\delta^{a_\beta^2}
		+
		h_\delta^{1-1/\gamma}
		\right)
		\le
		C\delta^{m_{\rm ext}}.
		\]
		Interpolation with the $L^\gamma$ bound yields
		\[
		\mathbb E
		\|
		\varrho_\delta
		\mathbf1_{\mathcal O_\alpha\setminus D_\delta}
		\|_{L^\infty(0,T;L^{p_{\rm ext}})}
		\le
		C\delta^{m_{\rm ext}\vartheta_{\rm ext}}.
		\]
		For
		\[
		q_{\rm ext}
		:=
		\frac{2p_{\rm ext}}{p_{\rm ext}-1}
		<6,
		\qquad
		\|\mathbf v_\delta\|_{L^2_tL^{q_{\rm ext}}_x}^2
		\le
		C\kappa_\delta^{-1},
		\]
		we obtain
		\[
		\mathbb E
		\int_0^T\!\!\int_{\mathcal O_\alpha\setminus D_\delta}
		\varrho_\delta|\mathbf v_\delta|^2
		\,\dd\bm x\,\dd t
		\le
		C
		\delta^{m_{\rm ext}\vartheta_{\rm ext}-\nu_\kappa}
		\longrightarrow0.
		\]
		Together with the finite-$L^r$ bounds of
		$\nabla_\Gamma\zeta_\delta^*$,
		\begin{equation}
		\label{app:eq_delta_exterior_conv_small}
		\mathbb E
		\int_0^T\!\!\int_{\mathcal O_\alpha\setminus D_\delta}
		\varrho_\delta|\mathbf v_\delta|^2
		\left(
		1+
		|\mathcal E_{\zeta_\delta^*}\nabla_\Gamma\zeta_\delta^*|
		\right)
		\,\dd\bm x\,\dd t
		\longrightarrow0.
		\end{equation}
		which is \eqref{app:eq_delta_exterior_conv_small}.
		
		Using the exterior multiplier $\Psi_r^+(d_\delta)$, the normal-penalty
		contribution satisfies
		\[
		|\mathcal P_{\delta,r}|
		\le
		C\delta^{-1/2}h_\delta^{3/2}
		\longrightarrow0
		\qquad(a_\beta>1/3).
		\]
		The remaining terms obey the preceding $r$-bounds, hence
		\[
		\limsup_{\delta\to0}
		\mathbb E
		\int_0^T\!\!\int_{L_{\delta,r}^+}
		p_\delta(\varrho_\delta)
		\,\dd\bm x\,\dd t
		\le
		Cr^{\sigma_\gamma}.
		\]
		The fixed-bottom cutoff gives the same estimate near $z=0$.
		Letting $r\to0$ proves
		\eqref{app:eq_delta_boundary_pressure}.
	\end{proof}
	
	\begin{proposition}[Pressure equi-integrability and artificial-pressure decay]
		\label{app:prop_artificial_pressure_decay}
		For some $\Theta>0$ and every fixed $r>0$, there exists
		$\delta_r>0$ such that
		\begin{equation}
			\label{app:eq_delta_exterior_higher_gain}
			\sup_{0<\delta<\delta_r}
			\mathbb E
			\int_0^T\!\!\int_{\{d_\delta>r\}}
			\left(
			\varrho_\delta^{\gamma+\Theta}
			+
			\delta\varrho_\delta^{\beta+\Theta}
			\right)
			\,\dd\bm x\,\dd t
			\le C_r.
		\end{equation}
		Consequently,
		$\{p_\delta(\varrho_\delta)\}_{\delta>0}$ is uniformly integrable and
		\begin{equation}
			\label{app:eq_delta_artificial_decay}
			\delta\varrho_\delta^\beta
			\to0
			\qquad
			\text{strongly in }
			L^1(\Omega\times(0,T)\times\mathcal O_\alpha).
		\end{equation}
	\end{proposition}
	
	\begin{proof}
		Choose $\Theta>0$ so that the exterior H\"older exponents remain within
		the margin
		\[
		\nu_\kappa
		<
		m_{\rm ext}\vartheta_{\rm ext}.
		\]
		A frozen Bogovski\u{\i} test supported in
		$\{d_\delta>r/2\}$ gives
		\[
		\sup_{0<\delta<\delta_r}
		\mathbb E
		\int_0^T\!\!\int_{\{d_\delta>r\}}
		\left(
		\varrho_\delta^{\gamma+\Theta}
		+
		\delta\varrho_\delta^{\beta+\Theta}
		\right)
		\,\dd\bm x\,\dd t
		\le C_r.
		\]
		Together with
		\eqref{app:eq_delta_interior_pressure} and
		\eqref{app:eq_delta_boundary_pressure}, this yields uniform
		integrability of
		$\{p_\delta(\varrho_\delta)\}_{\delta>0}$.
		
		On every region $Q$ for which
		\[
		\sup_\delta
		\int_Q
		\delta\varrho_\delta^{\beta+\Theta}
		<\infty,
		\]
		H\"older's inequality gives
		\[
		\int_Q
		\delta\varrho_\delta^\beta
		\le
		\delta^{\Theta/(\beta+\Theta)}
		|Q|^{\Theta/(\beta+\Theta)}
		\left(
		\int_Q
		\delta\varrho_\delta^{\beta+\Theta}
		\right)^{\beta/(\beta+\Theta)}
		\longrightarrow0.
		\]
		Thus the artificial pressure vanishes on the trimmed interior and on
		the exterior region at positive distance from the interface. On the
		remaining boundary layers,
		\[
		0
		\le
		\delta\varrho_\delta^\beta
		\le
		p_\delta(\varrho_\delta),
		\]
		and \eqref{app:eq_delta_boundary_pressure} applies. Therefore
		\[
		\delta\varrho_\delta^\beta
		\longrightarrow0
		\qquad
		\text{in }
		L^1(\Omega\times(0,T)\times\mathcal O_\alpha),
		\]
		which is \eqref{app:eq_delta_artificial_decay}.
	\end{proof}

	\subsubsection{Effective viscous flux and strong density convergence}

	\begin{proposition}[Effective viscous flux in the interior]
		\label{app:prop_delta_effective_flux}
		Let
		$
		K\subset
		\{(t,\bm x):\bm x\in\mathcal O_{\zeta^*(t)}\}.
		$
		For every $\psi\in C_c^\infty(K)$ and $k\ge1$,
		\begin{align*}			& \lim_{\delta\to0}
			\int
			\psi
			\left[
			a\varrho_\delta^\gamma
			+
			\delta\varrho_\delta^\beta
			-
			(\lambda+2\mu)
			\Div^{\phi_\delta}\mathbf v_\delta
			\right]
			T_k(\varrho_\delta)
			\notag \\
			& \qquad=
			\int
			\psi
			\left[
			\overline p
			-
			(\lambda+2\mu)
			\Div^\phi\mathbf v
			\right]
			\overline{T_k(\varrho)}. 
		\end{align*}
	\end{proposition}
	
	\begin{proof}
		For sufficiently small $\delta$,
		\[
		\supp\psi
		\subset
		\{(t,\bm x):\bm x\in D_\delta(t)\},
		\]
		and on this support
		$
		\mathbb S_{\kappa_\delta}^{\zeta_\delta^*}
		=
		\mathbb S.
		$
		Apply the frozen local-potential construction of
		Subsection~\ref{app:effective_flux}. The Lions commutator treats the
		inertial terms, while the frozen-coefficient remainders vanish with the
		localization scale. Passing $\delta\to0$ gives the stated effective-flux
		identity.
	\end{proof}
	
	\begin{proposition}[Global strong convergence of the density]
		\label{app:prop_delta_global_density}
		Almost surely,
		\begin{equation}
			\label{app:eq_delta_rho_strong}
			\varrho_\delta
			\to
			\varrho
			\qquad
			\text{strongly in }
			L^1((0,T)\times\mathcal O_\alpha).
		\end{equation}
		Moreover,
		\[
		p_\delta(\varrho_\delta)
		\to
		a\varrho^\gamma
		\qquad\text{strongly in }
		L^1(\Omega\times(0,T)\times\mathcal O_\alpha).
		\]
	\end{proposition}
	
	\begin{proof}
		Fix $r>0$. On the trimmed physical region,
		Proposition~\ref{app:prop_delta_effective_flux}, the monotonicity of
		$z\mapsto az^\gamma$, and the renormalized continuity equation give
		\[
		\int_0^T\!\!\int_{D_\delta^r(t)}
		|\varrho_\delta-\varrho|
		\,\dd\bm x\,\dd t
		\longrightarrow0.
		\]
		For the moving and fixed boundary layers,
		\[
		\int_0^T\!\!\int_{\{|d_\delta|<r\}\cup\{0<z<r\}}
		\varrho_\delta
		\,\dd\bm x\,\dd t
		\le
		C
		\|\varrho_\delta\|_{L^\infty_tL^\gamma_x}
		r^{1-1/\gamma},
		\]
		while Proposition~\ref{app:prop_exterior_mass} controls the pure exterior.
		Thus, first $\delta\to0$ and then $r\to0$ give
		\[
		\varrho_\delta
		\to
		\varrho
		\qquad
		\text{strongly in }
		L^1((0,T)\times\mathcal O_\alpha)
		\quad\text{a.s.}
		\]
		
		Moreover,
		\[
		\delta\varrho_\delta^\beta\to0
		\quad\text{in }L^1
		\]
		by \eqref{app:eq_delta_artificial_decay}, and
		$\{p_\delta(\varrho_\delta)\}_{\delta>0}$ is uniformly integrable by
		Proposition~\ref{app:prop_artificial_pressure_decay}. Since
		\[
		p_\delta(\varrho_\delta)
		\to
		a\varrho^\gamma
		\qquad\text{in measure},
		\]
		Vitali's theorem yields
		\[
		p_\delta(\varrho_\delta)
		\to
		a\varrho^\gamma
		\qquad
		\text{strongly in }
		L^1(\Omega\times(0,T)\times\mathcal O_\alpha).
		\]
The proof is complete.
	\end{proof}

	\subsubsection{Adapted tests and nonlinear limits}

	\begin{proposition}[Adapted normal-compatible approximation]
		\label{app:prop_delta_test_construction}
		Let $(\mathbf q,\psi)\in\mathcal T_{\rm ad}(\zeta,\phi)$ and
		$T'<\tau^\zeta$. After bounded-flow localization there exist
		progressively measurable It\^o-semimartingale pairs
		$(\mathbf q_\delta,\psi_\delta)$, adapted to the $\delta$-filtration,
		\[
		\dd\mathbf q_\delta
		=
		\mathbf q_{\delta,0}\,\dd t
		+
		\sum_{k=1}^K\mathbf q_{\delta,k}\,\dd W_{\delta,k},
		\qquad
		\dd\psi_\delta
		=
		\psi_{\delta,0}\,\dd t
		+
		\sum_{k=1}^K\psi_{\delta,k}\,\dd W_{\delta,k},
		\]
		which are admissible for the $\delta$-problem and satisfy
		\begin{equation}
			\label{app:eq_delta_test_exact_normal}
			(\mathbf q_\delta-\psi_\delta\mathbf e_z)
			\cdot\mathbf n^{\zeta_\delta^*,\phi_\delta}
			=0
			\qquad\text{throughout }T_\delta .
		\end{equation}
		Moreover,
		\[
		(\mathbf q_\delta,\psi_\delta)
		\longrightarrow
		(\mathbf q,\psi)
		\]
		in every non-pressure test duality occurring in
		Definition~\ref{def:trans_solution}, including the drift, trace, and
		quadratic-variation pairings. Along the same deterministic diagonal, the
		complete $\delta$-level coupled residual converges in probability in
		$C([0,T'])$ to the fixed causal residual
		$\mathfrak C(\varrho,\mathbf v,\zeta,w;\mathbf q,\psi)$.
		Consequently, the $\delta$-level identities imply
		$
		\mathfrak C(\varrho,\mathbf v,\zeta,w;\mathbf q,\psi)=0.
		$
	\end{proposition}
	
	\begin{proof}
		Let $(\widetilde{\mathbf q},\psi)$ be the smooth precursor associated with
		$(\mathbf q,\psi)$. Approximation of its initial value and It\^o
		characteristics by bounded simple predictable cylinder functionals gives
		\begin{align*}
			\widetilde{\mathbf q}_\delta(t)
			&=
			\widetilde{\mathbf q}_\delta(0)
			+
			\int_0^t
			\widetilde{\mathbf q}_{\delta,0}\,\dd s
			+
			\sum_{k=1}^K
			\int_0^t
			\widetilde{\mathbf q}_{\delta,k}\,\dd W_{\delta,k},
			\\
			\widetilde\psi_\delta(t)
			&=
			\widetilde\psi_\delta(0)
			+
			\int_0^t
			\widetilde\psi_{\delta,0}\,\dd s
			+
			\sum_{k=1}^K
			\int_0^t
			\widetilde\psi_{\delta,k}\,\dd W_{\delta,k},
		\end{align*}
		and, along a deterministic diagonal,
		$
		(\widetilde{\mathbf q}_\delta,\widetilde\psi_\delta)
		\longrightarrow
		(\widetilde{\mathbf q},\psi)
		$
		in the corresponding drift and martingale topologies.
		
		Let $a_\delta$ be the normal mismatch and set
		\[
		\omega_\delta
		:=
		\|a_\delta\|_{L^2(\Omega\times(0,T')\times\Gamma)},
		\qquad
		r_\delta
		:=
		h_\delta^{1/2}+\omega_\delta.
		\]
		Then
		\[
		\sup_\delta
		\mathbb E
		\|a_\delta\|_{L^2(0,T';H^1(\Gamma))}^2
		<\infty,
		\qquad
		\omega_\delta\to0,
		\]
		and, after a deterministic subsequence,
		\[
		r_\delta\to0,
		\qquad
		\frac{h_\delta}{r_\delta}\to0,
		\qquad
		\frac{\omega_\delta^2}{r_\delta}\to0.
		\]
		Define
		\[
		\mathbf q_\delta
		:=
		\widetilde{\mathbf q}_\delta
		-
		R_{\delta,r_\delta},
		\qquad
		\psi_\delta
		:=
		\widetilde\psi_\delta .
		\]
		By \eqref{app:eq_projection_correction},
		\[
		\mathbb E
		\|R_{\delta,r_\delta}\|_{L^2(0,T';H^1)}^2
		\le
		C
		\left(
		r_\delta
		+
		\frac{\omega_\delta^2}{r_\delta}
		\right)
		\longrightarrow0,
		\]
		while
		\[
		(\mathbf q_\delta-\psi_\delta\mathbf e_z)
		\cdot
		\mathbf n^{\zeta_\delta^*,\phi_\delta}
		=0
		\qquad\text{in }T_\delta.
		\]
		It\^o's formula for the regularized projection and BDG give the required
		drift and quadratic-variation convergences.
		
		For the complete residual, the exact normal constraint gives
		\[
		\frac1\delta
		\int_{T_\delta}
		\bigl(
		(\mathbf v_\delta-w_\delta\mathbf e_z)
		\cdot\mathbf n^{\zeta_\delta^*,\phi_\delta}
		\bigr)
		\bigl(
		(\mathbf q_\delta-\psi_\delta\mathbf e_z)
		\cdot\mathbf n^{\zeta_\delta^*,\phi_\delta}
		\bigr)
		\,\dd\bm x
		=0.
		\]
		For each fixed smooth representative,
		Proposition~\ref{app:prop_delta_global_density} identifies the pressure,
		and the preceding convergences yield
		\[
		\mathfrak C_\delta
		(\varrho_\delta,\mathbf v_\delta,\zeta_\delta,w_\delta;
		\mathbf q_\delta,\psi_\delta)
		\longrightarrow
		\mathfrak C
		(\varrho,\mathbf v,\zeta,w;\mathbf q,\psi)
		\]
		in probability in $C([0,T'])$. Since the left-hand side vanishes,
		\[
		\mathfrak C
		(\varrho,\mathbf v,\zeta,w;\mathbf q,\psi)
		=0.
		\]
	\end{proof}
	
	\begin{proposition}[Trace identification for the Navier friction]
		\label{app:prop_delta_trace_compactness}
		If $\iota>0$, then
		\[
		\left(
		\mathbf v_\delta
		-
		w_\delta\mathbf e_z\circ\Phi_{\zeta_\delta^*}^{-1}
		\right)_{\tau^{\zeta_\delta^*,\phi_\delta}}
		\rightharpoonup
		\left(
		\mathbf v
		-
		w\mathbf e_z\circ\Phi_{\zeta^*}^{-1}
		\right)_{\tau^{\zeta^*,\phi}}
		\]
		weakly in the weighted interface $L^2$ space.
		Consequently, the Navier-friction term passes to the limit against the
		test traces of Proposition~\ref{app:prop_delta_test_construction}.
	\end{proposition}
	
	\begin{proof}
		After pullback to $\Gamma$, the geometric weights and tangential
		projections converge strongly, while the friction bound gives weak
		$L^2$ compactness. For every fixed offset $r>0$,
		\eqref{app:eq_delta_local_H1} gives
		\[
		\Tr_{\delta,r}\mathbf v_\delta
		\rightharpoonup
		\Tr_r\mathbf v
		\]
		in the corresponding trace space. Moreover,
		\[
		\|
		\Tr_{\zeta_\delta^*}\mathbf v_\delta
		-
		\Tr_{\delta,r}\mathbf v_\delta
		\|_{L^2_tL^q_\Gamma}
		\le
		Cr^{1-1/q},
		\qquad q<2.
		\]
		Hence
		\[
		\lim_{r\to0}
		\limsup_{\delta\to0}
		\|
		\Tr_{\zeta_\delta^*}\mathbf v_\delta
		-
		\Tr_{\delta,r}\mathbf v_\delta
		\|_{L^2_tL^q_\Gamma}
		=0.
		\]
		Combining this with
		\[
		w_\delta\rightharpoonup w
		\qquad\text{in }L^2_tH^1_\Gamma
		\]
		and the strong convergence of the geometry identifies the weak limit of
		the relative tangential velocity and proves the claim.
	\end{proof}
	
	\begin{proposition}[Convective limit]
		\label{app:prop_delta_convection}
		For every smooth compactly supported test tensor field in the physical fluid
		region,
		\[
		\varrho_\delta
		\mathbf v_\delta\otimes\mathbf v_\delta
		\rightharpoonup
		\varrho\mathbf v\otimes\mathbf v
		\qquad
		\text{in }\mathcal D'.
		\]
		The same convergence holds against the adapted test fields of
		Proposition~\ref{app:prop_delta_test_construction}.
	\end{proposition}
	
	\begin{proof}
		Let $Q$ be compactly contained in the limiting physical fluid region.
		By \eqref{app:eq_delta_rho_strong} and
		\eqref{app:eq_delta_local_H1},
		\[
		\varrho_\delta
		\mathbf v_\delta\otimes\mathbf v_\delta
		\rightharpoonup
		\varrho\mathbf v\otimes\mathbf v
		\qquad
		\text{in }\mathcal D'(Q).
		\]
		For the moving and tubular boundary layers, the subcritical
		Korn--Sobolev estimate and their vanishing measure give
		\[
		\lim_{r\to0}
		\limsup_{\delta\to0}
		\left|
		\int_{\{|d_\delta|<r\}}
		\varrho_\delta
		\mathbf v_\delta\otimes\mathbf v_\delta:
		\boldsymbol\Psi
		\right|
		=0
		\]
		for every admissible bounded test tensor $\boldsymbol\Psi$.
		The pure exterior contribution vanishes by
		\eqref{app:eq_delta_exterior_conv_small}. This gives the global
		distributional convergence; the adapted-test convergence follows from
		Proposition~\ref{app:prop_delta_test_construction}.
	\end{proof}
	
	\begin{proposition}[Limit of the viscous stress]
		\label{app:prop_delta_stress_limit}
		For every limiting admissible test field and its approximation from
		Proposition~\ref{app:prop_delta_test_construction},
		\begin{align}
			& \int_0^T\!\!\int_{\mathcal O_\alpha}
			\mathbb S_{\kappa_\delta}^{\zeta_\delta^*}
			(\nabla^{\phi_\delta}\mathbf v_\delta):
			\nabla^{\phi_\delta}\mathbf q_\delta
			\,\dd\bm x\,\dd t
			\to
			\int_0^T\!\!\int_{\mathcal O_{\zeta^*}}
			\mathbb S(\nabla^\phi\mathbf v):
			\nabla^\phi\mathbf q
			\,\dd\bm x\,\dd t .
			\label{app:eq_delta_stress_limit} 
		\end{align}
	\end{proposition}
	
	\begin{proof}
		Fix an interior offset $r>0$. On the corresponding trimmed physical
		region,
		$
		\mathbb S_{\kappa_\delta}^{\zeta_\delta^*}
		=
		\mathbb S,
		$
		and \eqref{app:eq_delta_local_H1} gives
		\[
		\int
		\mathbb S_{\kappa_\delta}^{\zeta_\delta^*}
		(\nabla^{\phi_\delta}\mathbf v_\delta):
		\nabla^{\phi_\delta}\mathbf q_\delta
		\longrightarrow
		\int
		\mathbb S(\nabla^\phi\mathbf v):
		\nabla^\phi\mathbf q
		\]
		on the trimmed region.
		
		Let $\mathcal R_{\delta,r}$ denote the contribution of the physical
		boundary layer, the viscosity-transition layer, and the pure exterior.
		By Cauchy--Schwarz inequality and the viscous dissipation,
		\[
		\lim_{r\to0}
		\limsup_{\delta\to0}
		|\mathcal R_{\delta,r}|
		=0,
		\]
		using the vanishing layer measure and the exterior-viscosity scaling
		\eqref{eq:global_nu_choice}. Hence
		\eqref{app:eq_delta_stress_limit} follows.
	\end{proof}

\subsection{Completion of the transformed existence proof}

\begin{proof}[Proof of Theorem~\ref{thm:transformed_existence}]
	Set
	\[
	\tau^\zeta
	:=
	T\wedge
	\inf\left\{
	t>0:
	\inf_\Gamma(1+\zeta(t))\le\alpha
	\ \text{or}\ 
	\|\zeta(t)\|_{H^s(\Gamma)}\ge\alpha^{-1}
	\right\}.
	\]
	By \eqref{eq:initial_localization_margin} and
	$\zeta\in C([0,T];H^s(\Gamma))$,
	$
	\mathbb P(\tau^\zeta>0)=1.
	$
	
	The compactness conclusions above give items~\textup{(1)}--\textup{(3)} of
	Definition~\ref{def:trans_solution}. Moreover,
	\[
	\varrho(t)=0
	\qquad
	\text{a.e. on }
	\mathcal O_\alpha\setminus\mathcal O_{\zeta(t)}
	\]
	for a.e. $t<\tau^\zeta$. Proposition~\ref{app:prop_normal_trace} gives item~\textup{(6)}:
	\[
	\left(
	\mathbf v-w\mathbf e_z\circ\Phi_\zeta^{-1}
	\right)
	\cdot\mathbf n^{\zeta,\phi}
	=0
	\qquad\text{on }\Gamma_\zeta .
	\]
	
	Proposition~\ref{app:prop_delta_global_density} and the renormalized
	continuity equation yield item~\textup{(4)}. For every
	$(\mathbf q,\psi)\in\mathcal T_{\rm ad}(\zeta,\phi)$,
	Proposition~\ref{app:prop_delta_test_construction} gives
	\begin{equation}
		\mathfrak C_t
		(\varrho,\mathbf v,\zeta,w;\mathbf q,\psi)
		=0,
		\qquad t<\tau^\zeta .
		\label{eq:final_coupled_residual}
	\end{equation}

	In addition, lower semicontinuity and Fatou's lemma give, for every $p\ge1$,
	\begin{equation}
		\mathbb E
		\left[
		\sup_{0\le t\le\tau^\zeta}\mathcal E_{\rm tr}(t)^p
		+\mathcal D_{\rm tr}(\tau^\zeta)^p
		\right]
		\le C_{p,T}.
		\label{eq:delta_final_energy_moment}
	\end{equation}
	
	Lemma~\ref{app:lem_total_energy_recovery} gives item~\textup{(7)} of
	Definition~\ref{def:trans_solution}. Hence all conditions of
	Definition~\ref{def:trans_solution} are satisfied.
\end{proof}

\begin{proof}[Proof of Theorem~\ref{thm:main}]
	Theorem~\ref{thm:transformed_existence} and
	Lemma~\ref{lem:equivalence} yield a martingale solution of the physical
	system in the sense of Definition~\ref{def:physical_solution}.
\end{proof}

	\medskip
	
			\noindent {\bf Acknowledgements.}
	This work is partially supported by the Baima Lake Laboratory Joint Fund of Zhejiang Provincial NSF of China under Grant No. LBMHZ26E060006, China NSF Grant Nos. 12571164, 12671292, the Jiangsu Provincial Scientific Research Center of Applied Mathematics under Grant No. BK20233002 and the Fundamental Research Funds of Zhejiang Sci-Tech University No. 26062140-Y.

		\medskip

	\appendix
	
	\section{Geometry-adapted tests and stochastic energy recovery}
	\label{app:shell_ito}
	
	This appendix provides the fixed causal test core and the two
	regularized identities used in the final stochastic energy recovery.
	
	\subsection{Geometry-adapted semimartingale test core}
	\label{app:semimartingale_tests}
	
	Fix a stopped geometry $(\zeta,\phi)$. Start from spatially smooth,
	progressively measurable It\^o-semimartingale precursor pairs
	$(\widetilde{\mathbf q},\psi)$ satisfying the fixed-bottom condition.
	Near $\Gamma_\zeta$ we project only the fluid component onto the scalar
	normal constraint, leaving $\psi$ unchanged. The resulting pair
	$(\mathbf q,\psi)$ satisfies
	\[
	\mathbf q|_{\Gamma_b}=0,
	\qquad
	(\mathbf q-\psi\mathbf e_z\circ\Phi_\zeta^{-1})
	\cdot\mathbf n^{\zeta,\phi}=0
	\qquad\text{on }\Gamma_\zeta,
	\]
	and is called an element of $\mathcal T_{\rm ad}(\zeta,\phi)$. Since
	$\Gamma_\zeta$ is only an $H^2$ graph, no $C^\infty$ regularity of the
	projected fluid component across the interface is asserted; its required
	spatial regularity is $H^1$. The shell component remains spatially
	smooth, so the elastic term in Definition~\ref{def:trans_solution} is an
	ordinary $L^2$ pairing. In particular, the definition of
	$\mathcal T_{\rm ad}$ depends only on $(\zeta,\phi)$ and not on
	$(\varrho,\mathbf v,w)$.
	
	Fix once and for all the spatial mollifier and backward Steklov
	regularization used in Lemma~\ref{app:lem_energy_core_density}. If
	$(\mathbf q^{(j)},\psi^{(j)})$ denotes the corresponding causal smooth
	representatives of a core pair, let
	\[
	\mathfrak C_t^{(j)}
	(\varrho,\mathbf v,\zeta,w;\mathbf q,\psi)
	\]
	be the left-hand side minus the right-hand side of the \emph{smooth}
	coupled identity displayed in item~(5) of
	Definition~\ref{def:trans_solution}, with
	$(\mathbf q,\psi)$ and their It\^o coefficients replaced by
	$(\mathbf q^{(j)},\psi^{(j)})$ and with the pressure contribution taken as
	the ordinary integral
	\[
	\int_0^t\!\!\int_{\mathcal O_\zeta(s)}
	p(\varrho)\Div^\phi\mathbf q^{(j)}
	\,\dd\mathbf x\,\dd s.
	\]
	For this fixed causal regularization, we say that the coupled core residual
	is closed if, for every $T'<\tau^\zeta$,
	\[
	\mathfrak C^{(j)}
	\longrightarrow
	\mathfrak C
	\qquad\text{in probability in }C([0,T']).
	\]
	The limit is denoted by
	$\mathfrak C_t(\varrho,\mathbf v,\zeta,w;\mathbf q,\psi)$.
	Definition~\ref{def:trans_solution} requires this \emph{complete}
	residual to vanish. No separate extension of
	$p(\varrho)\Div^\phi\mathbf q$ is made. For a spatially smooth compatible
	core pair the representatives may be chosen constant in $j$, so that
	$\mathfrak C$ is exactly the ordinary smooth weak residual.

	For $(\mathbf q,\psi)\in\mathcal T_{\rm ad}(\zeta,\phi)$ write
	\[
	\dd\mathbf q=\mathbf q_0\,\dd t+\sum_{k=1}^K\mathbf q_k\,\dd W_k,
	\qquad
	\dd\psi=\psi_0\,\dd t+\sum_{k=1}^K\psi_k\,\dd W_k.
	\]
	The coupled weak formulation is understood in the It\^o sense: the
	time-derivative pairings are replaced by the drift pairings with
	$(\mathbf q_0,\psi_0)$ and the additional stochastic term is
	\begin{equation}
		\label{eq:semimart_test_martingale}
		\sum_{k=1}^K\int_0^t
		\left[
		\int_{\mathcal O_\zeta(s)}
		\varrho\mathbf v\cdot\mathbf q_k\,\dd\mathbf x
		+
		\int_\Gamma w\psi_k\,\dd\mathbf y
		\right]\dd W_k(s).
	\end{equation}

\begin{lemma}[Semimartingale extension of the continuity identity]
	\label{app:lem_semimart_continuity}
	Let $T'<\tau^\zeta$ and let $\varphi$ be a spatially smooth,
	compactly supported adapted scalar It\^o semimartingale,
	\[
	\dd\varphi
	=
	\varphi_0\,\dd t
	+
	\sum_{k=1}^K \varphi_k\,\dd W_k,
	\]
	with bounded coefficients after localization. Then, for every
	$b\in C^1(\mathbb R)$ with $b'$ compactly supported,
	\begin{align*}
		& \int_{\mathcal O_\zeta(t)}
		b(\varrho(t))\varphi(t)\,\dd\mathbf x
		-
		\int_{\mathcal O_{\zeta_0}}
		b(\varrho_0)\varphi(0)\,\dd\mathbf x
		\\
		& =
		\int_0^t\!\!\int_{\mathcal O_\zeta(s)}
		b(\varrho)
		\bigl(
		\varphi_0+\mathbf v\cdot\nabla^\phi\varphi
		\bigr)\,\dd\mathbf x\,\dd s
		+
		\int_0^t\!\!\int_{\mathcal O_\zeta(s)}
		\bigl(
		b'(\varrho)\varrho-b(\varrho)
		\bigr)
		\Div^\phi\mathbf v\,\varphi
		\,\dd\mathbf x\,\dd s
		\\
		& \quad
		+
		\sum_{k=1}^K
		\int_0^t
		\left[
		\int_{\mathcal O_\zeta(s)}
		b(\varrho)\varphi_k\,\dd\mathbf x
		\right]\dd W_k(s)
	\end{align*}
	for every $t\le T'$, almost surely.
		The same identity also holds for the linear choice $b(z)=z$.
\end{lemma}

\begin{proof}
	Let
	\[
	\pi_m=\{0=t_0<\cdots<t_{N_m}=T'\},
	\qquad |\pi_m|\to0,
	\]
	and let $\varphi^{(m)}$ be the left-adapted simple approximation of
	$\varphi$ on $\pi_m$. Apply the deterministic renormalized continuity
	identity on each $(t_i,t_{i+1}]$ and sum over $i$. The discrete increments of
	$\varphi^{(m)}$ converge to
	\[
	\sum_{k=1}^K
	\int_0^t
	\left[
	\int_{\mathcal O_\zeta(s)}
	b(\varrho)\varphi_k\,\dd\mathbf x
	\right]\dd W_k(s).
	\]
	After localization,
	\[
	\varphi^{(m)}\to\varphi,
	\qquad
	\varphi_0^{(m)}\to\varphi_0,
	\qquad
	\varphi_k^{(m)}\to\varphi_k
	\]
	in the deterministic and quadratic-variation dualities of the continuity
	identity. The drift terms pass by dominated convergence, and the stochastic
	terms pass by It\^o's isometry and BDG. This yields the stated identity.
	The case $b(z)=z$ follows by the standard truncation argument.
\end{proof}

\begin{remark}
	The lemma is used only for $b(z)=z$ or for renormalizations with compactly
	supported derivative; no final-level product
	$p(\varrho)\Div^\phi\mathbf v$ is invoked.
\end{remark}

	\subsection{Bending-energy commutator identity}

Set
$
A_k:=\mathcal Y_k\cdot\nabla_\Gamma,
C_k:=[\Delta_\Gamma,A_k],
D_k:=[C_k,A_k].
$
By Lemma~\ref{lem:shell_energy_commutator},
\[
\|C_k\eta\|_{L^2}+\|D_k\eta\|_{L^2}
\le C_k\|\eta\|_{H^2}.
\]

\begin{lemma}[Regularized bending-energy identity]
	\label{app:lem_bending_commutator}
	Let $(\eta,r)$ be progressively measurable and, for every $T'<\tau^\eta$,
	\[
	\eta\in L^2\bigl(\Omega;L^\infty(0,T';H^2(\Gamma))\bigr),
	\qquad
	r\in L^2\bigl(\Omega;L^2(0,T';H^1(\Gamma))\bigr).
	\]
	Assume, in the distributional It\^o sense,
	\begin{equation}
		\dd\eta
		=
		\left(r+\frac12\sum_{k=1}^K A_k^2\eta\right)\dd t
		+
		\sum_{k=1}^K A_k\eta\,\dd W_k .
		\label{app:eq_eta_ito_bending}
	\end{equation}
	Let $J_n$ be a symmetric spatial mollifier, $K_n:=J_n^2$, and set
	\[
	z:=\Delta_\Gamma\eta,
	\qquad
	z_n:=J_nz,
	\qquad
	g_k:=C_k\eta,
	\qquad
	h_k:=D_k\eta .
	\]
	Then
	\begin{align}
		\frac12\|z_n(t)\|_2^2
		={}&
		\frac12\|z_n(0)\|_2^2
		+
		\int_0^t
		\left\langle
		\Delta_\Gamma\eta,
		\Delta_\Gamma K_nr
		\right\rangle\,\dd s
		\notag\\
		&+
		\sum_{k=1}^K\int_0^t
		\widetilde M_{n,k}\,\dd W_k
		+
		\frac12\sum_{k=1}^K\int_0^t
		\widetilde R_{n,k}\,\dd s,
		\label{app:eq_bending_regularized_energy}
	\end{align}
	where
	\begin{align}
		\widetilde M_{n,k}
		&:=
		\left\langle
		z_n,J_n(A_kz+g_k)
		\right\rangle,
		\label{app:eq_Mnk_def}
		\\
		\widetilde R_{n,k}
		&:=
		\|J_n(A_kz+g_k)\|_2^2
		+
		\left\langle
		z_n,J_n(A_k^2z+2A_kg_k+h_k)
		\right\rangle.
		\label{app:eq_Rnk_def}
	\end{align}
	Moreover,
	\begin{equation}
		\widetilde M_{n,k}
		\longrightarrow
		\mathcal G_k(\eta)
		\quad\text{in }L^2(\Omega\times(0,T')),
		\qquad
		\widetilde R_{n,k}
		\longrightarrow
		\mathcal R_k(\eta)
		\quad\text{in }L^1(\Omega\times(0,T')).
		\label{app:eq_bending_Rnk_limit}
	\end{equation}
\end{lemma}

\begin{proof}
	From
	\[
	\Delta_\Gamma A_k\eta=A_kz+g_k,
	\qquad
	\Delta_\Gamma A_k^2\eta=A_k^2z+2A_kg_k+h_k,
	\]
	we obtain from \eqref{app:eq_eta_ito_bending}
	\begin{align*}
		\dd z_n
		={}&
		J_n\Delta_\Gamma r\,\dd t
		+
		\frac12\sum_{k=1}^K
		J_n(A_k^2z+2A_kg_k+h_k)\,\dd t
		\\
		&+
		\sum_{k=1}^KJ_n(A_kz+g_k)\,\dd W_k .
	\end{align*}
	It\^o's formula gives \eqref{app:eq_bending_regularized_energy}.
	Since $A_k^*=-A_k$,
	\begin{align*}
		\langle J_nf,J_nA_kf\rangle
		&=\frac12\langle f,[K_n,A_k]f\rangle,
		\\
		\langle J_nf,J_nA_k^2f\rangle
		+\|J_nA_kf\|_2^2
		&=\frac12\langle f,[[K_n,A_k],A_k]f\rangle .
	\end{align*}
	The Friedrichs commutator limits
	\[
	[K_n,A_k]f\to0,
	\qquad
	[[K_n,A_k],A_k]f\to0
	\qquad\text{in }L^2(\Gamma)
	\]
	therefore imply
	\[
	\widetilde M_{n,k}\to\langle z,g_k\rangle=\mathcal G_k(\eta),
	\qquad
	\widetilde R_{n,k}\to\|g_k\|_2^2+\langle z,h_k\rangle
	=\mathcal R_k(\eta),
	\]
	in the spaces stated in \eqref{app:eq_bending_Rnk_limit}. The bounds follow
	from \eqref{eq:commutator_energy_bound}.
\end{proof}

\subsection{Normal-compatible kinetic--internal-energy approximation}
	
	Lemma~\ref{app:lem_bending_commutator} provides the bending-energy
	identity at a fixed spatial scale. We now construct normal-compatible
	semimartingale tests for the kinetic--internal-energy part at the same
	scale; the two identities are combined in
	Subsection~\ref{app:total_energy_recovery}.
	
\begin{lemma}[Projected-core density and causal representatives]
	\label{app:lem_energy_core_density}
	Fix $T'<\tau^\zeta$ and a bounded-flow localization. For every smooth
	adapted semimartingale precursor $(\widetilde{\mathbf q},\psi)$ there is a
	projected pair
	\[
	(\mathbf q,\psi)\in\mathcal T_{\rm ad}(\zeta,\phi),
	\qquad
	(\mathbf q-\psi\mathbf e_z\circ\Phi_\zeta^{-1})
	\cdot\mathbf n^{\zeta,\phi}=0
	\quad\text{on }\Gamma_\zeta,
	\]
	and fixed causal smooth representatives
	$(\mathbf q^{(j)},\psi^{(j)})$ satisfying
	\begin{equation}
		\label{app:eq_core_density}
		\mathbf q^{(j)}\to\mathbf q
		\quad\text{in }L^2(0,T';H^1),
		\qquad
		\psi^{(j)}\to\psi
		\quad\text{in }L^2(0,T';H^2(\Gamma)).
	\end{equation}
	If
	\[
	\dd\mathbf q^{(j)}
	=\mathbf q_0^{(j)}\,\dd t
	+\sum_{k=1}^K\mathbf q_k^{(j)}\,\dd W_k,
	\qquad
	\dd\psi^{(j)}
	=\psi_0^{(j)}\,\dd t
	+\sum_{k=1}^K\psi_k^{(j)}\,\dd W_k,
	\]
	then the corresponding coefficients converge to those of
	$(\mathbf q,\psi)$ in the drift and quadratic-variation dualities of
	Definition~\ref{def:trans_solution}. Consequently all non-pressure
	pairings in the coupled identity converge along the same causal family.
\end{lemma}

\begin{proof}
	Use deterministic spatial mollification and backward Steklov averaging.
	Let $a_j$ be the scalar normal mismatch before projection and let
	$R_{r_j}$ be the vertical-fiber correction in a layer of width $r_j$.
	Then
	\begin{equation}
		\label{app:eq_projection_correction}
		\|R_{r_j}\|_{H^1}
		\le
		C\left(
		r_j^{1/2}\|a_j\|_{H^1(\Gamma)}
		+r_j^{-1/2}\|a_j\|_{L^2(\Gamma)}
		\right).
	\end{equation}
	After localization,
	\[
	\sup_j\|a_j\|_{H^1(\Gamma)}<\infty,
	\qquad
	\|a_j\|_{L^2(\Gamma)}\to0.
	\]
	Choose a deterministic $r_j\to0$ such that
	\[
	r_j\|a_j\|_{H^1(\Gamma)}^2\to0,
	\qquad
	\frac{\|a_j\|_{L^2(\Gamma)}^2}{r_j}\to0.
	\]
	Hence
	$
	\|R_{r_j}\|_{H^1}\to0,
	$
	which gives \eqref{app:eq_core_density}. The projection depends smoothly
	on the regularized graph and flow; It\^o's formula therefore gives the
	convergence of the drift and martingale coefficients. The remaining
	non-pressure pairings follow from \eqref{app:eq_core_density}, the trace
	estimates of Subsection~\ref{app:geometry}, and BDG inequality. The representatives
	so obtained are precisely the fixed causal family used in
	Appendix~\ref{app:semimartingale_tests}.
\end{proof}

\begin{lemma}[Causal normal-compatible cross-energy approximation]
	\label{app:lem_fluid_energy_compatible_regularization}
	Fix $T'<\tau^\eta$ and set
	\[
	P(\rho):=\frac{a}{\gamma-1}\rho^\gamma,
	\qquad
	p_\gamma:=\frac{2\gamma}{\gamma-1}<6,
	\qquad
	K_n:=J_n^2,
	\qquad
	r_n:=K_nr .
	\]
	Then there exist positive numbers $\sigma_n\to0$ and progressively measurable
	pairs
	$
	(\mathbf U_n,\psi_n),
	$
	each of which is an It\^o semimartingale, with
	\[
	\psi_n
	:=
	\mathcal S^-_{\sigma_n}r_n,
	\qquad
	(\mathcal S^-_\sigma f)(t)
	:=
	\frac1\sigma
	\int_{(t-\sigma)_+}^t
	f(s)\,\dd s,
	\]
	such that
	\begin{equation}
		(\mathbf U_n-\psi_n\mathbf e_z)\cdot\mathbf N^\eta
		=0
		\qquad\text{on }\Gamma_\eta .
		\label{app:eq_energy_core_normal}
	\end{equation}
	
	Moreover,
	\begin{align}
	&	\mathbf U_n-\mathbf u
		\to0\;\;
		\text{in }
		L^2\bigl(
		\Omega\times(0,T');
		L^{p_\gamma}(\mathcal O_\eta)
		\bigr),
		\label{app:eq_energy_core_Lpgamma}
		\\
	&	\mathbb E
		\int_0^{T'}\!\!\int_{\mathcal O_\eta}
		\rho|\mathbf U_n-\mathbf u|^2
		\,\dd\bm x\,\dd t
		\to0,
		\label{app:eq_energy_core_weighted_kinetic}
		\\
		&\mathbb E
		\left|
		\int_0^{T'}\!\!\int_{\mathcal O_\eta}
		\mathbb S(\nabla\mathbf u):
		\nabla(\mathbf U_n-\mathbf u)
		\,\dd\bm x\,\dd t
		\right|
		\to0,
		\label{app:eq_energy_core_viscous_duality}
		\\
		&\mathbb E
		\left|
		\int_0^{T'}\!\!\int_{\Gamma_\eta}
		(\mathbf u-r\mathbf e_z)_{\tau^\eta}
		\cdot
		\Bigl[
		(\mathbf U_n-\psi_n\mathbf e_z)_{\tau^\eta}
		-
		(\mathbf u-r\mathbf e_z)_{\tau^\eta}
		\Bigr]
		\,\dd S\,\dd t
		\right|
		\to0.
		\label{app:eq_energy_core_navier_duality}
	\end{align}
	The corresponding drift and martingale characteristics converge in the
	dualities of Lemma~\ref{app:lem_energy_core_density}.
	Define
	\begin{align*}
		\mathcal K_{{\rm cpl},n}(t)
		&:=
		\int_{\mathcal O_\eta(t)}
		\left(
		\rho\mathbf u\cdot\mathbf U_n
		-\frac12\rho|\mathbf U_n|^2
		+P(\rho)
		\right)\,\dd\bm x
		+\frac12\|J_nr(t)\|_2^2,
		\\
		\mathcal V_n(t)
		&:=
		\int_0^t\!\!\int_{\mathcal O_\eta(s)}
		\mathbb S(\nabla\mathbf u):
		\nabla\mathbf U_n
		\,\dd\bm x\,\dd s,
		\\
		\mathcal N_n(t)
		&:=
		\iota
		\int_0^t\!\!\int_{\Gamma_\eta(s)}
		(\mathbf u-r\mathbf e_z)_{\tau^\eta}
		\cdot
		(\mathbf U_n-\psi_n\mathbf e_z)_{\tau^\eta}
		\,\dd S\,\dd s .
	\end{align*}
	Then
	\begin{align}
		&\mathcal K_{{\rm cpl},n}(t)
		+\mathcal V_n(t)
		+\mathcal N_n(t)
		+\nu_s
		\int_0^t
		\|\nabla_\Gamma J_nr\|_2^2\,\dd s
		\notag\\
		&\qquad\le
		\mathcal K_{{\rm cpl},n}(0)
		-
		\int_0^t
		\left\langle
		\Delta_\Gamma\eta,
		\Delta_\Gamma K_nr
		\right\rangle\,\dd s
		+
		\varepsilon_n(t),
		\label{app:eq_fluid_cross_balance}
	\end{align}
	where
$		\mathbb E
		\sup_{t\le T'}
		|\varepsilon_n(t)|
		\longrightarrow0 .$	
	Finally,
	\begin{align}
		\mathcal K_{{\rm cpl},n}(t)
		&\longrightarrow
		\int_{\mathcal O_\eta(t)}
		\left(
		\frac12\rho|\mathbf u|^2+P(\rho)
		\right)\,\dd\bm x
		+\frac12\|r(t)\|_2^2,
		\label{app:eq_fluid_cross_energy_limit}
		\\
		\mathcal V_n(t)
		&\longrightarrow
		\int_0^t\!\!\int_{\mathcal O_\eta(s)}
		\mathbb S(\nabla\mathbf u):
		\nabla\mathbf u
		\,\dd\bm x\,\dd s,
		\label{app:eq_fluid_cross_viscous_limit}
		\\
		\mathcal N_n(t)
		&\longrightarrow
		\iota
		\int_0^t\!\!\int_{\Gamma_\eta(s)}
		|(\mathbf u-r\mathbf e_z)_{\tau^\eta}|^2
		\,\dd S\,\dd s .
		\label{app:eq_fluid_cross_navier_limit}
	\end{align}
\end{lemma}

\begin{proof}
	Choose
	$
	\psi_n=\mathcal S^-_{\sigma_n}K_nr
	$
	and apply Lemma~\ref{app:lem_energy_core_density}. By
	Lemma~\ref{app:lem_subgraph_subcritical_korn} and
	$p_\gamma<6$,
	\[
	\mathbf U_n-\mathbf u
	\to0
	\quad\text{in }
	L^2\bigl(
	\Omega\times(0,T');
	L^{p_\gamma}(\mathcal O_\eta)
	\bigr).
	\]
	Hence, by H\"older's inequality,
	\begin{align}
		\mathbb E
		\int_0^{T'}\!\!\int_{\mathcal O_\eta}
		\rho|\mathbf U_n-\mathbf u|^2
		\,\dd\bm x\,\dd t\le
		\mathbb E
		\left[
		\|\rho\|_{L^\infty_tL^\gamma_x}
		\|\mathbf U_n-\mathbf u\|_{L^2_tL^{p_\gamma}_x}^2
		\right]
		\longrightarrow0 .
		\label{app:eq_weighted_kinetic_estimate}
	\end{align}
	The viscous and Navier limits follow from the corresponding
	$H^1$ and trace dualities of
	Lemma~\ref{app:lem_energy_core_density}, which proves
	\eqref{app:eq_energy_core_Lpgamma}--%
	\eqref{app:eq_energy_core_navier_duality}.
	
	For fixed $n$, let
	$(\mathbf U_n^{(j)},\psi_n^{(j)})$ be the causal smooth
	representatives furnished by
	Lemma~\ref{app:lem_energy_core_density}, and set
	$
	\varphi_n^{(j)}
	:=
	\frac12|\mathbf U_n^{(j)}|^2 .
	$
	If
	\[
	\dd\mathbf U_n^{(j)}
	=
	\mathbf U_{n,0}^{(j)}\,\dd t
	+
	\sum_{k=1}^K
	\mathbf U_{n,k}^{(j)}\,\dd W_k,
	\]
	then
	\begin{align}
		\dd\varphi_n^{(j)}
		={}
		\left(
		\mathbf U_n^{(j)}
		\cdot\mathbf U_{n,0}^{(j)}
		+
		\frac12
		\sum_{k=1}^K
		|\mathbf U_{n,k}^{(j)}|^2
		\right)\dd t+
		\sum_{k=1}^K
		\mathbf U_n^{(j)}
		\cdot\mathbf U_{n,k}^{(j)}
		\,\dd W_k .
		\label{app:eq_scalar_kinetic_ito}
	\end{align}
	Lemma~\ref{app:lem_semimart_continuity}, applied with
	$b(z)=z$ and $\varphi=\varphi_n^{(j)}$, gives the corresponding
	kinetic correction.
		At every regularized approximation level, combine this identity with
	the pressure-potential balance and the smooth coupled
	momentum--structure identity before passing to the limit. The resulting
	identity has the form
	\begin{align}
		&\mathcal K_{{\rm cpl},n}^{(j)}(t)
		+\mathcal V_n^{(j)}(t)
		+\mathcal N_n^{(j)}(t)
		+\nu_s
		\int_0^t
		\|\nabla_\Gamma J_nr\|_2^2\,\dd s
		\notag\\
		&\qquad\le
		\mathcal K_{{\rm cpl},n}^{(j)}(0)
		-
		\int_0^t
		\left\langle
		\Delta_\Gamma\eta,
		\Delta_\Gamma K_nr
		\right\rangle\,\dd s
		+
		\varepsilon_{n,j}(t).
		\label{app:eq_cross_balance_smooth}
	\end{align}
	For fixed $n,j$, the pressure pairings in
	\eqref{app:eq_cross_balance_smooth} are classical. Applying the limit chain
	\[
	N\to\infty,
	\qquad
	m\to\infty,
	\qquad
	\epsilon\to0,
	\qquad
	l\to0,
	\qquad
	\delta\to0,
	\]
	with Proposition~\ref{app:prop_delta_global_density} at the last step gives,
	for fixed $n$,
	\[
	\mathcal K_{{\rm cpl},n}^{(j)}
	+\mathcal V_n^{(j)}
	+\mathcal N_n^{(j)}
	\longrightarrow
	\mathcal K_{{\rm cpl},n}
	+\mathcal V_n
	+\mathcal N_n
	\qquad (j\to\infty).
	\]
	Hence \eqref{app:eq_cross_balance_smooth} converges to
	\eqref{app:eq_fluid_cross_balance}. The drift and martingale dualities of
	Lemma~\ref{app:lem_energy_core_density}, together with BDG, yield
	\[
	\mathbb E
	\sup_{t\le T'}
	|\varepsilon_n(t)|
	\longrightarrow0
	\qquad (n\to\infty).
	\]

	Finally,
	\begin{equation}
		\rho\mathbf u\cdot\mathbf U_n
		-\frac12\rho|\mathbf U_n|^2
		=
		\frac12\rho|\mathbf u|^2
		-\frac12\rho|\mathbf U_n-\mathbf u|^2.
		\label{app:eq_cross_kinetic_identity}
	\end{equation}
	Combining
	\eqref{app:eq_cross_kinetic_identity} with
	\eqref{app:eq_energy_core_weighted_kinetic},
	\eqref{app:eq_energy_core_viscous_duality}, and
	\eqref{app:eq_energy_core_navier_duality} yields
	\eqref{app:eq_fluid_cross_energy_limit}--%
	\eqref{app:eq_fluid_cross_navier_limit}.
\end{proof}

	\subsection{Total-energy recovery}\label{app:total_energy_recovery}

\begin{lemma}[Coupled total-energy recovery]
	\label{app:lem_total_energy_recovery}
	Let $(\varrho,\mathbf v,\zeta,w)$ be the final transformed solution and set
	$
	\widehat\eta:=\zeta\circ\chi^{-1},
	\widehat r:=w\circ\chi^{-1}.
	$
	Fix $T'<\tau^\zeta$. Then, for every $t\le T'$,
	\begin{align}
		\mathcal E_{\rm tr}(t)+\mathcal D_{\rm tr}(t)
		\le{}&
		\mathcal E_{\rm tr}(0)
		+
		\sum_{k=1}^K
		\int_0^t
		\mathcal G_k(\widehat\eta(s))\,\dd W_k(s)+
		\frac12
		\sum_{k=1}^K
		\int_0^t
		\mathcal R_k(\widehat\eta(s))\,\dd s .
		\label{app:eq_total_energy_recovered}
	\end{align}
	Here $\mathcal E_{\rm tr},\mathcal D_{\rm tr}$ are defined by
	\eqref{eq:trans_energy_def}--\eqref{eq:trans_dissipation_def}, and
	$\mathcal G_k,\mathcal R_k$ by
	\eqref{eq:Gk_commutator}--\eqref{eq:Rk_commutator}.
\end{lemma}

\begin{proof}
	Apply
	\eqref{app:eq_bending_regularized_energy} and
	\eqref{app:eq_fluid_cross_balance} with
	$
	\eta=\widehat\eta,
	r=\widehat r.
	$
	The terms
$
	\pm
	\int_0^t
	\left\langle
	\Delta_\Gamma\widehat\eta,
	\Delta_\Gamma K_n\widehat r
	\right\rangle\,\dd s
	$
	cancel, and therefore
	\begin{align}
		&\mathcal K_{{\rm cpl},n}(t)
		+\frac12
		\|J_n\Delta_\Gamma\widehat\eta(t)\|_2^2
		+\mathcal V_n(t)
		+\mathcal N_n(t)
		+\nu_s
		\int_0^t
		\|\nabla_\Gamma J_n\widehat r\|_2^2\,\dd s
		\notag\\
		&\qquad\le
		\mathcal K_{{\rm cpl},n}(0)
		+\frac12
		\|J_n\Delta_\Gamma\widehat\eta(0)\|_2^2
		+\sum_{k=1}^K
		\int_0^t
		\widetilde M_{n,k}\,\dd W_k
		\notag\\
		&\qquad\quad
		+\frac12
		\sum_{k=1}^K
		\int_0^t
		\widetilde R_{n,k}\,\dd s
		+\varepsilon_n(t).
		\label{app:eq_total_energy_regularized}
	\end{align}
	
	By
	\eqref{app:eq_fluid_cross_energy_limit}--%
	\eqref{app:eq_fluid_cross_navier_limit},
	\[
	J_n\Delta_\Gamma\widehat\eta
	\to
	\Delta_\Gamma\widehat\eta
	\quad\text{in }
	L^2(0,T';L^2(\Gamma)),
	\]
	\[
	J_n\widehat r
	\to
	\widehat r
	\quad\text{in }
	L^2(0,T';H^1(\Gamma)),
	\qquad
	\mathbb E
	\sup_{t\le T'}
	|\varepsilon_n(t)|
	\to0.
	\]
	Furthermore,
	\eqref{app:eq_bending_Rnk_limit} and the BDG inequality give
	\[
	\sum_{k=1}^K
	\int_0^\cdot
	\widetilde M_{n,k}\,\dd W_k
	\longrightarrow
	\sum_{k=1}^K
	\int_0^\cdot
	\mathcal G_k(\widehat\eta(s))\,\dd W_k(s)
	\quad
	\text{in }L^1(\Omega;C([0,T'])),
	\]
	and
	\[
	\sum_{k=1}^K
	\int_0^t
	\widetilde R_{n,k}\,\dd s
	\longrightarrow
	\sum_{k=1}^K
	\int_0^t
	\mathcal R_k(\widehat\eta(s))\,\dd s
	\quad
	\text{in }L^1(\Omega).
	\]
	Passing to the limit in
	\eqref{app:eq_total_energy_regularized} and using the
	volume- and area-preserving change of variables yields
	\eqref{app:eq_total_energy_recovered}.
\end{proof}

\bibliographystyle{amsplain}

\end{document}